\documentclass[11pt, DIV10,a4paper]{article}
\usepackage{color}
\usepackage{float}
\usepackage[bf,small]{caption2}
\usepackage{comment}
\usepackage{amsmath}
\usepackage{bigints}
\usepackage{rotating, booktabs}
\usepackage{amsopn}
\usepackage{amsfonts}
\usepackage{amsthm}
\usepackage{amssymb}
\usepackage{amsbsy}
\usepackage{multirow}
\usepackage{slashbox}
\usepackage{natbib}
\usepackage{tocbibind}
\usepackage[a4paper,colorlinks,breaklinks,bookmarksopen,bookmarksnumbered]{hyperref}
\usepackage[refpage]{nomencl}
\usepackage{makeidx}
\usepackage{pst-all}
\usepackage{epsfig,psfrag}
\usepackage{graphicx}
\usepackage{amsmath}
\usepackage{epstopdf}
\usepackage{tabularx}
\usepackage{subfigure}
\usepackage{setspace}
\usepackage[mathscr]{euscript}
\usepackage[margin=1in]{geometry}
\usepackage{xr-hyper}
\usepackage{amsfonts} % Mathesymbole
\usepackage{calc}
\usepackage{enumitem}

\newcommand{\ueq}[1][]{%
  \if\relax\detokenize{#1}\relax
    \sbox0{$\underbrace{=}_{}$}%
    \mathrel{\mathmakebox[\wd0]{=}}
  \else
    \mathrel{\underbrace{=}_{\mathclap{#1}}}
  \fi}

\newcommand{\bone}{\boldsymbol{1}}

\newcommand {\ctn}{\citet}       % change to \citet if using natbib
\usepackage{url}
\newcommand{\bu}{\mathbf u}
\newcommand{\bv}{\mathbf v}
\newcommand{\bV}{\mathbf V}
\newcommand{\bW}{\mathbf W}
\newcommand{\bZ}{\mathbf Z}

\newcommand{\bX}{\mathbf X}
\newcommand{\by}{\mathbf y}
\newcommand{\bY}{\mathbf Y}

\newcommand{\bd}{\mathbf d}

\DeclareMathOperator*{\argmax}{argmax}

\newtheorem{theorem}{Theorem}
\newtheorem{definition}{Definition}

\numberwithin{equation}{section}
\numberwithin{algo}{section}
\numberwithin{table}{section}
\numberwithin{figure}{section}

\newtheorem{remark}[theorem]{Remark}

\normalsize
\begin{document}

\title{\textbf{A Bayesian Multiple Testing Paradigm for Model Selection in Inverse Regression Problems}}
\author{Debashis Chatterjee$^{\dag}$ and Sourabh Bhattacharya$^{\dag, +}$ }
\date{}
\maketitle
\begin{center}
%$^{\dag}$  University of Chicago \\
$^{\dag}$ Indian Statistical Institute\\
$+$ Corresponding author:  \href{mailto: bhsourabh@gmail.com}{bhsourabh@gmail.com}
%\href{mailto: kshldey@gmail.com}{kshldey@gmail.com}
%,  \href{mailto: bhsourabh@gmail.com}{bhsourabh@gmail.com}\\
\end{center}

\begin{abstract}

%Model selection is arguably the most important area of statistics, which has received, and is continuing to receive, considerable attention. In the Bayesian
%statistical literature, this field is dominated by Bayes factors, the most principled and coherent approach to model comparison. But in spite of its advantages,
%Bayes factors are usually difficult to compute in practice and suffer from numerical instability. Moreover, it is well-known to suffer from the so-called
%Lindley's paradox. 
%
%The area of multiple hypotheses testing can be envisaged as a promising alternative to Bayes factors for model selection if properly formulated. Unfortunately,
%in spite of rising popularity of the multiple testing paradigm for general testing problems, its applicability and utility in model selection problems
%remain yet to be thoroughly investigated. 
%
%We are specifically interested in inverse regression problems where the objective is to infer about unobserved covariate values from observed responses and covariates,
%and hence from the Bayesian perspective, a prior must be specified for the unknown covariate values. Model selection in such inverse setup is almost non-existent
%in the statistical literature, a recent exception being consideration of pseudo-Bayes factors for such purpose (\ctn{Chat20a}).

Model selection in inverse regression problems where the objective is to infer about unobserved covariate values from observed responses and covariates, is almost non-existent
in the statistical literature, a recent exception being consideration of pseudo-Bayes factors for such purpose (\ctn{Chat20a}).

In this article, we propose a novel Bayesian multiple testing formulation for model and variable selection in inverse setups, judiciously embedding the idea of inverse
reference distributions proposed by \ctn{Bhattacharya13} in a mixture framework consisting of the competing models. We develop the theory and methods
in the general context encompassing parametric and nonparametric competing models, dependent data, as well as misspecifications. Our 
investigation shows that asymptotically the multiple testing procedure almost surely selects the best possible inverse model
that minimizes the minimum Kullback-Leibler divergence from the true model. 
We also show that the error rates, namely, versions of the false discovery rate and the false non-discovery rate
converge to zero almost surely as the sample size goes to infinity. Asymptotic $\alpha$-control of versions of the false discovery rate and its impact
on the convergence of false non-discovery rate versions, are also investigated.

With an aim to compare our multiple testing procedure with pseudo-Bayes factor, we consider the same simulation experiments with the same datasets reported in \ctn{Chat20a}.
The experiments involve small sample based selection among inverse Poisson log regression and inverse geometric logit and probit regression, 
where the regressions are either linear or based on Gaussian processes. Additionally, variable selection is also considered. Our multiple testing results 
turn out to be very encouraging in the sense of selecting the best models in all the cases and convincingly outperforming the pseudo-Bayes factors.  
\\[2mm]
{\bf Keywords:} {\it Bayesian multiple testing; Forward and inverse regression; Importance Resampling MCMC; Kullback-Leibler divergence; Model and variable selection;
Leave-one-out cross-validation.}
\end{abstract}

\section{Introduction}
\label{sec:intro}

Model selection is arguably the most important area of statistics, which has received, and is continuing to receive, considerable attention. 
But in spite of immense importance and popularity of this field, the issue of model selection in the context of inverse regression problems has received almost no attention
in either the classical or the Bayesian statistical literature.

In inverse regression problems the objective is to infer about unobserved covariate values from observed responses and covariates,
and hence from the Bayesian perspective, a prior must be specified for the unknown covariate values. 
Thus, it is in contrast with the traditional forward regression problems where given some covariate values, the response needs to be predicted.
An interesting motivation for the inverse regression setup is the quantitative palaeoclimate reconstruction problem where  
multivariate counts of a number of species
are available along with the observed climate values in modern times. Typically, data collected on or after the year 1950 are regarded as `modern data'. 
Also available are fossil assemblages of the same set of species, but deposited in lake sediments 
for past thousands of years. This is the fossil species data. However, the past climates corresponding to the fossil species data are unknown, and it is of interest
to predict the past climates given the modern data and the fossil species data. Roughly, the species composition are regarded as functions
of climate variables, since in general ecological terms, variations in climate drives variations in species, but not vice versa.
Thus, the species count data, which are the response variables, are modeled as functions of the climate variables, which are the covariates in this case. 
But the interest lies in prediction of climate variables, given the species count data, thereby pointing towards the inverse nature of the problem. 
%Note that from the Bayesian perspective, it is natural to consider a prior distribution on the unobserved covariate values.
%The past climates, which must be regarded as random variables,
%may also be interpreted as {\it unobserved covariate values}. It is thus natural to put a prior probability distribution on the unobserved covariate values.
\ctn{Chatterjee17} provide other examples of inverse regression problems.

As already mentioned, model selection in such inverse setups is almost non-existent
in the statistical literature. A recent exception is the consideration of pseudo-Bayes factors for such purpose (\ctn{Chat20a}).
Pseudo-Bayes factors seem to have been first constructed by \ctn{Geisser79} by combining the ideas of Bayes factor and cross-validation.
Notably, although the Bayes factor approach is arguably the most principled and coherent approach to model comparison, 
Bayes factors are usually difficult to compute in practice and suffer from numerical instability. Moreover, they are well-known to suffer from the so-called
Lindley's paradox. The cross-validation idea proposed by \ctn{Geisser79} is to replace the marginal density of the entire dataset in Bayes factors 
with products of cross-validation densities of individual data points. This constitutes the pseudo-Bayes factors which are computationally far simpler and numerically
much more stable than the corresponding Bayes factors. Furthermore, they are also immune to Lindley's paradox.
Recognizing the importance, \ctn{Chat20a} establish the asymptotic theory for pseudo-Bayes factors for both forward and inverse parametric and nonparametric regression problems
in a very general setup that allows for dependent data and misspecified models.
They illustrate their results with various theoretical examples and simulation experiments for small samples that even include simultaneous selection of models and covariates.
The results of their simulation experiments, although interesting and insightful, do leave the scope for further improvement.

The area of multiple hypotheses testing can be envisaged as a promising alternative to Bayes factors for model selection if properly formulated, and
can bring about the aforementioned desired improvement in inverse model selection. Unfortunately,
in spite of rising popularity of the multiple testing paradigm for general testing problems, its applicability and utility in general model selection problems
remain yet to be thoroughly investigated. 
In the classical multiple comparison context, \ctn{Shim98} use the sampling error of the Akaike Information Criterion (AIC) 
to select a ``confidence set of models" rather than a single model.
The method requires computation of standardized difference of AIC for every pair of models. Since every pair of models is involved, clearly, 
for even a moderate number of competing models
the computation becomes infeasible, and reliability of the proposed normal approximation need not be unquestionable in general situations.
We are not aware of any other significant research on model selection in the multiple testing framework.
Furthermore, multiple testing based model selection in inverse setups has not been hitherto even perceived. 

In this article, for the first time ever, we propose and develop a Bayesian multiple testing paradigm for inverse model selection problems.
Our starting point is the inverse reference distribution approach to Bayesian assessment of adequacy of inverse models introduced by \ctn{Bhattacharya13}.
In a nutshell, the inverse model adequacy assessment idea is as follows. Given response data $\bY_n=\{y_1,\ldots,y_n\}$, covariate data $\bX_n=\{x_1,\ldots,x_n\}$,
and the Bayesian model for the data, consider the inverse leave-one-out cross-validation setup where for each $i=1,\ldots,n$, 
$x_i$ needs to be predicted from the rest of the data and the underlying Bayesian model. Letting
$\tilde x_i$ denote the random variable corresponding to $x_i$ when the latter is treated as unknown, the interest is then in the cross-validation posteriors
$\pi(\tilde x_i|\bX_{n,-i},\bY_n)$; $i=1,\ldots,n$, where $\bX_{n,-i}=\{x_1,\ldots,x_{i-1},x_{i+1},\ldots,x_n\}$. Letting $\tilde\bX_n=\{\tilde x_1,\ldots,\tilde x_n\}$,
\ctn{Bhattacharya13} considers the `inverse reference distribution' of some suitable discrepancy measure 
$T(\tilde\bX_n)$ where $\tilde x_i\sim \pi(\cdot|\bX_{n,-i},\bY_n)$; $i=1,\ldots,n$.
If the observed discrepancy measure $T(\bX_n)$ falls within the desired $100(1-\alpha)\%$ credible interval of $T(\tilde\bX_n)$ where $\alpha\in (0,1)$, then
the underlying Bayesian model fits the data and not otherwise. \ctn{Bhattacharya13} provides a Bayesian decision theoretic formalization of the above idea and 
investigates its theoretical and methodological properties, pointing out its advantages over existing ideas on forward Bayesian model assessment. The encouraging results
obtained in simulation experiments and real data analyses reported in \ctn{Bhattacharya13}, \ctn{Bhatta06} and \ctn{Sabya13} demonstrate the worth of the 
inverse model assessment idea using inverse reference distributions of appropriate discrepancy measures. 
Typical examples of discrepancy measures are given, for any $n$-dimensional vector $\bv_n=(v_1,\ldots,v_n)$, by
\begin{align}
	T_1(\bv_n) &=\sum_{i=1}^n\frac{\left|v_i - E(\tilde x_i|\bX_{n,-i},\bY_{n})\right|}
	{\sqrt{Var(\tilde x_i|\bX_{n,-i},\bY_{n})}}\label{eq:T1_old}
	%T_1(\tilde\bX_n) &=\frac{1}{n} \sum_{i=1}^n\frac{(\tilde x_i - E_{\pi}(\tilde x_i))^2}{V_{\pi}(\tilde x_i)+c},\label{eq:T1_1}\\
	%T_1(\bX_n) &=\frac{1}{n} \sum_{i=1}^n\frac{(x_i - E_{\pi}(\tilde x_i))^2}{V_{\pi}(\tilde x_i)+c},\label{eq:T1_2}\\
	%T_1(\bX^*_n) &=\frac{1}{n} \sum_{i=1}^n\frac{(x^*_i - E_{\pi}(\tilde x_i))^2}{V_{\pi}(\tilde x_i)+c}.\label{eq:T1_3}
\end{align}
and
\begin{align}
	T_2(\bv_n) &=\sum_{i=1}^n\frac{(v_i - E(\tilde x_i|\bX_{n,-i},\bY_{n}))^2}
	{Var(\tilde x_i|\bX_{n,-i},\bY_{n})}.\label{eq:T2_old}
	%T_2(\tilde\bX_n) &=\frac{1}{n} \sum_{i=1}^n\frac{\left|\tilde x_i - E_{\pi}(\tilde x_i))\right|}{\sqrt{V_{\pi}(\tilde x_i)+c}},\label{eq:T2_1}\\
	%T_2(\bX_n) &=\frac{1}{n} \sum_{i=1}^n\frac{\left|x_i - E_{\pi}(\tilde x_i)\right|}{\sqrt{V_{\pi}(\tilde x_i)+c}},\label{eq:T2_2}\\
	%T_2(\bX^*_n) &=\frac{1}{n} \sum_{i=1}^n\frac{\left|x^*_i - E_{\pi}(\tilde x_i)\right|}{\sqrt{V_{\pi}(\tilde x_i)+c}}.\label{eq:T2_3}
\end{align}
Since the inverse reference distribution approach turned out to be useful for assessing adequacy of inverse models, it is natural to discern that such 
an approach would be valuable even for inverse model selection. This very perception provided the motivation for our Bayesian multiple testing approach to
inverse model selection using inverse reference distributions. The key idea is to embed all the competing inverse regression models in a mixture setting
to constitute a single model needed for multiple testing. In simple terms, each hypothesis of the multiple testing procedure then essentially tests if the inverse
reference distribution of the corresponding inverse regression model gives high posterior probability to appropriate regions containing 
the observed discrepancy measure for the model,
in addition to testing if the posterior model probability is sufficiently high. 
The best inverse model is expected to have the highest posterior probability with respect to the above and our multiple testing formalism is so designed that
it renders this idea precise with relevant coherent supports.

Our theoretical and methodological development deals with parametric and nonparametric inverse competing models, allowing dependent data as well as 
misspecified models. In this highly general framework we show that our multiple testing procedure
almost surely selects the best possible model, as the sample size tends tends to infinity. Here ``best" is in terms of 
the minimizer of the minimum Kullback-Leibler (KL) divergence from the true model, concepts that will be subsequently clarified. 
Our investigation also brings out the desirable results that the error rates, namely, relevant versions of the false discovery rate and the false non-discovery rate,
asymptotically converge to zero almost surely. Insightful theoretical results on asymptotic $\alpha$-control of versions of the false discovery rate and its impact
on the convergence of versions of the false non-discovery rate, are also presented.

Monte Carlo based computations of the model-specific posterior probabilities associated with the inverse reference distributions  
proceed via fast and efficient Importance Re-sampling Markov Chain Monte Carlo (IRMCMC) (\ctn{Bhatta07}) 
aided by Transformation based Markov Chain Monte Carlo (TMCMC) (\ctn{Dutta13}) for generation of MCMC samples from the cross-validation posterior distributions
having excellent mixing properties. The posterior model probabilities are based on an efficient Gibbs sampling scheme that utilizes the forward pseudo-Bayes factors
for sampling from the relevant full conditional distributions of the model indices. Thus, our entire computational methodology is fast and efficient, more so because
each hypothesis is associated with a single inverse model, and pairwise comparison as in \ctn{Shim98} is ruled out.

Recalling that one of our objectives behind development of this multiple testing paradigm is to obtain superior inverse model selection results compared to 
those obtained by \ctn{Chat20a} using pseudo-Bayes factors, we apply our multiple testing formalism to the same simulation experiments with the 
same datasets as in \ctn{Chat20a}. The simulation experiments consist of two sets. In one set small sample based selection among inverse Poisson log regression and 
inverse geometric logit and probit regression is considered, where the regressions are either linear or based on Gaussian processes. In the other set,
variable selection among two covariates is considered in addition to the aforementioned inverse model selection problem. We conduct the experiments
in both non-misspecified and misspecified situations.
Not only does our multiple testing procedure succeeds in selecting the best inverse models and variables in all the cases, it significantly outperforms
the results yielded by the pseudo-Bayes factors.

The rest of our paper is structured as follows. We begin by distinguishing  
forward and inverse regression problems in Section \ref{sec:prelims}.
In Section \ref{sec:mult_inv} we introduce and develop our Bayesian multiple testing paradigm for inverse model selection. 
Then in Section \ref{sec:shalizi_briefing} we include a brief overview of Shalizi's approach (\ctn{Shalizi09}) to dealing with posterior convergence which plays a significant
role in the development of the asymptotic theory of our multiple testing procedure; further details are provided in Appendix \ref{subsec:assumptions_shalizi}.
We progress towards a general asymptotic theory by establishing in Section \ref{sec:asymp_v} the asymptotic properties of the posterior probabilities of the
alternative hypotheses. Asymptotic optimality theory for our multiple testing procedure is then provided in Section \ref{sec:optimality}, followed by
convergence theory of the measures of error in Section \ref{sec:error_asymptotics}. 
In Section \ref{sec:modification_practical} we recommend some judicious modifications of the hypotheses to suit practical implementation, and in Sections
\ref{sec:simstudy_ms} and \ref{sec:simstudy_vs} we provide details on two sets of simulation experiments with small samples involving Poisson and geometric 
linear and Gaussian process regression for relevant link functions, the second set also including in addition the problem of variable selection involving two covariates. 
Non-misspecified and misspecified situations are addressed in both the simulation experiments. 
Finally, in Section \ref{sec:conclusion}, we summarize our contributions and discuss selection of inverse models in the context of two palaeoclimate reconstruction problems,
recasting our previous results on inverse model assessment in the current multiple testing context.

\section{Distinction between forward and inverse regression problems}
\label{sec:prelims}
Here we essentially follow the discussion provided in \ctn{Chat20a}.
%Let us first consider the forward regression setup.
\subsection{Forward regression problem}
\label{subsec:forward}
For $i=1,\ldots,n$, let observed response $y_i$ be related to observed covariate $x_i$ through
\begin{equation}
	y_1\sim f(\cdot|\theta,x_1)~\mbox{and}~y_i\sim f(\cdot|\theta,x_i,\bY^{(i-1)})~\mbox{for}~i=2,\ldots,n,
	\label{eq:forward0}
\end{equation}
where for $i=2,\ldots,n$, $\bY^{(i)}=\{y_1,\ldots,y_i\}$ and $f(\cdot|\theta,x_1)$, $f(\cdot|\theta,x_i,\bY^{(i-1)})$ are known densities depending upon 
(a set of) parameters $\theta\in\Theta$, where $\Theta$ is the parameter space, which may be infinite-dimensional. 
For the sake of generality, we shall consider $\theta=(\eta,\xi)$, where $\eta$ is a function of the covariates, which we more explicitly denote as $\eta(x)$. 
The covariate $x\in\mathcal X$, $\mathcal X$ being the space of covariates. The part $\xi$ of $\theta$ will be assumed to consist of other parameters, such as the unknown
error variance. For Bayesian forward regression problems, some prior needs to be assigned on the parameter space $\Theta$.
For notational convenience, we shall denote $f(\cdot|\theta,x_1)$ by $f(\cdot|\theta,x_1,\bY^{(0)})$, so that we can represent (\ref{eq:forward0}) more conveniently
as 
\begin{equation}
	y_i\sim f(\cdot|\theta,x_i,\bY^{(i-1)})~\mbox{for}~i=1,\ldots,n.
	\label{eq:forward1}
\end{equation}
%We assume that given $\theta$, $y_i$ are conditionally independent.

\subsubsection{Examples of the forward regression setup}
\label{subsubsec:model_examples_forward}
\begin{itemize}
%\item[(i)] $y_i=\alpha+\beta x_i+\epsilon_i$, where $\theta=(\alpha,\beta)$ and $\epsilon_i$ are $iid$ Gaussian errors.
\item[(i)] $y_{i}\sim Bernoulli(p_i)$, where $p_i=H\left(\eta(x_i)\right)$, where $H$ is some appropriate link function and $\eta$ is some 
function with known or unknown form. For known, suitably parameterized form, the model is parametric. If the form of $\eta$ is unknown, 
one may model it by a Gaussian process, assuming adequate smoothness of the function.
\item[(ii)] $y_{i}\sim Poisson(\lambda_i)$, where $\lambda_i=H\left(\eta(x_i)\right)$, where $H$ is some appropriate link function and $\eta$ is 
some function with known (parametric) or unknown (nonparametric) form. Again, in case of unknown form of $\eta$, 
the Gaussian process can be used as a suitable model under sufficient smoothness
assumptions.
\item[(iii)] $y_{i}=\eta(x_i)+\epsilon_{i}$, where $\eta$ is a parametric or nonparametric function and
$\epsilon_{i}$ are $iid$ Gaussian errors. In particular, $\eta(x_i)$ may be a linear regression function, that is, $\eta(x_i)=\beta'x_i$, where
$\beta$ is a vector of unknown parameters.
Non-linear forms of $\eta$ are also permitted. 
Also, $\eta$ may be a reasonably smooth function of unknown form, modeled by some appropriate Gaussian process.		
\end{itemize}

\subsection{Inverse regression problem: first setup}
\label{subsec:inverse}

%The distribution $f_\theta$, parameter $\theta$, the parameter and the covariate space remain the same as in the forward regression setup. 
%Unlike in Bayesian forward regression problems where a prior needs to be assigned only to the unknown parameter $\theta$, a prior is also required for
%$\tilde x$, the unknown covariate observation associated with known response $\tilde y$, say. Given the entire dataset and $\tilde y$, the problem
%in inverse regression problems is to predict $\tilde x$. Hence, in the Bayesian inverse setup, a prior on $\tilde x$ is necessary.
%In (\ref{model}), $f_\theta$ is a known distribution depending upon 
%(a set of) parameters $\theta\in\Theta$, where $\Theta$ is the parameter space, which may be infinite-dimensional. 
%For the same of generality, we shall consider $\theta=(\eta,\xi)$, where $\eta$ is a function of the covariates, which we more explicitly denote as $\eta(x)$, 
%where $x\in\mathcal X$, $\mathcal X$ being the space of covariates. The part $\xi$ of $\eta$ will be assumed to consist of other parameters, such as the unknown
%error variance.

In inverse regression, the basic premise remains the same as in forward regression detailed in Section \ref{subsec:forward}. In other words,
the distribution $f(\cdot|\theta,x_i,\bY^{(i-1)})$, parameter $\theta$, the parameter and the covariate space remain the same as in the forward regression setup. 
However, unlike in Bayesian forward regression problems where a prior needs to be assigned only to the unknown parameter $\theta$, a prior is also required for
$\tilde x$, the unknown covariate observation associated with known response $\tilde y$, say. Given the entire dataset and $\tilde y$, the problem
in inverse regression is to predict $\tilde x$. Hence, in the Bayesian inverse setup, a prior on $\tilde x$ is necessary. Given model $\mathcal M$
and the corresponding parameters $\theta$, we denote such prior
by $\pi(\tilde x|\theta,\mathcal M)$. 
%For Bayesian cross-validation in inverse problems it is pertinent to successively leave out $(y_i,x_i)$; $i=1,\ldots,n$,
%and compute the posterior predictive distribution $\pi(\tilde x_i|\bY_n,\bX_{n,-i})$, from $y_i$ and the rest of the data $(\bY_{n,-i},\bX_{n,-i})$
%(see \ctn{Bhatta07}).
%But these posteriors are not useful for Bayes of pseudo-Bayes factors even for inverse regression setups. 
%%although the prior $\pi(\tilde x_i|\theta,\mathcal M)$
%%plays a significant role in inverse model comparison with Bayes of pesudo-Bayes factors. 
%The reason is that the Bayes factor for inverse regression is still the ratio of posterior odds and prior odds associated with the competing models, which
%as usual translates to the ratio of the marginal densities of the data under the two competing models. The marginal densities depend upon the prior 
%for $(\theta,\tilde x)$, however, under the competing models. The pseudo-Bayes factor for inverse models is then the ratio of products of the cross-validation
%posteriors of $y_i$, where $\theta$ and $\tilde x_i$ are marginalized out. Details of such inverse cross-validation posteriors and the definition of
%pseudo-Bayes factors for inverse regression are given below.

\subsection{Inverse regression problem: second setup}
\label{subsec:inverse2}
In the inverse regression context, we consider another setup under which \ctn{Chat20} establish consistency of the inverse cross-validation posteriors 
of $\tilde x_i$. Here we consider experiments with covariate observations $x_1, x_2, \ldots, x_n $ along with responses 
$\bY_{nm}=\{y_{ij} :i=1,\ldots,n, j=1,\ldots,m\}$.
In other words, the experiment considered here will allow us to have $m$ samples of responses $\by_i=\{y_{i1}, y_{i2}, \ldots, y_{im}\}$ against each covariate observation
$x_{i}$, for $i=1, 2,\ldots, n$. Again, both $x_i$ and $y_{ij}$ are allowed to be multidimensional. Let $\bY_{nm,-i}=\bY_{nm}\backslash\{\by_i\}$.
%We shall investigate the large sample scenario where both $m, n\rightarrow \infty$. 
%We will denote the covariate space along with assumption of 
%compactness (Assumption \ref{cov1}) by  $\mathfrak{X}$  and response space  by $\mathfrak{Y} \subset \Re$.  

For $i=1,\ldots,n$ %and $j=1,\ldots,m$, 
consider the following general model setup: %as \eqref{model} for $i \in \{1, 2, \cdots, n\}, j \in \{1, 2, \cdots, m\}$.
conditionally on $\theta$, $x_i$ and $\bY^{(i-1)}_j=\{y_{1j},\ldots,y_{i-1,j}\}$,
\begin{eqnarray}\label{model}
\begin{aligned}
	& y_{ij} \sim f\left(\cdot|\theta,x_i,\bY^{(i-1)}_j\right);~j=1,\ldots,m, % f_{\theta}\left(x_i\right), 
%& X_i\sim Q.
\end{aligned}
\end{eqnarray}
%and that for every $i\geq 1$, $y_{i1},\ldots,y_{im}$ are conditionally independent, given $\theta$ and $x_i$. 
independently, where $f(\cdot|\theta,x_1,\bY^{(0)})=f(\cdot|\theta,x_1)$ as before.

%\subsubsection{Examples of the inverse regression setup}
%\label{subsubsec:model_examples_inverse}
%For the inverse regression setup, the examples in Section \ref{subsubsec:model_examples_forward} can be generalized as follows:
%\begin{itemize}
%%	\item[(i)] $y_{ij}=\alpha+\beta x_i+\epsilon_{ij}$, where $\theta=(\alpha,\beta)$ and $\epsilon_{ij}$ are $iid$ Gaussian errors.
%\item[(i)] $y_{ij}\sim Bernoulli(p_i)$, where $p_i=H\left(\eta(x_i)\right)$, where $H$ is some appropriate link function and $\eta$ is some 
%function with known or unknown form. For known, suitably parameterized form, the model is parametric. If the form of $\eta$ is unknown, 
%one may model it by a Gaussian process, assuming adequate smoothness of the function.
%\item[(ii)] $y_{ij}\sim Poisson(\lambda_i)$, where $\lambda_i=H\left(\eta(x_i)\right)$, where $H$ is some appropriate link function and $\eta$ is 
%some function with known (parametric) or unknown (nonparametric) form. Again, in case of unknown form of $\eta$, 
%the Gaussian process can be used as a suitable model under sufficient smoothness
%assumptions.
%\item[(iii)] $y_{ij}=\eta(x_i)+\epsilon_{ij}$, where $\eta$ is a parametric or nonparametric function and
%$\epsilon_{ij}$ are $iid$ Gaussian errors. In particular, $\eta(x_i)$ may be a linear regression function, that is, $\eta(x_i)=\beta'x_i$, where
%$\beta$ is a vector of unknown parameters.
%Non-linear forms of $\eta$ are also permitted. 
%Also, $\eta$ may be a reasonably smooth function of unknown form, modeled by some appropriate Gaussian process.		
%\end{itemize}

\subsubsection{Prior for $\tilde x_i$}
\label{subsec:prior}
Following \ctn{Chat20}, we consider the following prior for $\tilde x_i$: given $\theta$,
\begin{equation}
	\tilde x_i\sim U\left(B_{im}(\theta)\right),
\label{eq:prior_x}
\end{equation}
the uniform distribution on 
\begin{equation}
	B_{im}(\theta)=\left(\left\{x:H\left(\eta(x)\right)\in \left[\bar y_i-\frac{cs_i}{\sqrt{m}},\bar y_i+\frac{cs_i}{\sqrt{m}}\right]\right\}\right),
\label{eq:set1}
\end{equation}
where $H$ is some suitable transformation of $\eta(x)$.
In (\ref{eq:set1}), $\bar y_i=\frac{1}{m}\sum_{j=1}^my_{ij}$ and $s^2_i=\frac{1}{m-1}\sum_{j=1}^m(y_{ij}-\bar y_i)^2$, and $c\geq 1$ is some constant.
We denote this prior by $\pi(\tilde x_i|\eta)$. \ctn{Chat20} show that the density or any probability 
associated with $\pi(\tilde x_i|\eta)$ is continuous with respect to $\eta$.
Quite importantly, the prior form (\ref{eq:prior_x}) leads to cross-validation posteriors that are consistent at $x_i$; see \ctn{Chat20}. 

\subsubsection{Examples of the prior}
\label{subsubsec:illustrations_prior}
\begin{itemize}
	\item[(i)] $y_{ij}\sim Poisson(\theta x_i)$, where $\theta>0$ and $x_i>0$ for all $i$. Here, under the prior $\pi(\tilde x_i|\theta)$, 
		$\tilde x_i$ has uniform distribution on the set 
		$B_{im}(\theta)=\left\{x>0:\frac{\bar y_i-\frac{cs_i}{\sqrt{m}}}{\theta}\leq x\leq \frac{\bar y_i+\frac{cs_i}{\sqrt{m}}}{\theta}\right\}$.
	\item[(ii)] $y_{ij}\sim Poisson(\lambda_i)$, where $\lambda_i=\lambda(x_i)$, with $\lambda(x)=H(\eta(x))$. Here $H$ is a known, one-to-one, 
		continuously differentiable function and $\eta(\cdot)$ is an unknown function modeled by Gaussian process.
		Here, the prior for $\tilde x_i$ is the uniform distribution on $$B_{im}(\eta)=\left\{x:\eta(x)\in 
		H^{-1}\left\{\left[\bar y_i-\frac{cs_i}{\sqrt{m}},\bar y_i+\frac{cs_i}{\sqrt{m}}\right]\right\}\right\}.$$
	\item[(iii)] $y_{ij}\sim Bernoulli(p_i)$, where $p_i=\lambda(x_i)$, with $\lambda(x)=H(\eta(x))$. Here $H$ is a known, increasing,
		continuously differentiable, cumulative distribution function and $\eta(\cdot)$ is an unknown function modeled by some appropriate Gaussian process.
		Here, the prior for $\tilde x_i$ is the uniform distribution on $B_{im}(\eta)=\left\{x:\eta(x)\in 
		H^{-1}\left\{\left[\bar y_i-\frac{cs_i}{\sqrt{m}},\bar y_i+\frac{cs_i}{\sqrt{m}}\right]\right\}\right\}$.
	\item[(iv)] $y_{ij}=\eta(x_i)+\epsilon_{ij}$, where $\eta(\cdot)$ is an unknown function modeled by some appropriate Gaussian process,
		and $\epsilon_{ij}$ are $iid$ zero-mean Gaussian noise with variance $\sigma^2$. 
		Here, the prior for $\tilde x_i$ is the uniform distribution on $B_{im}(\eta)=\left\{x:\eta(x)\in 
		\left[\bar y_i-\frac{cs_i}{\sqrt{m}},\bar y_i+\frac{cs_i}{\sqrt{m}}\right]\right\}$.
		If $\eta(x_i)=\alpha+\beta x_i$, then the prior for $\tilde x_i$ is the uniform distribution on $[a,b]$, where 
		%$\left[\frac{\bar y_i-\alpha}{\beta}-\frac{cs_i}{\sqrt{m}},\frac{\bar y_i-\alpha}{\beta}+\frac{cs_i}{\sqrt{m}}\right]$.
		$a=\min\left\{\frac{\bar y_i-\frac{cs_i}{\sqrt{m}}-\alpha}{\beta},\frac{\bar y_i+\frac{cs_i}{\sqrt{m}}-\alpha}{\beta}\right\}$
		and $b=\max\left\{\frac{\bar y_i-\frac{cs_i}{\sqrt{m}}-\alpha}{\beta},\frac{\bar y_i+\frac{cs_i}{\sqrt{m}}-\alpha}{\beta}\right\}$.
\end{itemize}
Further examples of the prior in various other inverse regression models are provided in \ctn{Chat20a}; see also Sections \ref{sec:simstudy_ms} and \ref{sec:simstudy_vs}.
In this article, we shall throughout assume that the space of covariates $\mathcal X$ is compact.

\section{A multiple testing framework for model selection in inverse regression problems}
\label{sec:mult_inv}
Let us consider models $\mathcal M_k$; $k=1,\ldots,K$, from among which the best model needs to be selected respecting the inverse perspective. 
%Here $K>1$; in fact, we allow $K=\infty$ as well. 
In this article, we assume that $1< K<\infty$.
We allow the provision that the true, data-generating model is not contained in the set of models being considered.
For $k=1,\ldots,K$, let $\theta_k$ and $\Theta_k$ denote the parameter set and the parameter space associated with model $\mathcal M_k$. 
Let $\pi(\theta_k|\mathcal M_k)$ denote the prior for $\theta_k$ under model $\mathcal M_k$.

For our multiple testing treatise, we shall consider the second inverse regression setup detailed in Section \ref{subsec:inverse2}. As such, for
$n>1$ and $m>1$, let $\bY_{nm}$ be generated from the marginal distribution of $\mathcal M_0$, the true model having parameters $\theta_0$ with prior 
$\pi(\theta_0|\mathcal M_0)$ on parameter space $\Theta_0$. Note that $\pi(\theta_0|\mathcal M_0)$ may even be the point mass on some element of $\Theta_0$.
The dimensions of the parameter spaces $\Theta_0,\Theta_1,\ldots,\Theta_K$ may all be different. We shall consider the consistent prior for $\tilde x_i$ detailed
in Section \ref{subsec:prior}.

%Letting $$f(\bY_n|\theta_k,\mathcal M_k)=\int_{\mathcal X^n}f(\bY_n|\tilde\bX_n,\theta_k,\mathcal M_k)d\pi(\tilde \bX_n|\theta_k,\mathcal M_k),$$ 
Now, for $k=1,\ldots,K$, let $f(\bY_{nm}|\bX_n,\theta_k,\mathcal M_k)$ denote the density of $\bY_{nm}$ under model $\mathcal M_k$. 
We combine the competing models in the following mixture form:
\begin{equation}
	f(\bY_{nm}|\bX_n,\theta)=\sum_{k=1}^Kp_kf(\bY_{nm}|\bX_n,\theta_k,\mathcal M_k),
	\label{eq:mix1}
\end{equation}
where $\theta=(\theta_1,\ldots,\theta_K)$, $0\leq p_k\leq 1$, for $k=1,\ldots,K$ and $\sum_{k=1}^Kp_k=1$.
Letting $\zeta$ denote the allocation variable (model index), with $P(\zeta=k)=p_k$, note that $f(\bY_{nm}|\bX_n,\theta,\zeta=k)=f(\bY_{nm}|\bX_n,\theta_k,\mathcal M_k)$. 
Now let $\tilde\Theta_k$ be a proper subset of $\Theta_k$ assumed to contain the minimizer of the KL-divergence from the true model $\mathcal M_0$.  
%The true model $\mathcal M_0$ may be of the form $f(\bY_n|\bX_n,\theta_0,\mathcal M_0)$ or 
%$$f(\bY_n|\theta_0,\mathcal M_0)=\int_{\mathcal X^n}f(\bY_n|\tilde\bX_n,\theta_0,\mathcal M_0)d\pi(\tilde \bX_n|\theta_0,\mathcal M_0).$$

Let $\pi(\tilde x_i|\theta_k,\mathcal M_k)$ be the prior for $\tilde x_i$ given $\theta_k$, under $\mathcal M_k$.
This yields the familiar (see, for example, \ctn{Bhatta07}, \ctn{Chat20}) inverse cross-validation posterior for $\tilde x_i$ given $\bX_{n,-i}$ and $\bY_{nm}$ given by
\begin{equation*}
	\pi(\tilde x_i|\bX_{n,-i},\bY_{nm},\mathcal M_k)=\int_{\Theta_k}\pi(\tilde x_i|\theta_k,\by_i,\mathcal M_k)d\pi(\theta_k|\bX_{n,-i},\bY_{nm}).
	%\label{eq:cv_post2}
\end{equation*}
However, if $\theta_k$ is restricted to $\tilde\Theta_k$, then we obtain the following $\tilde\Theta_k$-restricted
inverse cross-validation posterior for $\tilde x_i$ given $\bX_{n,-i}$ and $\bY_n$:
\begin{equation}
	\pi(\tilde x_i|\bX_{n,-i},\bY_{nm},\mathcal M_k,\tilde\Theta_k)
	=\frac{\int_{\tilde\Theta_k}\pi(\tilde x_i|\theta_k,\by_i,\mathcal M_k)d\pi(\theta_k|\bX_{n,-i},\bY_{nm})}{\pi(\tilde\Theta_k|\bX_{n,-i},\bY_{nm})}.
	\label{eq:cv_post2}
\end{equation}
In the misspecified situation, $\theta_0\notin\Theta_{k}$, and $\tilde\theta_{k}$ is the minimizer of the limiting KL-divergence rate from $\mathcal M_0$.
Thus, in the case of misspecification of $\theta_k$, $B_{im}(\tilde\theta_{k})\stackrel{a.s.}{\longrightarrow}\{x^*_{ik}\}$ 
as $m\rightarrow\infty$, for some non-random $x^*_{ik}~(\neq x_i)$, depending upon model $\mathcal M_k$. 
In other words, the prior distribution of $\tilde x_i$ given $\tilde\theta_{k}$ and $\by_i$ 
concentrates around $x^*_{ik}$, as $m\rightarrow\infty$. In Theorem \ref{theorem:inv_cons} we show that the cross-validation posterior of 
$\tilde x_i$ also concentrates around $x^*_{ik}$.
Note that $x^*_{ik}$ depends upon both $\tilde\theta_{k}$ and $\theta_0$, apart from $x_i$ (and perhaps $x_j$ for some $j\neq i$).

For any $n$-dimensional vector $\bv_n=(v_1,\ldots,v_n)$, and for some $c>0$, define
\begin{align}
	T^{(k)}_1(\bv_n) &=\frac{1}{n} \sum_{i=1}^n\frac{\left|v_i - E(\tilde x_i|\bX_{n,-i},\bY_{nm},\mathcal M_k,\tilde\Theta_k)\right|}
	{\sqrt{Var(\tilde x_i|\bX_{n,-i},\bY_{nm},\mathcal M_k,\tilde\Theta_k)+c}}.\label{eq:T1}
	%T_1(\tilde\bX_n) &=\frac{1}{n} \sum_{i=1}^n\frac{(\tilde x_i - E_{\pi}(\tilde x_i))^2}{V_{\pi}(\tilde x_i)+c},\label{eq:T1_1}\\
	%T_1(\bX_n) &=\frac{1}{n} \sum_{i=1}^n\frac{(x_i - E_{\pi}(\tilde x_i))^2}{V_{\pi}(\tilde x_i)+c},\label{eq:T1_2}\\
	%T_1(\bX^*_n) &=\frac{1}{n} \sum_{i=1}^n\frac{(x^*_i - E_{\pi}(\tilde x_i))^2}{V_{\pi}(\tilde x_i)+c}.\label{eq:T1_3}
\end{align}
Similarly, let
\begin{align}
	T^{(k)}_2(\bv_n) &=\frac{1}{n} \sum_{i=1}^n\frac{(v_i - E(\tilde x_i|\bX_{n,-i},\bY_{nm},\mathcal M_k,\tilde\Theta_k))^2}
	{Var(\tilde x_i|\bX_{n,-i},\bY_{nm},\mathcal M_k,\tilde\Theta_k)+c}.\label{eq:T2}
	%T_2(\tilde\bX_n) &=\frac{1}{n} \sum_{i=1}^n\frac{\left|\tilde x_i - E_{\pi}(\tilde x_i))\right|}{\sqrt{V_{\pi}(\tilde x_i)+c}},\label{eq:T2_1}\\
	%T_2(\bX_n) &=\frac{1}{n} \sum_{i=1}^n\frac{\left|x_i - E_{\pi}(\tilde x_i)\right|}{\sqrt{V_{\pi}(\tilde x_i)+c}},\label{eq:T2_2}\\
	%T_2(\bX^*_n) &=\frac{1}{n} \sum_{i=1}^n\frac{\left|x^*_i - E_{\pi}(\tilde x_i)\right|}{\sqrt{V_{\pi}(\tilde x_i)+c}}.\label{eq:T2_3}
\end{align}
In (\ref{eq:T1}) and (\ref{eq:T2}), $\tilde x_i$ has the cross-validation posterior distribution (\ref{eq:cv_post2}), for $i=1,\ldots,n$.
The positive constant $c$ is not only needed for asymptotics, it plays the role of maintaining stability of the discrepancy measures when 
$Var(\tilde x_i|\bX_{n,-i},\bY_{nm},\mathcal M_k,\tilde\Theta_k)$ is close to zero for some $i\geq 1$.
Various other measures of discrepancy can be defined (see \ctn{Bhattacharya13} for a discussion on such discrepancy measures; see also \ctn{Sabya13}), 
but for brevity we focus on these two measures in this paper.

For a given discrepancy measure $T^{(k)}$, 
let $[\tilde\ell_{knm},\tilde u_{knm}]$ denote the $100(1-\alpha)\%$ credible interval for the posterior distribution of $T^{(k)}(\tilde\bX_n)$ for any desired $\alpha\in (0,1)$.
In Theorem \ref{theorem:consistency_T} we show that for any $\varepsilon>0$, the posterior probability of the event
$$\left\{T^{(k)}(\tilde \bX_n)-T^{(k)}(\bX_n)\in [\tilde\ell_{knm}-a_{k}-\varepsilon,\tilde u_{knm}-a_k+\varepsilon]\right\}$$ tends to one almost surely as $m\rightarrow\infty$
and $n\rightarrow\infty$. Here $a_k$ are positive constants reflecting misspecification. If there is no misspecification, then $a_k=0$.

With the above notions and ideas it seems reasonable to formulate the following multiple testing problem for inverse model selection. 
For given $\varepsilon>0$ and $\eta>0$, and given discrepancy measure $T^{(k)}$
associated with model $\mathcal M_k$,
for $k=1,\ldots,K$, consider testing
\begin{equation*}
	H_{0k}:p_k>1-\eta,\theta_k\in\tilde\Theta_k,
	%\frac{\left|T^{(k)}(\tilde \bX_n)-T^{(k)}(\bX_n)\right|}{\sqrt{Var\left(T^{(k)}(\tilde \bX_n)|\bX_n,\bY_{nm},\mathcal M_k,\tilde\Theta_k\right)}}
	T^{(k)}(\tilde \bX_n)-T^{(k)}(\bX_n) \in [\tilde\ell_{knm}-a_k-\varepsilon, \tilde u_{knm}-a_k+\varepsilon]
	%\label{eq:H0_1}
\end{equation*}
versus
\begin{align*}
	&H_{1k}:\left\{p_k\leq 1-\eta\right\}\bigcup\left\{p_k>1-\eta,\theta_k\in\tilde\Theta^c_k\right\}\notag\\
	&\qquad\qquad\bigcup\left\{p_k>1-\eta,\theta_k\in\tilde\Theta_k,
	%\frac{\left|T^{(k)}(\tilde \bX_n)-T^{(k)}(\bX_n)\right|}{\sqrt{Var\left(T^{(k)}(\tilde \bX_n)|\bX_n,\bY_{nm},\mathcal M_k,\tilde\Theta_k\right)}}
	%\in (a_k-\varepsilon,a_k+\varepsilon)^c\right\}.
	T^{(k)}(\tilde \bX_n)-T^{(k)}(\bX_n) \in [\tilde\ell_{knm}-a_k-\varepsilon, \tilde u_{knm}-a_k+\varepsilon]^c\right\}.
	%\label{eq:H1_1}
\end{align*}
The positive constants $a_k$ in the hypotheses should be perceived as analogous to $a_{1k}$ and $a_{2k}$ in (\ref{eq:T1_bound}) and (\ref{eq:T2_bound}).

However, the above multiple testing formulation depends upon the choice of $\eta$. More importantly, even though the posterior probability of $\zeta=\tilde k$ goes
to $1$ asymptotically for the best model $\mathcal M_{\tilde k}$, that of $\left\{p_k>1-\eta\right\}$, for any $\eta>0$, does not tend to one for any prior on $(p_1,\ldots,p_K)$.
For example, for a Dirichlet prior with parameters $(\alpha_1,\ldots,\alpha_K)$, where $\alpha_k>0$ for $k=1,\ldots,K$, the posterior distribution
of $(p_1,\ldots,p_K)$ given $\zeta$, the other parameters and the data, is Dirichlet with parameters 
$(\alpha_1+I(\zeta=1),\ldots,\alpha_K+I(\zeta=K))$, where for any $k$, $I(\zeta=k)=1$ if $\zeta=k$ and zero otherwise. Thus, even if $\zeta=\tilde k$ with posterior
probability tending to one, asymptotically the posterior distribution of $p_{\tilde k}$ does not converge to one. It is thus necessary to modify the above multiple testing
formulation, replacing the statements involving $p_k$ with those involving $\zeta$. Specifically, we re-write the hypotheses as follows:
\begin{equation}
	H_{0k}:\zeta=k,\theta_k\in\tilde\Theta_k,
	%\frac{\left|T^{(k)}(\tilde \bX_n)-T^{(k)}(\bX_n)\right|}{\sqrt{Var\left(T^{(k)}(\tilde \bX_n)|\bX_n,\bY_{nm},\mathcal M_k,\tilde\Theta_k\right)}}
	%\in (a_k-\varepsilon, a_k+\varepsilon)
	T^{(k)}(\tilde \bX_n)-T^{(k)}(\bX_n) \in [\tilde\ell_{knm}-a_k-\varepsilon, \tilde u_{knm}-a_k+\varepsilon]
	\label{eq:H0}
\end{equation}
versus
\begin{align}
	&H_{1k}:\left\{\zeta\neq k\right\}\bigcup\left\{\zeta=k,\theta_k\in\tilde\Theta^c_k\right\}\notag\\
	&\qquad\qquad\bigcup\left\{\zeta=k,\theta_k\in\tilde\Theta_k,
	%\frac{\left|T^{(k)}(\tilde \bX_n)-T^{(k)}(\bX_n)\right|}{\sqrt{Var\left(T^{(k)}(\tilde \bX_n)|\bX_n,\bY_{nm},\mathcal M_k,\tilde\Theta_k\right)}}
	%\in (a_k-\varepsilon,a_k+\varepsilon)^c\right\}.
	T^{(k)}(\tilde \bX_n)-T^{(k)}(\bX_n) \in [\tilde\ell_{knm}-a_k-\varepsilon, \tilde u_{knm}-a_k+\varepsilon]^c\right\}.
	\label{eq:H1}
\end{align}
Henceforth, unless stated otherwise, we shall refer to (\ref{eq:H0}) and (\ref{eq:H1}) for our multiple testing purpose.

\subsection{Further discussion of the multiple testing formulation} 
\label{subsec:key}
%The idea behind the multiple testing formulation (\ref{eq:H0}) and (\ref{eq:H1}) is as follows. 
To select the best model from an inverse perspective 
we first need to choose a model $f(\bY_{nm}|\bX_n,\theta_{\tilde k},\mathcal M_{\tilde k})$ indexed by $\zeta=\tilde k$ 
which has high marginal posterior probability. But this is not enough
as the inverse context is not reflected in this selection. Indeed, such a selection is the same as in the forward context.
%the only difference being that the competing forward models are of the form $f(\bY_n|\bX_n,\theta_k,\mathcal M_k)$, where $\bX_n$ are retained.

Thus, in addition to selecting such a $\tilde k$, we demand that for such model 
\begin{equation}
	%\frac{\left|T^{(k)}(\tilde \bX_n)-T^{(k)}(\bX_n)\right|}{\sqrt{Var\left(T^{(k)}(\tilde \bX_n)|\bX_n,\bY_{nm},\mathcal M_k,\tilde\Theta_k\right)}}
	%\in (a_k-\varepsilon, a_k+\varepsilon).
	T^{(k)}(\tilde \bX_n)-T^{(k)}(\bX_n) \in [\tilde\ell_{knm}-a_k-\varepsilon, \tilde u_{knm}-a_k+\varepsilon].
	\label{eq:inv_test1}
\end{equation}
%holds in the same way as proposed in \ctn{Bhattacharya13}. 
This reflects the inverse perspective.
We further demand that this holds for $\tilde\bX_n$ associated with some region $\tilde\Theta_{\tilde k}$ 
of the parameter space that contains the minimizer
of the KL-divergence of $f(\bY_{nm}|\bX_n,\theta_{\tilde k},\mathcal M_{\tilde k})$ from the true model. 
The reason for this is that $\tilde\Theta_{\tilde k}$ 
is the region that has the highest 
posterior probability, at least asymptotically, which we shall subsequently establish.
Moreover, it follows from \ctn{Chat20} that $\pi(\theta_k|\bX_n,\bY_{nm},\mathcal M_k)$ and $\pi(\theta_k|\bX_{n,-i},\bY_{nm},\mathcal M_k)$ 
are asymptotically the same for any $i\geq 1$, for any $m\geq 1$.
Hence the event (\ref{eq:inv_test1}) associated with $\tilde\Theta_{\tilde k}$ for $k=\tilde k$, 
is expected to be reliable. 

We shall also show that asymptotically the posterior probability of the best model, $\zeta=\tilde k$, tends to $1$ almost surely.
As already mentioned, here the notion the best model is with respect to minimization of the minimum KL-divergence rate from the true model. 
We shall show that for this $\tilde k$, 
the posterior probability of $H_{0\tilde k}$ goes to $1$ asymptotically, for any $\varepsilon>0$ in (\ref{eq:inv_test1}).
%Here note that any $\varepsilon>0$ need not lead to full posterior probability of $H_{0\tilde k}$ asymptotically. This is because the best model need not be the true model
%and so the posterior probability of (\ref{eq:inv_test1}) need not tend to zero for all $\varepsilon>0$. In fact, unless we consider the second inverse setup
%where $m$ observations are available corresponding to each covariate value, this does not hold good even when the best model is the true model.
That is, asymptotically, only one inverse model, namely, the best inverse model satisfying the conditions of $H_{0\tilde k}$, will be selected.

It is useful to remark here that the KL-divergence rate referred to above is completely in the forward sense, where all the $x_i$; $i\geq 1$, are assumed to be known.
Hence, the above arguments and our subsequent theoretical underpinnings show that the asymptotic theory is dominated by the forward perspective. 
In fact, any consistent prior for $\tilde x_i$ would asymptotically lead to the best forward model.    
However, the above can not be guaranteed in any non-asymptotic sense. The model $\mathcal M_{\tilde k}$ with high posterior probability of $\{\zeta=\tilde k\}$
may have low posterior probability of $T^{(k)}(\tilde \bX_n)-T^{(k)}(\bX_n) \in [\tilde\ell_{knm}-a_k-\varepsilon, \tilde u_{knm}-a_k+\varepsilon]$, which may
result in overall lower posterior probability of $H_{0\tilde k}$ compared to $H_{0k}$ for several $k\neq \tilde k$. 
In such situations, $\mathcal M_{\tilde k}$ will not be the best choice non-asymptotically.
Thus, the inverse perspective is particularly important in realistic, non-asymptotic situations.
An appropriate Bayesian multiple testing procedure is expected to yield the best possible
inference regarding inverse model selection in both asymptotic and non-asymptotic situations, which we now devise. 
%In this regard, we shall formulate a Bayesian multiple testing method and study its consistency properties and asymptotic
%theory of the error rates.

\subsection{The Bayesian multiple testing procedure} 
\label{subsec:muller}
%Let $\bX_n=\{X_1,\ldots,X_n\}$ denote the available data set. Suppose the data is modelled by the family of distributions $P_{\bX_n|\btheta}$ (which may also be non-parametric). 
%For $M>1$, let us denote by $\bTheta=\Theta_1\times\cdots\times\Theta_M$ 
%the relevant parameter space associated with $\btheta=(\theta_1,\ldots,\theta_M)$, where we allow $M$ to be infinity as well. 
%Let $\postp(\cdot)$ and $\pexp(\cdot)$ denote the posterior distribution and expectation respectively of $\btheta$ given $\bX_n$ and let $P_{\bX_n}(\cdot)$ and $E_{\bX_n} (\cdot)$ denote the marginal distribution and expectation of $\bX_n$ respectively. Let us consider the problem of testing $m$ hypotheses simultaneously corresponding to the actual parameters of interest, where
%$1<m\leq M$. In this work, however, we assume $m$ to be finite. %The case of infinite $m$ will be communicated elsewhere.

%Without loss of generality, let us consider testing the parameters associated with $\Theta_i$; $i=1,\ldots,m$,
%formalized as:
%$$ H_{0i}:\theta_i \in \Theta_{0i}  \hbox{ versus } H_{1i} : \theta_i \in \Theta_{1i},$$ 
%where $\Theta_{0i} \bigcap \Theta_{1i}=\emptyset \mbox{ and } \Theta_{0i} \bigcup \Theta_{1i} 
%= \Theta_{i},\mbox{ for $i=1,\cdots,m$}.$

\ctn{Chandra19} proposed a novel Bayesian non-marginal testing procedure for testing general dependent hypotheses. We first briefly discuss their method and then
consider a special case of their idea to be applied to inverse model selection context.

Let
\begin{align*}
d_k=&\begin{cases}
1&\text{if the $k$-th hypothesis is rejected;}\\
0&\text{otherwise;}
\end{cases}\\
r_k=&\begin{cases}
1&\text{if $H_{1k}$ is true;}\\
0&\text{if $H_{0k}$ is true.} 
\end{cases}
\end{align*}

%In many real life situations, dependent prior structure is envisaged on the parameter space based on available domain knowledge. For example in spatial statistics, Gaussian process prior is often considered. In fMRI data, Gaussian Markov random field prior is a common prior. In such cases, the additional information on the parameters are incorporated in the model through the prior distribution. Various applications in recent times in fields as diverse as spatio-temporal statistics, neurosciences, biological sciences, engineering, environmental and ecological sciences, astrostatistics, psychometrics, demography, geostatistics, 
%reliability engineering, statistical signal processing, statistical physics, finance, actuarial science, to name only a few, consider Bayesian models with dependent prior structures. The basic idea behind the new multiple testing methodology is to incorporate such information, when available, in the testing procedure to obtain improved decision rule. This principle is in accordance with the traditional Bayesian philosophy
%that when prior information is available, inference can be enhanced. 
%dependent prior is considered based on domain knowledge. Our proposal is to exploit this additional information in the multiple testing methodology to enhance inference. 

Let $G_k$ be the set of hypotheses (including hypothesis $k$) where the parameters are %a priori
dependent on the $k$-th hypothesis. In the new procedure, the decision of each hypothesis is penalized by incorrect decisions regarding other dependent parameters. Thus a compound criterion where all the decisions in $G_k$ deterministically depends upon each other. 
Define the following quantity
\begin{equation}
z_k=\begin{cases}
1&\mbox{if $H_{d_j,j}$ is true for all $j\in G_k\setminus\{k\}$;}\\
0&\mbox{otherwise.}
\end{cases}\label{eq:z}
\end{equation}
If, for any $k\in\{1,\ldots,K\}$, $G_k=\{k\}$, a singleton, then we define $z_k=1$.
The notion of true positives $(TP)$ are modified as the following
%\ctn{chandra2017} maximize the posterior expectation of the number of true positives
\begin{equation}
TP=\sum_{k=1}^Kd_kr_kz_k,
\label{eq:tp}
\end{equation}
The posterior expectation of $TP$ is maximized subject to controlling the posterior expectation of the error term
\begin{equation}
E=\sum_{k=1}^Kd_k(1-r_kz_k).
\label{eq:e}
\end{equation}
%which is actually the posterior mean of the sum of three error terms $E_1=\sum_{i=1}^md_i(1-r_i)z_i$, $E_2=\sum_{i=1}^md_i(1-r_i)(1-z_i)$ and $E_3=\sum_{i=1}^md_ir_i(1-z_i)$. For detailed discussion regarding these, see \ctn{chandra2017}.
It follows that the decision configuration can be obtained by minimizing the function
\begin{align}
	\xi(\bd)&=-\sum_{k=1}^Kd_kE(r_kz_k|\bX_n,\bY_{nm})+\lambda_{nm}\sum_{k=1}^Kd_k E\left[ (1-r_kz_k)|\bX_n,\bY_{nm}\right]\notag\\
	&= -(1+\lambda_{nm})\sum_{k=1}^Kd_k\left(w_{knm}(\bd)-\frac{\lambda_{nm}}{1+\lambda_{nm}}\right),\notag
\end{align}
with respect to all possible decision configurations of the form $\bd=\{d_1,\ldots,d_K\}$, where
$\lambda_{nm}>0$,
and
\begin{equation*}
	w_{knm}(\bd)=E(r_kz_k|\bX_n,\bY_{nm})= \pi\left(H_{1k}\cap\left\{\cap_{j\neq k,j\in G_k}H_{d_j,j}\right\}\big | \bX_n,\bY_{nm}\right)
%\label{eq:w}
\end{equation*}
is the posterior probability of the decision configuration $\{d_1,\ldots,d_{k-1},1,d_{k+1},\ldots,d_K\}$
being correct.
Letting $\beta_{nm}=\lambda_{nm}/(1+\lambda_{nm})$, one can equivalently maximize
\begin{equation}
	f_{\beta_{nm}}(\bd)=\sum_{k=1}^K d_k\left(w_{knm}(\bd)-\beta_{nm}\right)\label{eq:beta1}
\end{equation}
with respect to $\bd$ and obtain the optimal decision configuration.

\begin{definition}
	Let $\mathbb D$ be the set of all $m$-dimensional binary vectors denoting all possible decision configurations. Define $$\widehat{\bd}=\argmax_{\bd\in\mathbb{D}} f_\beta(\bd)$$ where $0<\beta<1$. Then $\widehat{\bd}$ is the \textit{optimal decision configuration} obtained as the solution of the non-marginal multiple testing method.
	\label{def:nmd}
\end{definition}

Note that in the definitions of both $TP$ and $E$, $d_i$ is penalized by incorrect decisions in the same group. 
This forces the decisions to be jointly taken also adjudging other dependent parameters. 

\subsection{Specialization of the general multiple testing procedure to inverse model selection problems}
\label{subsec:specialization}
In our inverse model selection problem note that since the models $\mathcal M_k$; $k=1,\ldots,K$, are independent, so are $\tilde \bX_n$ associated with the different models. 
Thus, the hypotheses are dependent only through the relation $\sum_{k=1}^KI(\zeta=k)=1$. As we shall show, the posterior probability of the event
$\{\zeta=\tilde k\}$ converges to one {\it a posteriori} as the sample size tends
to infinity, irrespective of any other dependence among $(I(\zeta=1),\ldots,I(\zeta=K))$ induced through $(p_1,\ldots,p_K)$. 
Hence, there is not enough reason to consider the hypotheses as dependent.
%and construct the groups $G_k$ accordingly consisting of the hypotheses dependent upon the $k$-th hypothesis. 
Thus, for our purpose, we simply set $G_k=\{k\}$.
Consequently, (\ref{eq:beta1}) in our case reduces to 
\begin{equation}
	f_{\beta_{nm}}(\bd)=\sum_{k=1}^K d_k\left(v_{knm}-\beta_{nm}\right),\label{eq:beta2}
\end{equation}
where
\begin{equation*}
	v_{knm}=E(r_k|\bX_n,\bY_{nm})= \pi\left(H_{1k}|\bX_n,\bY_{nm}\right).
%\label{eq:w}
\end{equation*}
In this case, the optimal decision configuration $\widehat\bd$ is given by the following: for $k=1,\ldots,K$,
\begin{equation}
\widehat d_k=\begin{cases}
	1&\text{if $v_{knm}>\beta_{nm}$;}\\
0&\text{otherwise.} 
\end{cases}
	\label{eq:optimal_decision}
\end{equation}
Hence, although our formulation of the multiple hypothesis test for inverse model selection is novel, the Bayesian procedure for testing parallels
that of \ctn{muller04} (see also \ctn{Guindani09}), which is a special case of the general procedure proposed in \ctn{Chandra19}.

\subsection{Error measures in multiple testing}
\label{subsec:Bayesian_errors}

\ctn{storey03} advocated \textit{positive False Discovery Rate} $(pFDR)$ as a measure of Type-I error in multiple testing. 
Let $\delta(\bd|\bX_n,\bY_{nm})$ be the probability of choosing $\bd$ as the optimal decision configuration given data 
$(\bX_n,\bY_{nm})$ when a given multiple testing method is employed. Then $pFDR$ is defined as:
\begin{equation}
	pFDR_{nm}=E_{\bY_{nm}|\bX_n} \left[ \sum_{\bd\in\mathbb{D}}  
	\frac{\sum_{k=1}^{K}d_k(1-r_k)}{\sum_{k=1}^{K}d_i}\delta(\bd|\bX_n,\bY_{nm})\bigg{|}\delta(\bd=\mathbf{0}|\bX_n,\bY_{nm})=0 \right].
\label{eq:pfdr}
\end{equation}

Analogous to Type-II error, the \textit{positive False Non-discovery Rate} $(pFNR)$ is defined as
\begin{align}
	pFNR_{nm}= E_{\bY_{nm}|\bX_n}\left[\sum_{\bd\in\mathbb D} \frac{\sum_{k=1}^K(1-d_k)r_k} {\sum_{k=1}^K(1-d_k)} \delta\left(\bd|\bX_n,\bY_{nm}\right)
	\bigg | \delta\left(\bd=\bone|\bX_n,\bY_{nm}\right)=0\right].
\label{eq:pFNR}
\end{align}

Under prior $\pi(\cdot)$, \ctn{SanatGhosh08} defined posterior $FDR$ and $FNR$. The measures are given as following:
\begin{align}
	posterior~FDR_{nm}
	&= E\left[\sum_{\bd\in\mathbb D}\frac{\sum_{k=1}^Kd_k(1-r_k)}{\sum_{k=1}^Kd_k \vee 1}\delta\left(\bd|\bX_n,\bY_{nm}\right)\bigg |\bX_n,\bY_{nm} \right]\\
	&= \sum_{\bd\in\mathbb{D}} 
	\frac{\sum_{k=1}^{K}d_k(1-v_{knm})}{\sum_{k=1}^{K}d_k \vee 1}\delta(\bd|\bX_n,\bY_{nm}); \label{eq:pBFDR}\\
	posterior~FNR_{nm}
	&= E\left[\sum_{\bd\in\mathbb D} \frac{\sum_{k=1}^K(1-d_k)r_k} {\sum_{k=1}^K(1-d_k)\vee 1} \delta\left(\bd|\bX_n,\bY_{nm}\right)\bigg |\bX_n,\bY_{nm}\right]\\
	&=\sum_{\bd\in\mathbb D}
	\frac{\sum_{k=1}^{K}(1-d_k)v_{knm}}{\sum_{k=1}^{K}(1-d_k)\vee 1}\delta(\bd|\bX_n,\bY_{nm}).\label{eq:pBFNR}
\end{align}
Also under any non-randomized decision rule, $\delta(\bd|\bX_n,\bY_{nm})$ is either 1 or 0 depending on data $(\bX_n,\bY_{nm})$. 
Given $(\bX_n,\bY_{nm})$, we denote these error measures conditional on the data by conditional $FDR$ ($cFDR_{nm}$) and conditional $FNR$ ($cFNR_{nm}$) respectively.

The positive Bayesian $FDR$ ($pBFDR_{nm}$) and $FNR$ ($pBFNR_{nm}$) are the expectations of $cFDR_{nm}$ and $cFNR_{nm}$ respectively,
with respect to the distribution of $\bY_{nm}$ given $\bX_n$.

For our Bayesian purpose, we shall consider the Bayesian measures $cFDR_{nm}$, $pBFDR_{nm}$, $cFNR_{nm}$ and $pBFNR_{nm}$, 
and investigate their asymptotic properties. \ctn{Chandra19} and \ctn{Chandra20} particularly recommend $cFDR_{nm}$ and $cFNR_{nm}$,
since they are conditioned on the observed data $(\bX_n,\bY_{nm})$ and hence qualify as {\it bona fide} Bayesian measures.

Let us now proceed towards development of the asymptotic theory for our proposed multiple testing strategy. The issue of misspecification will play a crucial role
in this context.
Suppose that the true data-generating parameter $\theta_0$ is not contained in $\Theta$, the parameter space considered.
This is a case of misspecification that we must incorporate in our asymptotic theory. Indeed, we shall build a general asymptotic framework that allows for 
possibly infinite-dimensional parameters, dependent data as well as misspecification. In this regard, the approach presented in \ctn{Shalizi09} seems to be very appropriate.
Before proceeding further, we first provide a brief overview of this approach, which we conveniently exploit for our purpose.

\section{A brief overview of Shalizi's approach to posterior convergence}
\label{sec:shalizi_briefing}
Let $\bY_n=\{Y_1,\ldots,Y_n\}$, and let $f_{\theta}(\bY_n)$ and $f_{\theta_0}(\bY_n)$ denote the observed and the true likelihoods respectively, under the given value of the parameter $\theta$
and the true parameter $\theta_0$. We assume that $\theta\in\Theta$, where $\Theta$ is the (often infinite-dimensional) parameter space. However, we {\it do not} 
assume that $\theta_0\in\Theta$, thus allowing misspecification.
%note that
%\begin{align}
%f_{\theta}(\bY_n)&=\frac{1}{\left(\sigma\sqrt{2\pi}\right)^n}\exp\left\{-\frac{1}{2\sigma^2}\sum_{i=1}^n(Y_i-\eta(\bx_i))^2\right\};\label{eq:like1}\\
%f_{\theta_0}(\bY_n)&=\frac{1}{\left(\sigma_0\sqrt{2\pi}\right)^n}\exp\left\{-\frac{1}{2\sigma^2_0}\sum_{i=1}^n(Y_i-\eta_0(\bx_i))^2\right\}.\label{eq:true_like1}
%\end{align}
The key ingredient associated with Shalizi's approach to proving convergence of the posterior distribution of $\theta$ is to show that the 
asymptotic equipartition property holds.
To elucidate, let us consider the following likelihood ratio:
\begin{equation*}
R_n(\theta)=\frac{f_{\theta}(\bY_n)}{f_{\theta_0}(\bY_n)}.
%\label{eq:R_n}
\end{equation*}
Then, to say that for each $\theta\in\Theta$, the generalized or relative asymptotic equipartition property holds, we mean
\begin{equation}
\underset{n\rightarrow\infty}{\lim}~\frac{1}{n}\log R_n(\theta)=-h(\theta),
\label{eq:equipartition}
\end{equation}
almost surely, where
$h(\theta)$ is the KL-divergence rate given by
\begin{equation}
h(\theta)=\underset{n\rightarrow\infty}{\lim}~\frac{1}{n}E_{\theta_0}\left(\log\frac{f_{\theta_0}(\bY_n)}{f_{\theta}(\bY_n)}\right),
\label{eq:S3}
\end{equation}
provided that it exists (possibly being infinite), where $E_{\theta_0}$ denotes expectation with respect to the true model.
Let
\begin{align}
h\left(A\right)&=\underset{\theta\in A}{\mbox{ess~inf}}~h(\theta);\label{eq:h2}\\
J(\theta)&=h(\theta)-h(\Theta);\label{eq:J}\\
J(A)&=\underset{\theta\in A}{\mbox{ess~inf}}~J(\theta).\label{eq:J2}
\end{align}
Thus, $h(A)$ can be roughly interpreted as the minimum KL-divergence between the postulated and the true model over the set $A$. If $h(\Theta)>0$, this indicates
model misspecification. 
%However, as we shall show, model misspecification need not always imply that $h(\Theta)>0$. 
For $A\subset\Theta$, $h(A)>h(\Theta)$, so that $J(A)>0$.

As regards the prior, it is required to construct an appropriate sequence of sieves $\mathcal G_n$ such that $\mathcal G_n\rightarrow\Theta$ and $\pi(\mathcal G^c_n)\leq\alpha\exp(-\beta n)$,
for some $\alpha>0$. 

With the above notions, verification of (\ref{eq:equipartition}) along with several other technical conditions ensure that for any $A\subseteq\Theta$ such that $\pi(A)>0$, 
\begin{equation}
\underset{n\rightarrow\infty}{\lim}~\pi(A|\bY_n)=0,
\label{eq:post_conv1}
\end{equation}
almost surely, provided that $h(A)>h(\Theta)$.

The seven assumptions of Shalizi leading to the above result, which we denote as (S1)--(S7), are provided in Appendix \ref{subsec:assumptions_shalizi}.
In what follows, we denote almost sure and in probability convergence by ``$\stackrel{a.s.}{\longrightarrow}$" and ``$\stackrel{P}{\longrightarrow}$", respectively, 
almost sure equality by ``$\stackrel{a.s.}{=}$" and weak convergence
by ``$\stackrel{w}{\longrightarrow}$".

\section{Asymptotic properties of the posterior probabilities of the alternative hypotheses}
\label{sec:asymp_v}

\subsection{Posterior convergence to the best model}
\label{subsec:conv_best_model}
\begin{theorem}
	\label{theorem:bf1}
	Assume that for $k=1,\ldots,K$, $\mathcal M_k$ satisfies conditions (S1)--(S6) of Shalizi, and that the competing models as well as the true model
	have densities with respect to some common $\sigma$-finite measure.
	Also assume that the posterior associated with $\mathcal M_k$ is dominated by the prior, which is again absolutely continuous
	with respect to some appropriate $\sigma$-finite measure, and that the priors satisfy $\pi(\theta_k|\mathcal M_k)>0$ for all
	$\theta_k\in\Theta_k$. Let $h_{\tilde k}\left(\Theta_{\tilde k}\right)=\min\{h_k\left(\Theta_k\right):k=1,\ldots,K\}$. 
	Then for any $m\geq 1$,  
\begin{equation}
	\underset{n\rightarrow\infty}{\lim}~\pi(\zeta=k|\bX_n,\bY_{nm})\stackrel{a.s.}{=}\begin{cases} 1 & \mbox{if}~k=\tilde k\\
		0 & \mbox{if}~k\neq\tilde k.
	\end{cases}
	\label{eq:post_p}
\end{equation}
\end{theorem}
\begin{proof}
	For any $k_1,k_2\in\{1,\ldots,K\}$, let $BF^{(nm)}(\mathcal M_{k_1},\mathcal M_{k_2})$ denote the Bayes factor of model $\mathcal M_{k_1}$ against
	model $\mathcal M_{k_2}$. Then as a direct consequence of Theorem 2 of \ctn{Chatterjee18}, the following holds for any $m\geq 1$:
	\begin{equation}
		\frac{1}{n}\log BF^{(nm)}(\mathcal M_k,\mathcal M_0)\rightarrow -h_k\left(\Theta_k\right),~\mbox{as}~n\rightarrow\infty,
		\label{eq:bf1}
	\end{equation}
	almost surely with respect to the true model $\mathcal M_0$. In the above, $h_k\left(\Theta_k\right)$ corresponds to (\ref{eq:equipartition}), 
	(\ref{eq:S3}) and (\ref{eq:h2}) for model $\mathcal M_k$ with parameter space $\Theta_k$.

Now, since $h_{\tilde k}\left(\Theta_{\tilde k}\right)=\min\{h_k\left(\Theta_k\right):k=1,\ldots,K\}$, it follows from (\ref{eq:bf1}) that as $n\rightarrow\infty$,
for any $m\geq 1$,
\begin{equation*}
	\frac{1}{n}\log BF^{(nm)}(\mathcal M_k,\mathcal M_{\tilde k})\rightarrow -\left[h_k\left(\Theta_k\right)-h_{\tilde k}\left(\Theta_{\tilde k}\right)\right],
\end{equation*}
so that as $n\rightarrow\infty$, for any $m\geq 1$,
\begin{equation}
	BF^{(nm)}(\mathcal M_k,\mathcal M_{\tilde k})=\begin{cases} 1 & \mbox{if}~k=\tilde k\\
		\stackrel{a.s.}{\longrightarrow} 0, & \mbox{if}~k\neq\tilde k.
	\end{cases}
	\label{eq:bf2}
\end{equation}
Now note that (see, for example, \ctn{Liang08})
\begin{equation}
	\pi(\zeta=k|\bX_n,\bY_{nm},p_1,\ldots,p_K)
	=\frac{p_kBF^{(nm)}(\mathcal M_k,\mathcal M_{\tilde k})}{\sum_{\ell=1}^Kp_\ell BF^{(nm)}(\mathcal M_{\ell},\mathcal M_{\tilde k})}.
	\label{eq:post_zeta}
\end{equation}
	Hence it follows by applying (\ref{eq:bf2}) to (\ref{eq:post_zeta}) that the following holds:
\begin{equation}
	\underset{n\rightarrow\infty}{\lim}~\pi(\zeta=k|\bX_n,\bY_{nm},p_1,\ldots,p_K)\stackrel{a.s.}{=}\begin{cases} 1 & \mbox{if}~k=\tilde k\\
		0 & \mbox{if}~k\neq\tilde k.
	\end{cases}
	\label{eq:post_zeta2}
\end{equation}
	Now note that $\pi(\zeta=k|\bX_n,\bY_{nm})=E\left[\pi(\zeta=k|\bX_n,\bY_{nm},p_1,\ldots,p_K)\right]$, the expectation being over the posterior
	distribution of $(p_1,\ldots,p_K)$ given $\bX_n$ and $\bY_{nm}$. Since $\pi(\zeta=k|\bX_n,\bY_{nm},p_1,\ldots,p_K)\leq 1$ almost surely, it follows
	by uniform integrability and (\ref{eq:post_zeta2}), that 
\begin{equation*}
	\underset{n\rightarrow\infty}{\lim}~\pi(\zeta=k|\bX_n,\bY_{nm})=E\left[\pi(\zeta=k|\bX_n,\bY_{nm},p_1,\ldots,p_K)\right]
	\stackrel{a.s.}{=}\begin{cases} 1 & \mbox{if}~k=\tilde k\\
		0 & \mbox{if}~k\neq\tilde k.
	\end{cases}
%	\label{eq:post_zeta3}
\end{equation*}
\end{proof}

\subsection{Convergence of the cross-validation posteriors of $\tilde x_i$}
\label{subsec:conv_cv_posteriors}

\begin{theorem}
	\label{theorem:inv_cons}
	For model $\mathcal M_{k}$ assume conditions (S1)--(S7) of Shalizi, and let the infimum of 
	$h_{k}(\theta_{k})$ over $\Theta_{\tilde k}$ be attained at $\tilde\theta_{k}\in\tilde\Theta_{k}$, 
	where $\tilde\theta_{k}\neq\theta_0$.
	Also assume that $\Theta_{k}$ and $\Theta_0$ are complete separable metric spaces.
	%for $i\geq 1$, $f(y_i|\theta,x_i,\bY^{(i-1)},\mathcal M_k)$ and $f(y_i|\theta,x_i,\bY^{(i-1)},\mathcal M_0)$ are bounded and continuous in $\theta$. 
	Then, with the prior (\ref{eq:prior_x}), under further assumptions that $\pi(\tilde x_i|\theta_{k},\by_i,\mathcal M_k)$ is contiuous in $\theta_{k}$,
	$f(\by_i|\tilde\theta_{k},\tilde x_i,\mathcal M_k)$ is continuous in $\tilde x_i$, for $i\geq 1$ and $\tilde\eta_k$ is a one-to-one function, the following holds:
	%Assumptions \ref{as:as1} -- \ref{as:as3}, for $i\geq 1$, 
\begin{equation}
\underset{m\rightarrow\infty}{\lim}\underset{n\rightarrow\infty}{\lim}~\pi(\tilde x_i\in V^c_{ik}|\bX_{n,-i},\bY_{nm},\mathcal M_k,\tilde\Theta_k)=0,~\mbox{almost surely},
\end{equation}
for any neighborhood $V_{ik}$ of $x^*_{ik}$.
\end{theorem}
\begin{proof}
By the hypotheses, (\ref{eq:post_conv1}) holds, from which it follows that for any $\epsilon>0$, and for any $m\geq 1$,
\begin{equation}
	\underset{n\rightarrow\infty}{\lim}~\pi(\mathbb N^c_{k,\epsilon}|\bX_{n,-i},\bY_{nm},\mathcal M_{k})=0, 
\label{eq:cons1}
\end{equation}
	where $\mathbb N_{k,\epsilon}=\left\{\theta_{k}:h_{k}(\theta_{k})\leq h_{k}\left(\Theta_{k}\right)+\epsilon\right\}$.

	Now, by hypothesis, the infimum of $h_{k}(\theta_{k})$ over $\Theta_{k}$ is attained at $\tilde\theta_{k}\in\Theta_{k}$, 
	where $\tilde\theta_{k}\neq\theta_0$. 
	Then by (\ref{eq:cons1}), the posterior of $\theta_{k}$ given $\bX_{n,-i}$ and $\bY_{nm}$, %given by (\ref{eq:post_forward2}),
	concentrates around $\tilde\theta_{k}$, the minimizer of the limiting KL-divergence rate from the true distribution. 
	Formally, given any neighborhood $U_k$ of $\tilde\theta_{k}$, the set $\mathbb N_{k,\epsilon}$ is contained in $U_k$ for sufficiently small $\epsilon$.
	It follows that for any neighborhood $U_k$ of $\tilde\theta_{k}$, $\pi(U_k|\bX_{n,-i},\bY_{nm},\mathcal M_{k})\rightarrow 1$, 
	almost surely, as $n\rightarrow\infty$.
	Since $\Theta_{k}$ is a complete, separable metric space, it follows that (see, for example, \ctn{Ghosh03}, \ctn{Ghosal17})
	\begin{equation}	
		\pi(\cdot|\bX_{n,-i},\bY_{nm},\mathcal M_{k})\stackrel{w}{\longrightarrow} \delta_{\tilde\theta_{k}}(\cdot),
		~\mbox{almost surely, as}~n\rightarrow\infty,~\mbox{for any}~m\geq 1.
		\label{eq:weak1}
	\end{equation}
	In the above, $\delta_{\tilde\theta_{k}}(\cdot)$ denotes point mass at $\tilde\theta_{k}$.
%where $``\stackrel{w}{\longrightarrow}"$ denotes weak convergence.
%	Then, due to (\ref{eq:weak1}) and the Portmanteau theorem, as $f(y_i|\theta,x_i,\bY^{(i-1)},\mathcal M)$ is bounded and continuous in $\theta$, it holds 
%	using (\ref{eq:post_forward1}), that
%	\begin{equation}
%		\pi(y_i|\bY_{n,-i},\bX_n,\mathcal M)\stackrel{a.s.}{\longrightarrow} f(y_i|\tilde\theta,x_i,\bY^{(i-1)},\mathcal M),~\mbox{as}~n\rightarrow\infty.
%		\label{eq:forward_cons1}
%	\end{equation}

	Now since $\tilde\Theta^c_{k}\subset\Theta_{k}$, $h_{k}\left(\tilde\Theta^c_{k}\right)>h_{k}\left(\Theta_{k}\right)$.
	Hence, from (\ref{eq:post_conv1}) it follows that for any $m\geq 1$, 
	\begin{equation}
		\pi\left(\theta_{k}\in\tilde\Theta^c_{k}|\bX_n,\bY_{nm},\mathcal M_k\right)\stackrel{a.s.}{\longrightarrow} 0,~\mbox{as}~n\rightarrow\infty.
		\label{eq:shalizi_conv1}
	\end{equation}
	Also note that since $\pi(\tilde x_i|\theta_{k},\by_i,\mathcal M_k)$ is continuous in $\theta_{k}$ by assumption, it follows by Scheffe's theorem
	that any probability associated with $\pi(\tilde x_i|\theta_{k},\by_i,\mathcal M_k)$ is continuous in $\theta_{k}$ (see Lemma 4.3 of \ctn{Chat20}). 
	Hence, for any neighborhood $V_{ik}$ of $x^*_{ik}$, the probability $\pi(\tilde x_i\in V^c_{ik}|\theta_{k},\by_i,\mathcal M_k)$ is continuous in $\theta_{k}$. 
%due to Lemma \ref{lemma:postprob_continuity}.
Moreover, since it is a probability, it is bounded. Hence, by the Portmanteau theorem, weak convergence of 
	$\pi\left(\theta_{k}|\bX_{n,-i},\bY_{nm},\mathcal M_{k}\right)$, and (\ref{eq:shalizi_conv1}) it holds almost surely that
\begin{align}
	\pi(\tilde x_i\in V^c_{ik}|\bX_{n,-i},\bY_{nm},\mathcal M_{k})
	&=\frac{\int_{\tilde\Theta_{k}}\pi(\tilde x_i\in V^c_i|\theta_{k},\by_i,\mathcal M_k)d\pi(\theta_{k}|\bX_{n,-i},\bY_{nm},\mathcal M_k)}
	{\pi\left(\tilde\Theta_{k}|\bX_{n,-i},\bY_{nm},\mathcal M_k\right)}\notag\\
	&\stackrel{a.s.}{\longrightarrow}
	\pi(\tilde x_i\in V^c_{ik}|\tilde\theta_{k},\by_i,\mathcal M_k),~\mbox{as}~n\rightarrow\infty,~\mbox{for any}~m\geq 1.\notag
%\label{eq:loo_cons1}
\end{align}
	That $\pi(\tilde x_i\in V^c_{ik}|\tilde\theta_{k},\by_i,\mathcal M_k)\stackrel{a.s.}{\longrightarrow}0$, as $m\rightarrow\infty$, follows in the same way
	as the proof of Theorem 2 of \ctn{Chat20} by replacing $\theta_0$ with $\tilde\theta_{k}$.
\end{proof}

\subsection{Posterior convergence of the discrepancy measures}
\label{subsec:conv_discrepancy}

\begin{theorem}
\label{theorem:consistency_T1}
%Assume conditions (S1)--(S7) of Shalizi, and the prior (\ref{eq:prior_x}). %Also let Assumptions \ref{as:as1} -- \ref{as:as3} hold, for $i\geq 1$,
%	Then, %with the prior (\ref{eq:prior_x}), 
%	under further assumptions that $\pi(\tilde x_i|\theta_{k},\by_i,\mathcal M_k)$ is contiuous in $\theta_{k}$,
%	$f(\by_i|\tilde\theta_{k},\tilde x_i,\mathcal M_k)$ is continuous in $\tilde x_i$, for $i\geq 1$ and $\tilde\eta_k$ is a one-to-one function, the following holds
%	for any $\varepsilon>0$:
%Define for some $\varepsilon>0$, the following: 
%$$T_1(\tilde\bX_n) =\frac{1}{n} \sum_{i=1}^n\frac{(\tilde x_i - E_{\pi}(\tilde x_i))^2}{V_{\pi}(\tilde x_i)+\varepsilon}$$
%and $$T_1(\bX_n) =\frac{1}{n} \sum_{i=1}^n\frac{(x_i - E_{\pi}(\tilde x_i))^2}{V_{\pi}(\tilde x_i)+\varepsilon}.$$
%Then 
	Under the conditions of Theorem \ref{theorem:inv_cons}, the following holds for any $\varepsilon>0$:
\begin{equation}	
	%\pi\left(\left|T^{(k)}(\tilde\bX_{n})-T^{(k)}(\bX^*_{nk})\right|>\varepsilon|\bX_n,\bY_{mn},\mathcal M_k,\tilde\Theta_k\right)\stackrel{a.s}{\longrightarrow}0,~\mbox{as}~m\rightarrow\infty,~n\rightarrow\infty,
	\pi\left(T^{(k)}(\tilde\bX_{n})>\varepsilon|\bX_n,\bY_{mn},\mathcal M_k,\tilde\Theta_k\right)\stackrel{a.s}{\longrightarrow}0,~\mbox{as}~m\rightarrow\infty,~n\rightarrow\infty,
	\label{eq:consistency_T1}
\end{equation}
	where $T^{(k)}=T^{(1)}_k$ or $T^{(2)}_k$.
%In the above, $``\stackrel{P}{\longrightarrow}"$ denotes convergence in probability.
\end{theorem}
\begin{proof}
	For $i\geq 1$, Theorem \ref{theorem:inv_cons} implies almost sure weak convergence of the $i$-th cross-validation posterior of $\tilde x_i$ 
	for model $\mathcal M_k$ to $\delta_{x^*_{ik}}$,
as $m\rightarrow\infty$ and $n\rightarrow\infty$. This is equivalent to convergence in (cross-validation posterior) distribution of $\tilde x_i$ 
	to the degenerate quantity $x^*_{ik}$, almost surely. Degeneracy guarantees that this is equivalent to convergence in probability, almost surely.
	In other words, with respect to the cross-validation posterior distribution of $\tilde x_i$ for model $\mathcal M_k$, 
	almost surely, as $m\rightarrow\infty$, $n\rightarrow\infty$,
	\begin{equation}
		\tilde x_i\stackrel{P}{\longrightarrow}x^*_{ik}.
		\label{eq:P1}
	\end{equation}
	Now note that $T^{(k)}(\tilde\bX_n)$ is an average of $n$ terms, the $i$-th term being 
	$\frac{\left|\tilde x_i-E\left(\tilde x_i|\bX_n,\bY_{mn},\mathcal M_k,\tilde\Theta_k\right)\right|}
	{\sqrt{Var\left(\tilde x_i|\bX_n,\bY_{mn},\mathcal M_k,\tilde\Theta_k\right)+c}}$ or its square.
	Since $\tilde x_i\in\mathcal X$ for $i\geq 1$ and $\mathcal X$ is compact, (\ref{eq:P1}) and uniform integrability entails that
	\begin{align}
		&\underset{m\rightarrow\infty}{\lim}\underset{n\rightarrow\infty}{\lim}~	
		E\left(\tilde x_i|\bX_n,\bY_{mn},\mathcal M_k,\tilde\Theta_k\right)\stackrel{a.s.}{=}x^*_{ik};\label{eq:E_P1}\\
		&\underset{m\rightarrow\infty}{\lim}\underset{n\rightarrow\infty}{\lim}~	
		Var\left(\tilde x_i|\bX_n,\bY_{mn},\mathcal M_k,\tilde\Theta_k\right)\stackrel{a.s.}{=}0.\label{eq:V_P1}
	\end{align}
	It follows from (\ref{eq:E_P1}) and (\ref{eq:V_P1}) that with respect to the cross-validation posterior distribution of $\tilde x_i$ for model $\mathcal M_k$,
	almost surely, as $m\rightarrow\infty$, $n\rightarrow\infty$,
	\begin{equation}
	\frac{\left|\tilde x_i-E\left(\tilde x_i|\bX_n,\bY_{mn},\mathcal M_k,\tilde\Theta_k\right)\right|}
	{\sqrt{Var\left(\tilde x_i|\bX_n,\bY_{mn},\mathcal M_k,\tilde\Theta_k\right)+c}}
		\stackrel{P}{\longrightarrow}0,~\mbox{for all}~i\geq 1.
		\label{eq:P2}
	\end{equation}
	Hence, by Theorem 7.15 of \ctn{Schervish95} (page 398), it follows that
	with respect to the cross-validation posterior distributions of $\{\tilde x_i; i\geq 1\}$, 
	for model $\mathcal M_k$, almost surely, as $m\rightarrow\infty$, $n\rightarrow\infty$,
	$$T^{(k)}(\tilde\bX_n)\stackrel{P}{\longrightarrow}0,$$
	which is equivalent to (\ref{eq:consistency_T1}).
\end{proof}

\begin{theorem}
	\label{theorem:disc_conv}
	Assume the conditions of Theorem \ref{theorem:consistency_T1}. Also assume that %$\mathcal X$ is compact and that 
	for $i\geq 1$, $x^*_{ik}$ is a 
	continuous function of $\left\{x_1,x_2,\ldots,x_{i-1},x_i,x_{i+1},\ldots,x_{i+\ell}\right\}$, for some non-negative integer $\ell$. 
	Then there exist positive constants $a_{1k}$ and $a_{2k}$
	such that
\begin{align}
%\underset{m\rightarrow\infty}{\lim}\underset{n\rightarrow\infty}{\lim}	
%	~\frac{\left|T^{(k)}_1(\bX^*_{nk})-T^{(k)}_1(\bX_n)\right|}{\sqrt{Var\left(T^{(k)}_1(\tilde \bX_n)|\bX_n,\bY_{nm},\mathcal M_k,\tilde\Theta_k\right)+c}}
%	&=a_{1k};\label{eq:T1_bound}\\
%\underset{m\rightarrow\infty}{\lim}\underset{n\rightarrow\infty}{\lim}	
%	~\frac{\left|T^{(k)}_2(\bX^*_{nk})-T^{(k)}_2(\bX_n)\right|}{\sqrt{Var\left(T^{(k)}_2(\tilde \bX_n)|\bX_n,\bY_{nm},\mathcal M_k,\tilde\Theta_k\right)+c}}
%	&=a_{2k}.\label{eq:T2_bound}
	\underset{m\rightarrow\infty}{\lim}\underset{n\rightarrow\infty}{\lim} ~T^{(k)}_1(\bX_n)&=a_{1k};\label{eq:T1_bound}\\
	\underset{m\rightarrow\infty}{\lim}\underset{n\rightarrow\infty}{\lim} ~T^{(k)}_2(\bX_n)&=a_{2k}.\label{eq:T2_bound}
\end{align}
\end{theorem}
\begin{proof}
	%Using the same ideas as in the proof of Theorem 4 of \ctn{Chat20} it is seen that as $m\rightarrow\infty$ and $n\rightarrow\infty$,
	%\begin{align}
	%	&E\left(\tilde x_i|\bX_n,\bY_{mn},\mathcal M_k,\tilde\Theta_k\right)\stackrel{a.s}{\longrightarrow}x^*_{ik};\label{eq:E1}\\
	%	&Var\left(\tilde x_i|\bX_n,\bY_{mn},\mathcal M_k,\tilde\Theta_k\right)\stackrel{a.s}{\longrightarrow}0.\label{eq:V1}
	%\end{align}
	It follows from %(\ref{eq:E1}) and (\ref{eq:V1}) that as $m\rightarrow\infty$ and $n\rightarrow\infty$,
	(\ref{eq:E_P1}) and (\ref{eq:V_P1}) that
	\begin{align}
		%&T^{(k)}_1(\bX^*_{nk})\stackrel{a.s}{\longrightarrow}0;\label{eq:T1_1}\\
		%&T^{(k)}_2(\bX^*_{nk})\stackrel{a.s}{\longrightarrow}0;\label{eq:T2_1}\\
		&T^{(k)}_1(\bX_n)\stackrel{a.s}{\longrightarrow}\underset{n\rightarrow\infty}{\lim} ~\frac{1}{n\sqrt{c}}\sum_{i=1}^n\left|x_i-x^*_{ik}\right|;\label{eq:T1_2}\\
		&T^{(k)}_2(\bX_n)\stackrel{a.s}{\longrightarrow}\underset{n\rightarrow\infty}{\lim} ~\frac{1}{nc}\sum_{i=1}^n\left(x_i-x^*_{ik}\right)^2.\label{eq:T2_2}
	\end{align}
	Now, by our assumption, $x^*_{ik}$ is a continuous function of $\left\{x_1,x_2,\ldots,x_{i-1},x_i,x_{i+1},\ldots,x_{i+\ell}\right\}$, 
	for some non-negative integer $\ell$.
	Hence, letting $u_{ik}=x_i-x^*_{ik}$, it follows by Riemann sum convergence that 
	\begin{align}
		&\underset{n\rightarrow\infty}{\lim} ~\frac{1}{n\sqrt{c}}\sum_{i=1}^n\left|x_i-x^*_{ik}\right|=c^{-\frac{1}{2}}
		|\tilde{\mathcal X_k}|^{-1}\int_{\tilde{\mathcal X_k}}|u|du;\label{eq:R1}\\
		&\underset{n\rightarrow\infty}{\lim} ~\frac{1}{nc}\sum_{i=1}^n\left(x_i-x^*_{ik}\right)^2
		=c^{-1}|\tilde{\mathcal X_k}|^{-1}\int_{\tilde{\mathcal X_k}}u^2du,\label{eq:R2}
	\end{align}
	where $\tilde{\mathcal X_k}$ is the appropriate compact co-domain of $u_{ik}$ induced by the transformation $u_{ik}=x_i-x^*_{ik}$ and the original compact covariate
	space $\mathcal X$, and $|\tilde{\mathcal X_k}|$ stands for the Lebesgue measure of $\tilde{\mathcal X_k}$.

	Since the right hand sides of (\ref{eq:R1}) and (\ref{eq:R2}) are well-defined positive quantities, the proof follows by combining (\ref{eq:T1_2}) -- (\ref{eq:R2}).
\end{proof}

\begin{theorem}
\label{theorem:consistency_T}
%Assume conditions (S1)--(S7) of Shalizi, and the prior (\ref{eq:prior_x}). %Also let Assumptions \ref{as:as1} -- \ref{as:as3} hold, for $i\geq 1$,
%	Then, %with the prior (\ref{eq:prior_x}), 
%	under further assumptions that $\pi(\tilde x_i|\theta_{k},\by_i,\mathcal M_k)$ is contiuous in $\theta_{k}$,
%	$f(\by_i|\tilde\theta_{k},\tilde x_i,\mathcal M_k)$ is continuous in $\tilde x_i$, for $i\geq 1$ and $\tilde\eta_k$ is a one-to-one function, i
	Assume the conditions of Theorem \ref{theorem:disc_conv}.
	Then the following holds
	for any $\varepsilon>0$, where $T^{(k)}=T^{(k)}_1$ or $T^{(k)}_2$ and respectively, $a_k=a_{1k}$ or $a_{2k}$:
%Define for some $\varepsilon>0$, the following: 
%$$T_1(\tilde\bX_n) =\frac{1}{n} \sum_{i=1}^n\frac{(\tilde x_i - E_{\pi}(\tilde x_i))^2}{V_{\pi}(\tilde x_i)+\varepsilon}$$
%and $$T_1(\bX_n) =\frac{1}{n} \sum_{i=1}^n\frac{(x_i - E_{\pi}(\tilde x_i))^2}{V_{\pi}(\tilde x_i)+\varepsilon}.$$
%Then 
%	almost surely with respect to the posterior distribution given $\bX_n,\bY_{nm},\mathcal M_k,\tilde\Theta_k$,
\begin{equation}	
%	\frac{\left|T^{(k)}(\tilde \bX_n)-T^{(k)}(\bX_n)\right|}{\sqrt{Var\left(T(\tilde \bX_n)|\bX_n,\bY_{nm},\mathcal M_k,\tilde\Theta_k\right)+c}}
%	\stackrel{P}{\longrightarrow}a_k,~\mbox{as}~m\rightarrow\infty,~n\rightarrow\infty.
%
%	\pi\left(\frac{\left|T^{(k)}(\tilde \bX_n)-T^{(k)}(\bX_n)\right|}{\sqrt{Var\left(T(\tilde \bX_n)|\bX_n,\bY_{nm},\mathcal M_k,\tilde\Theta_k\right)+c}}>a_k+\varepsilon
%	\bigg |\bX_n,\bY_{nm},\mathcal M_{k},\tilde\Theta_k\right)
%	\stackrel{a.s.}{\longrightarrow} 0,~\mbox{as}~m\rightarrow\infty,~n\rightarrow\infty.
	\underset{m\rightarrow\infty}{\lim}\underset{n\rightarrow\infty}{\lim}~
	\pi\left(T^{(k)}(\tilde \bX_n)-T^{(k)}(\bX_n)\in [\tilde\ell_{knm}-a_k-\varepsilon,\tilde u_{knm}-a_k+\varepsilon]^c
	\bigg |\bX_n,\bY_{nm},\mathcal M_k,\tilde\Theta_k\right)
	\stackrel{a.s.}{=}0.
	\label{eq:consistency_T}
\end{equation}
\end{theorem}
\begin{proof}
	First, observe that since for $i=1,\ldots,n$, $\tilde x_i\in \mathcal X$ almost surely, where $\mathcal X$ is compact, 
	$\frac{\left|\tilde x_i-E\left(\tilde x_i|\bX_n,\bY_{mn},\mathcal M_k,\tilde\Theta_k\right)\right|}
	{\sqrt{Var\left(\tilde x_i|\bX_n,\bY_{mn},\mathcal M_k,\tilde\Theta_k\right)+c}}$ are almost surely uniformly bounded.
	Hence, $T^{(k)}_1(\tilde \bX_n)$ and $T^{(k)}_2(\tilde \bX_n)$ are almost surely bounded. Consequently, 
	using (\ref{eq:consistency_T1}) of Theorem \ref{theorem:consistency_T1}
	and uniform integrability it follows that
	%Hence, it follows, using independence among $\tilde x_i$; $i=1\geq 1$, with respect to the cross-validation posteriors, that
	\begin{align}
		&\underset{m\rightarrow\infty}{\lim}\underset{n\rightarrow\infty}{\lim}~E\left(T^{(k)}_1(\tilde \bX_n)|\bX_n,\bY_{nm},\mathcal M_k,\tilde\Theta_k\right)
		\stackrel{a.s}{=}0;\label{eq:E_T1}\\
		&\underset{m\rightarrow\infty}{\lim}\underset{n\rightarrow\infty}{\lim}~E\left(T^{(k)}_2(\tilde \bX_n)|\bX_n,\bY_{nm},\mathcal M_k,\tilde\Theta_k\right)
		\stackrel{a.s}{=}0; \label{eq:E_T2}\\
		&\underset{m\rightarrow\infty}{\lim}\underset{n\rightarrow\infty}{\lim}~Var\left(T^{(k)}_1(\tilde \bX_n)|\bX_n,\bY_{nm},\mathcal M_k,\tilde\Theta_k\right)
		\stackrel{a.s}{=}0; \label{eq:Var_T1}\\
		&\underset{m\rightarrow\infty}{\lim}\underset{n\rightarrow\infty}{\lim}~Var\left(T^{(k)}_2(\tilde \bX_n)|\bX_n,\bY_{nm},\mathcal M_k,\tilde\Theta_k\right)
		\stackrel{a.s}{=}0. \label{eq:Var_T2}
	\end{align}
	The limits (\ref{eq:E_T1}) -- (\ref{eq:Var_T2}) imply that
	\begin{align}
		&\underset{m\rightarrow\infty}{\lim}\underset{n\rightarrow\infty}{\lim}~\tilde\ell_{knm}\stackrel{a.s}{=}0;\label{eq:l_limit}\\
		&\underset{m\rightarrow\infty}{\lim}\underset{n\rightarrow\infty}{\lim}~\tilde u_{knm}\stackrel{a.s}{=}0.\label{eq:u_limit}
	\end{align}
	Due to (\ref{eq:l_limit}) and Theorem \ref{theorem:disc_conv}, given any $\varepsilon>0$, for sufficiently large $m$ and $n$, 
	$\tilde\ell_{knm}-a_k+T^{(k)}(\bX_n)-\varepsilon<0$. Since $T^{(k)}(\tilde\bX_n)>0$ with probability one, we thus have 
	\begin{equation}
		\underset{m\rightarrow\infty}{\lim}\underset{n\rightarrow\infty}{\lim}~
		\pi\left(T^{(k)}(\tilde\bX_n)>\tilde\ell_{knm}-a_k+T^{(k)}(\bX_n)-\varepsilon\bigg |\bX_n,\bY_{nm},\mathcal M_k,\tilde\Theta_k\right)
		\stackrel{a.s}{=}1.
		\label{eq:prob_l}
	\end{equation}
	Also, due to (\ref{eq:u_limit}) and Theorem \ref{theorem:disc_conv}, given any $\varepsilon>0$, for sufficiently large $m$ and $n$,
	$\tilde u_{knm}-a_k+T^{(k)}(\bX_n)+\varepsilon>0$.
	Hence, given any $\varepsilon>0$, for sufficiently large $m$ and $n$, we have by Markov's inequality, 
	\begin{align}
		&\pi\left(T^{(k)}(\tilde\bX_n)>\tilde u_{knm}-a_k+T^{(k)}(\bX_n)+\varepsilon\bigg |\bX_n,\bY_{nm},\mathcal M_k,\tilde\Theta_k\right)\notag\\
		&<\left(\tilde u_{knm}-a_k+T^{(k)}(\bX_n)+\varepsilon\right)^{-2}\notag\\
		&\qquad\times
		\left[Var\left(T^{(k)}(\tilde \bX_n)|\bX_n,\bY_{nm},\mathcal M_k,\tilde\Theta_k\right)
		+\left\{E\left(T^{(k)}(\tilde \bX_n)|\bX_n,\bY_{nm},\mathcal M_k,\tilde\Theta_k\right)\right\}^2\right].
		\label{eq:prob_u0}
	\end{align}
	Taking limits of both sides of (\ref{eq:prob_u0}) and using (\ref{eq:E_T1}) -- (\ref{eq:Var_T2}) we obtain
	\begin{equation}
		\underset{m\rightarrow\infty}{\lim}\underset{n\rightarrow\infty}{\lim}~
		\pi\left(T^{(k)}(\tilde\bX_n)>\tilde u_{knm}-a_k+T^{(k)}(\bX_n)+\varepsilon\bigg |\bX_n,\bY_{nm},\mathcal M_k,\tilde\Theta_k\right)
		\stackrel{a.s}{=}0.
		\label{eq:prob_u}
	\end{equation}
	Combining (\ref{eq:prob_l}) and (\ref{eq:prob_u}) yields
	\begin{align}
		&\underset{m\rightarrow\infty}{\lim}\underset{n\rightarrow\infty}{\lim}~
		\pi\left(T^{(k)}(\tilde \bX_n)-T^{(k)}(\bX_n)\in [\tilde\ell_{knm}-a_k-\varepsilon,\tilde u_{knm}-a_k+\varepsilon]
		\bigg |\bX_n,\bY_{nm},\mathcal M_k,\tilde\Theta_k\right)\notag\\
		&=\underset{m\rightarrow\infty}{\lim}\underset{n\rightarrow\infty}{\lim}~
		\pi\left(T^{(k)}(\tilde\bX_n)>\tilde\ell_{knm}-a_k+T^{(k)}(\bX_n)-\varepsilon\bigg |\bX_n,\bY_{nm},\mathcal M_k,\tilde\Theta_k\right)\notag\\
		&\qquad\qquad-\underset{m\rightarrow\infty}{\lim}\underset{n\rightarrow\infty}{\lim}~
		\pi\left(T^{(k)}(\tilde\bX_n)>\tilde u_{knm}-a_k+T^{(k)}(\bX_n)+\varepsilon\bigg |\bX_n,\bY_{nm},\mathcal M_k,\tilde\Theta_k\right)\notag\\
		&	\stackrel{a.s}{=}1,\notag
	\end{align}
	thus proving (\ref{eq:consistency_T}).

\end{proof}

\begin{remark}
	\label{remark:consistency_T}
	In all the examples provided in \ctn{Chat20a}, it has been shown that the conditions of Theorem \ref{theorem:disc_conv} are satisfied.
	Hence, Theorem \ref{theorem:consistency_T} holds for all the examples presented in \ctn{Chat20a}.
\end{remark}

\subsection{Convergence of the posterior probabilities of $H_{1k}$}
\label{subsec:conv_H1k}
\begin{theorem}
	\label{theorem:conv_vknm}
	Assume that for $k=1,\ldots,K$, $\mathcal M_k$ satisfies conditions (S1)--(S7) of Shalizi, and that the competing models as well as the true model
	have densities with respect to some common $\sigma$-finite measure.
	Also assume that the posterior associated with $\mathcal M_k$ is dominated by the prior, which is again absolutely continuous
	with respect to some appropriate $\sigma$-finite measure, and that the priors satisfy $\pi(\theta_k|\mathcal M_k)>0$ for all $\theta_k\in\Theta_k$. 
	Let $h_{\tilde k}\left(\Theta_{\tilde k}\right)=\min\{h_k\left(\Theta_k\right):k=1,\ldots,K\}$.
	%If $\gamma>2h_k\left(\tilde\Theta^c_k\right)$, where $\gamma$ given in (\ref{eq:S5_1}) under assumption (S5), %then for any $m\geq 1$,  
%\begin{equation}
%	\underset{n\rightarrow\infty}{\lim}~\frac{1}{n}\log\pi(\tilde\Theta^c_k|\bX_n,\bY_{nm})\stackrel{a.s.}{=}-J_k\left(\tilde\Theta^c_k\right),
%	\label{eq:post_theta}
%\end{equation}
%	where $J_k$ corresponds to (\ref{eq:J}) and (\ref{eq:J2}) associated with model $\mathcal M_k$.
Then 
\begin{equation}
	\underset{m\rightarrow\infty}{\lim}\underset{n\rightarrow\infty}{\lim}~v_{knm}\stackrel{a.s.}{=}
	\begin{cases}1 & \mbox{if}~k\neq\tilde k\\ 0 & \mbox{if}~k=\tilde k.  \end{cases}
	\label{eq:post_v2}
\end{equation}
\end{theorem}
\begin{proof}
%\begin{equation}
%	H_{1k}:\left\{p_k\leq 1-\eta\right\}\bigcup\left\{p_k>1-\eta,\theta_k\in\tilde\Theta^c_k\right\}\bigcup\left\{p_k>1-\eta,\theta_k\in\tilde\Theta_k,
%	\frac{\left|T(\tilde \bX_n)-T(\bX_n)\right|}{\sqrt{V_k\left(T(\tilde \bX_n)|\bY_n\right)}}>\varepsilon\right\}.
%	\label{eq:H1}
%\end{equation}
First, let $k\neq\tilde k$. Then 
\begin{align}
v_{knm}&=\pi\left(\zeta\neq k|\bX_n,\bY_{nm}\right)
	+\pi\left(\zeta=k,\theta_k\in\tilde\Theta^c_k|\bX_n,\bY_{nm}\right)\notag\\
	&+\pi\left(\zeta=k,\theta_k\in\tilde\Theta_k,
	%\frac{\left|T^{(k)}(\tilde \bX_n)-T^{(k)}(\bX_n)\right|}
	%{\sqrt{Var\left(T^{(k)}(\tilde \bX_n)|\bX_n,\bY_{nm},\mathcal M_k,\tilde\Theta_k\right)}}\in (a_k-\varepsilon,a_k+\varepsilon)^c
	%\bigg |\bX_n,\bY_{nm}
	T^{(k)}(\tilde \bX_n)-T^{(k)}(\bX_n)\in [\tilde\ell_{knm}-a_k-\varepsilon,\tilde u_{knm}-a_k+\varepsilon]^c\bigg |\bX_n,\bY_{nm}
	\right).
\end{align}
	Since $k\neq\tilde k$, it follows due to (\ref{eq:post_p}) that for any $m\geq 1$, as $n\rightarrow\infty$,
	\begin{equation}
		\pi\left(\zeta\neq k|\bX_n,\bY_{nm}\right)=\pi\left(\zeta=\tilde k|\bX_n,\bY_{nm}\right)
		+\sum_{j\neq k,\tilde k}\pi\left(\zeta\neq k|\bX_n,\bY_{nm}\right)\stackrel{a.s.}{\longrightarrow} 1.
		\label{eq:conv1}
	\end{equation}
	Using (\ref{eq:post_p}) again it follows that for any $m\geq 1$,
	\begin{align}
		&\pi\left(\zeta=k,\theta_k\in\tilde\Theta^c_k|\bX_n,\bY_{nm}\right)\leq \pi\left(\zeta=k|\bX_n,\bY_{nm}\right)
		\stackrel{a.s.}{\longrightarrow} 0,~\mbox{as}~n\rightarrow\infty\label{eq:term2}
	\end{align}
	and
	\begin{align}
		&\pi\left(\zeta=k,\theta_k\in\tilde\Theta_k,
		%\frac{\left|T^{(k)}(\tilde \bX_n)-T^{(k)}(\bX_n)\right|}
		%{\sqrt{Var\left(T^{(k)}(\tilde \bX_n)|\bX_n,\bY_{nm},\mathcal M_k,\tilde\Theta_k\right)}}\in (a_k-\varepsilon,a_k+\varepsilon)^c
		T^{(k)}(\tilde \bX_n)-T^{(k)}(\bX_n)\in [\tilde\ell_{knm}-a_k-\varepsilon,\tilde u_{knm}-a_k+\varepsilon]^c
		\bigg |\bX_n,\bY_{nm}\right)\notag\\
		&\qquad\qquad\leq \pi\left(\zeta=k|\bX_n,\bY_{nm}\right)\stackrel{a.s.}{\longrightarrow} 0,~\mbox{as}~n\rightarrow\infty\label{eq:term3}.
	%	&\qquad=\sum_{r=1}^K\pi\left(\zeta=r|\bX_n,\bY_{nm}\right)\pi(p_k>1-\eta|\zeta=r)
	%	\pi\left(\theta_k\in\tilde\Theta^c_k\big |\bX_n,\bY_{nm},\zeta=r\right)\notag\\
	%	&\qquad\leq\sum_{r=1}^K\pi\left(\zeta=r|\bX_n,\bY_{nm}\right)\pi\left(\theta_k\in\tilde\Theta^c_k\big |\bX_n,\bY_{nm},\zeta=r\right)\notag\\
	%	&\qquad\leq \pi\left(\theta_k\in\tilde\Theta^c_k\big |\bX_n,\bY_{nm},\zeta=k\right)\notag\\
	%	&\qquad+\sum_{r=1,r\neq k}^K\pi\left(\zeta=r|\bX_n,\bY_{nm}\right)\pi(p_k>1-\eta|\zeta=r)\pi\left(\theta_k\in\tilde\Theta^c_k\big |\bX_n,\bY_{nm},\zeta=r\right)
	\end{align}
	Results (\ref{eq:conv1}), (\ref{eq:term2}) and (\ref{eq:term3}) imply that if $k\neq\tilde k$, then for any $m\geq 1$, 
	\begin{equation}
		v_{knm}\stackrel{a.s.}{\longrightarrow} 1,~\mbox{as}~n\rightarrow\infty.
		\label{eq:conv_v1}
	\end{equation}
	Now let us obtain the limit of $v_{knm}$ when $k=\tilde k$. By (\ref{eq:post_p}), 
	\begin{equation}
		\pi\left(\zeta\neq \tilde k|\bX_n,\bY_{nm}\right)\stackrel{a.s.}{\longrightarrow} 0,~\mbox{as}~n\rightarrow\infty.
		\label{eq:conv2}
	\end{equation}
	%Now since $\tilde\Theta^c_{\tilde k}\subset\Theta_{\tilde k}$, $h_{\tilde k}\left(\tilde\Theta^c_{\tilde k}\right)>h_{\tilde k}\left(\Theta_{\tilde k}\right)$.
	%Hence, from (\ref{eq:post_conv1}) it follows that for any $m\geq 1$, 
	%\begin{equation}
	%	\pi\left(\theta_{\tilde k}\in\tilde\Theta^c_{\tilde k}|\bX_n,\bY_{nm}\right)\stackrel{a.s.}{\longrightarrow} 0,~\mbox{as}~n\rightarrow\infty.
	%	\label{eq:shalizi_conv1}
	%\end{equation}
	For any $m\geq 1$, using (\ref{eq:shalizi_conv1}) we obtain
	\begin{align}
		&\pi\left(\zeta=\tilde k,\theta_{\tilde k}\in\tilde\Theta^c_{\tilde k}|\bX_n,\bY_{nm}\right)
		\leq \pi\left(\theta_{\tilde k}\in\tilde\Theta^c_{\tilde k}|\bX_n,\bY_{nm}\right)
		\stackrel{a.s.}{\longrightarrow} 0,~\mbox{as}~n\rightarrow\infty.\label{eq:term22}
	\end{align}
	Now note that
	\begin{align}
		&\pi\left(\zeta=\tilde k,\theta_{\tilde k}\in\tilde\Theta_{\tilde k},
		%\frac{\left|T^{(\tilde k)}(\tilde \bX_n)-T^{(\tilde k)}(\bX_n)\right|}
		%{\sqrt{Var\left(T^{(\tilde k)}(\tilde \bX_n)|\bX_n,\bY_{nm},\mathcal M_{\tilde k},\tilde\Theta_{\tilde k}\right)}}
		%\in (a_{\tilde k}-\varepsilon,a_{\tilde k}+\varepsilon)^c
		T^{(\tilde k)}(\tilde \bX_n)-T^{(\tilde k)}(\bX_n)\in [\tilde\ell_{\tilde knm}-a_{\tilde k}-\varepsilon,\tilde u_{\tilde knm}-a_{\tilde k}+\varepsilon]^c
	        \bigg |\bX_n,\bY_{nm}\right)\notag\\
		&=\pi\left(\zeta=\tilde k|\bX_n,\bY_{nm}\right)\notag\\
		&\qquad\times\pi\left(\theta_{\tilde k}\in\tilde\Theta_{\tilde k},
		%\frac{\left|T^{(\tilde k)}(\tilde \bX_n)-T^{(\tilde k)}(\bX_n)\right|}
		%{\sqrt{Var\left(T^{(\tilde k)}(\tilde \bX_n)|\bX_n,\bY_{nm},\mathcal M_{\tilde k},\tilde\Theta_{\tilde k}\right)}}
		%\in (a_{\tilde k}-\varepsilon,a_{\tilde k}+\varepsilon)^c
		T^{(\tilde k)}(\tilde \bX_n)-T^{(\tilde k)}(\bX_n)\in [\tilde\ell_{\tilde knm}-a_{\tilde k}-\varepsilon,\tilde u_{\tilde knm}-a_{\tilde k}+\varepsilon]^c
		\bigg |\bX_n,\bY_{nm},\zeta=\tilde k\right)\notag\\
		&\leq \pi\left(
		%\frac{\left|T^{(\tilde k)}(\tilde \bX_n)-T^{(\tilde k)}(\bX_n)\right|}
		%{\sqrt{Var\left(T^{(\tilde k)}(\tilde \bX_n)|\bX_n,\bY_{nm},\mathcal M_{\tilde k},\tilde\Theta_{\tilde k}\right)}}
		%\in (a_{\tilde k}-\varepsilon,a_{\tilde k}+\varepsilon)^c
		T^{(\tilde k)}(\tilde \bX_n)-T^{(\tilde k)}(\bX_n)\in [\tilde\ell_{\tilde knm}-a_{\tilde k}-\varepsilon,\tilde u_{\tilde knm}-a_{\tilde k}+\varepsilon]^c
		\bigg|\bX_n,\bY_{nm},\zeta=\tilde k\right)\notag\\
		&	\stackrel{a.s.}{\longrightarrow}0,~\mbox{as}~m\rightarrow\infty,~n\rightarrow\infty,~\mbox{due to (\ref{eq:consistency_T}).}\label{eq:term32}
	\end{align}
	From (\ref{eq:conv2}), (\ref{eq:term22}) and (\ref{eq:term32}) it follows that
	\begin{equation}
		v_{\tilde knm}\stackrel{a.s.}{\longrightarrow} 0,~\mbox{as}~m\rightarrow\infty,n\rightarrow\infty.
		\label{eq:conv_v1_tilde}
	\end{equation}
The limits (\ref{eq:conv_v1}) and (\ref{eq:conv_v1_tilde}) show that (\ref{eq:post_v2}) holds.
\end{proof}

\section{Asymptotic optimality theory for our multiple testing procedure}
\label{sec:optimality}
%With the above notations, in this section we show that our multiple procedure is asymptotically consistent under any general dependent model satisfying the conditions 
%in Section \ref{subsec:assumptions_shalizi}. 
%Let us first formally define what we mean by asymptotic consistency of a multiple testing procedure when

Let $h_{\tilde k}\left(\Theta_{\tilde k}\right)=\min\{h_k\left(\Theta_k\right):k=1,\ldots,K\}$.
Also let us define $\tilde\bd=(\tilde d_1,\ldots,\tilde d_K)$, where
\begin{equation}
	\tilde d_{k}=\begin{cases} 1 & \mbox{if}~k\neq \tilde k\\ 0 & \mbox{if}~k=\tilde k.
	\end{cases}
	\label{eq:tilde_d}
\end{equation}

\begin{definition}
A multiple testing method for the inverse model selection is said to be asymptotically optimal for which
	\begin{equation*}
		\lim_{m\rightarrow\infty} \lim_{n\rightarrow\infty} \delta(\tilde\bd|\bX_n,\bY_{nm}) \stackrel{a.s.}{=}1.
	\end{equation*}
	\label{def:optimal}
\end{definition}
Recall the constant $\beta_{nm}$ in (\ref{eq:beta1}), which is the penalizing constant between the error $E$ and true positives $TP$. 
%Also $mpBFDR$ is decreasing in $\beta_n$ for any fixed sample size $\beta_n$ \ctp{chandra2017}.
For consistency of the non-marginal procedure, we need certain conditions on $\beta_n$, which we state below. These conditions will also play important roles in the asymptotic studies
of the different versions of $FDR$ and $FNR$ that we consider.
%Also for \textcolor{red}{consistent} definitions of the multiple testing error measures, we need condition on the true decision configuration $\bd^t$. We state the conditions below:
\begin{enumerate}[label={(A\arabic*)}]	
	\item \label{A1} 
		We assume that the sequence $\beta_{nm}$ is neither too small nor too large, that is,
	\begin{align}
		\underline\beta&=\underset{m\geq 1,n\geq 1}{\liminf}~\beta_{nm}>0;\label{eq:liminf_beta}\\
		\overline\beta&=\underset{m\geq 1,n\geq 1}{\limsup}~\beta_{nm}<1.\label{eq:limsup_beta}
	\end{align}
	%\item \label{A2} We assume that neither all the null hypotheses are true and nor all of then are false, that is, $\bd^t\neq\bzero$ and $\bd^t\neq\bone$, where
	%$\bzero$ and $\bone$ are vectors of 0's and 1's respectively.
\end{enumerate}
With this conditions we propose and prove the following results.
\begin{theorem}
	\label{theorem:asymp_opt}
	Let $\delta(\cdot|\bX_n,\bY_{nm})$ denote the decision rule given data $\bX_n$ and $\bY_{nm}$. %Assume conditions \ref{shalizi1}--\ref{shalizi7}.
	Assume the conditions of Theorem \ref{theorem:conv_vknm} and condition \ref{A1} on $\beta_{nm}$. Then the decision procedure is asymptotically optimal.
\end{theorem}
\begin{proof}
	Due to \ref{A1}, given $\epsilon_1>0$, there exist $m_0\geq 1$ and $n_0\geq 1$ such that for $m\geq m_0$ and $n\geq n_0$, 
	\begin{equation}
	0<\underline\beta-\epsilon_1<\beta_{nm}<\overline\beta+\epsilon_1<1.
		\label{eq:beta_bound}
	\end{equation}
	By (\ref{eq:post_v2}), for any $0<\epsilon_2<1-\overline\beta-\epsilon_1$, for $k\neq \tilde k$, there exist $m_k\geq 1$ and $n_k\geq 1$ 
	such that for $m\geq m_k$ and $n\geq n_k$,
	\begin{equation}
		v_{knm}>1-\epsilon_2>\overline\beta+\epsilon_1.
		\label{eq:post_v_bound1}
	\end{equation}
	Also, for $0<\epsilon_3<\underline\beta-\epsilon_1$, there exist $m_{\tilde k}\geq 1$ and $n_{\tilde k}\geq 1$
	such that for $m\geq m_{\tilde k}$ and $n\geq n_{\tilde k}$,
	\begin{equation}
		v_{\tilde knm}<\epsilon_3<\underline\beta-\epsilon_1.
		\label{eq:post_v_bound2}
	\end{equation}
	Let $\tilde m=\max\{m_0,m_1,\ldots,m_K\}$ and $\tilde n=\max\{n_0,n_1,\ldots,n_K\}$. Then it can be seen from (\ref{eq:beta_bound}), (\ref{eq:post_v_bound1})
	and (\ref{eq:post_v_bound2}) that for $m\geq\tilde m$ and $n\geq\tilde n$ the following hold almost surely:
	\begin{align}
		v_{knm}&>\beta_{nm},~\mbox{if}~k\neq\tilde k;\label{eq:post_v_lower_bound}\\
		v_{knm}&<\beta_{nm},~\mbox{if}~k=\tilde k.\label{eq:post_v_upper_bound}
	\end{align}
	Using (\ref{eq:post_v_lower_bound}) and (\ref{eq:post_v_upper_bound}) in (\ref{eq:optimal_decision}) shows that
	for $m\geq\tilde m$ and $n\geq\tilde n$,
	\begin{equation}
		\widehat d_k=\begin{cases}1 & \mbox{if}~k\neq\tilde k;\\ 0 & \mbox{if}~k=\tilde k.\end{cases}
			\label{eq:opt_d_asymp}
	\end{equation}
	In other words, almost surely, $\widehat\bd=\tilde\bd$ for $m\geq\tilde m$ and $n\geq\tilde n$. This completes the proof.
%	This is immediate from (\ref{eq:optimal_decision}), (\ref{eq:post_v2}) and \ref{A1}.
\end{proof}

\begin{remark}
	\label{remark:remark_expectation_delta}
Since $\delta(\cdot|\bX_n,\bY_{nm})$ is an indicator function, the following also holds:
	\begin{equation*}
		\lim_{m\rightarrow\infty}\lim_{n\rightarrow\infty} E_{\bY_{nm}|\bX_n}\left[\delta(\tilde\bd|\bX_n,\bY_{nm}) \right]=1.
	\end{equation*}	
\end{remark}
%We have already mentioned that the optimal decision rules corresponding to the loss function in (\ref{eq:loss_mul}) is a special case of the non-marginal method when dependence among the hypotheses is ignored. As we have not considered any particular structure of $G_i$'s in Theorem \ref{th:nmd_consistent}, consistency of the additive loss-function based method can also be obtained from the previous theorem.
%\begin{corollary}
%	Assuming condition \ref{A1}, the optimal decision rule corresponding to the additive loss function (\ref{eq:loss_mul}) is asymptotically consistent.
%\end{corollary}

%The condition \ref{A1} is necessary for the asymptotic consistency of both the non-marginal method and additive loss function based method. This ensures that the penalizing constant is asymptotically bounded away from 0 and 1, that is, it is neither too small nor too large. 
%Notably, \ref{A2} is not required for the consistency results. The role of \ref{A2} is to ensure that the denominator terms in the multiple testing error measures (defined in Section \ref{subsec:Bayesian_errors}) do not become 0. 

\section{Asymptotic theory of the error measures}
\label{sec:error_asymptotics}

\subsection{Convergence of versions of $FDR$ and $FNR$}
\label{subsec:fdr_fnr_conv}

\begin{theorem}
	\label{theorem:fdr_conv1}
	%Let $\delta(\cdot|\bX_n,\bY_{nm})$ denote the decision rule given data $\bX_n$ and $\bY_{nm}$. %Assume conditions \ref{shalizi1}--\ref{shalizi7}.
	Assume the conditions of Theorem \ref{theorem:conv_vknm} and condition \ref{A1} on $\beta_{nm}$. Then 
	\begin{align}
		&\underset{m\rightarrow\infty}{\lim}\underset{n\rightarrow\infty}{\lim}~cFDR_{nm}\stackrel{a.s.}{=}0;\label{eq:cfdr_conv}\\
		&\underset{m\rightarrow\infty}{\lim}\underset{n\rightarrow\infty}{\lim}~pBFDR_{nm}~{=}~0.\label{eq:pbfdr_conv}
	\end{align}
\end{theorem}

\begin{proof}
From (\ref{eq:pBFDR}) observe that
	\begin{equation}
		cFDR_{nm}=\frac{\sum_{k=1}^{K}\tilde d_k(1-v_{knm})}{\sum_{k=1}^{K}\tilde d_k \vee 1}\delta(\tilde\bd|\bX_n,\bY_{nm})
		+\sum_{\bd\neq\tilde\bd\in\mathbb{D}}\frac{\sum_{k=1}^{K}d_k(1-v_{knm})}{\sum_{k=1}^{K}d_k \vee 1}\delta(\bd|\bX_n,\bY_{nm})
		\label{eq:fdr1}
	\end{equation}
	The proof of Theorem \ref{theorem:asymp_opt} shows that there exist $\tilde m\geq 1$ and $\tilde n\geq 1$ such that
	$\delta(\tilde\bd|\bX_n,\bY_{nm})=1$ almost surely for $m\geq\tilde m$ and $n\geq\tilde n$. This, combined with (\ref{eq:fdr1}) %and (\ref{eq:post_v2}) 
	shows that for $m\geq\tilde m$ and $n\geq\tilde n$, almost surely,
	\begin{equation}
		cFDR_{nm}=\frac{\sum_{k=1}^{K}\tilde d_k(1-v_{knm})}{\sum_{k=1}^{K}\tilde d_k \vee 1}=\frac{\sum_{k\neq\tilde k}(1-v_{knm})}{K-1}.
		%\rightarrow 0,~\mbox{as}~m\rightarrow\infty,~n\rightarrow\infty.
		\label{eq:fdr2}
	\end{equation}
	Applying (\ref{eq:post_v2}) to the right most side of (\ref{eq:fdr2}) shows that
	$$cFDR_{nm}\stackrel{a.s.}{\longrightarrow}0,~\mbox{as}~m\rightarrow\infty,~n\rightarrow\infty,$$
	establishing (\ref{eq:cfdr_conv}).

	Since $cFDR_{nm}<1$ almost surely, (\ref{eq:pbfdr_conv}) follows from (\ref{eq:cfdr_conv}) by uniform integrability.
\end{proof}

\begin{theorem}
	\label{theorem:fnr_conv1}
	%Let $\delta(\cdot|\bX_n,\bY_{nm})$ denote the decision rule given data $\bX_n$ and $\bY_{nm}$. %Assume conditions \ref{shalizi1}--\ref{shalizi7}.
	Assume the conditions of Theorem \ref{theorem:conv_vknm} and condition \ref{A1} on $\beta_{nm}$. Then 
	\begin{align}
		&\underset{m\rightarrow\infty}{\lim}\underset{n\rightarrow\infty}{\lim}~cFNR_{nm}\stackrel{a.s.}{=}0;\label{eq:cfnr_conv}\\
		&\underset{m\rightarrow\infty}{\lim}\underset{n\rightarrow\infty}{\lim}~pBFNR_{nm}~{=}~0.\label{eq:pbfnr_conv}
	\end{align}
\end{theorem}
\begin{proof}
	It follows from (\ref{eq:pBFNR}) and the proof of Theorem \ref{theorem:asymp_opt} that there exist $\tilde m\geq 1$ and $\tilde n\geq 1$ such that
	for $m\geq\tilde m$ and $n\geq\tilde n$, almost surely,
	\begin{align}
		cFNR_{nm}&=\frac{\sum_{k=1}^{K}(1-\tilde d_k)v_{knm}}{\sum_{k=1}^{K}(1-\tilde d_k) \vee 1}\delta(\tilde\bd|\bX_n,\bY_{nm})
		+\sum_{\bd\neq\tilde\bd\in\mathbb{D}}\frac{\sum_{k=1}^{K}(1-d_k)v_{knm}}{\sum_{k=1}^{K}(1-d_k) \vee 1}\delta(\bd|\bX_n,\bY_{nm})\notag\\
		&=\frac{\sum_{k=1}^{K}(1-\tilde d_k)v_{knm}}{\sum_{k=1}^{K}(1-\tilde d_k) \vee 1}=v_{\tilde k nm}.
		\label{eq:fnr1}
	\end{align}
	Application of (\ref{eq:post_v2}) to the right most side of (\ref{eq:fnr1}) yields
	$$cFNR_{nm}\stackrel{a.s.}{\longrightarrow}0,~\mbox{as}~m\rightarrow\infty,~n\rightarrow\infty,$$
	establishing (\ref{eq:cfnr_conv}).

	Again, (\ref{eq:pbfnr_conv}) follows from (\ref{eq:cfnr_conv}) by uniform integrability, since $cFNR_{nm}$ is almost surely
	bounded above by one.

\end{proof}

\subsection{Convergence of versions of $FNR$ when versions of $FDR$ are $\alpha$-controlled}
\label{subsec:alpha_control}

\begin{theorem}
	\label{theorem:fdr_alpha_control}
	Assume the conditions of Theorem \ref{theorem:conv_vknm}. %Then for any decision $\hat\bd$ such that $\hat d_{\tilde k}=1$, 
	%for any $\alpha\in\left(0,K^{-1}\right)$, there exists a 
	%for any 
	Then $\alpha=K^{-1}$ is the only asymptotic $FDR$ control possible in the sense that there exist sequences 
	$\beta_{nm}\rightarrow 0$ as $m\rightarrow\infty$ and $n\rightarrow\infty$ such that 
	the following hold:
	\begin{align}
		&\underset{m\rightarrow\infty}{\lim}\underset{n\rightarrow\infty}{\lim}~cFDR_{nm}\stackrel{a.s.}{=}K^{-1};\label{eq:cfdr_alpha1}\\
		&\underset{m\rightarrow\infty}{\lim}\underset{n\rightarrow\infty}{\lim}~pBFDR_{nm}~=~K^{-1}. %\alpha.
		\label{eq:pbfdr_alpha1}
	\end{align}
%	Also, for such decisions, $\beta_{nm}\rightarrow 0$ as $m\rightarrow\infty$ and $n\rightarrow\infty$.
\end{theorem}
\begin{proof}
	%The proof will employ concepts proposed in \ctn{Chandra19} and \ctn{Chandra20}. 

	It follows from \ctn{Chandra19} (see also \ctn{Chandra20}) that $pBFDR_{nm}$ is continuous and decreasing in $\beta_{nm}$, for any given $m\geq 1$ and $n\geq 1$. 
	Hence, the maximum error given any $m\geq 1$ and $n\geq 1$ occurs when $\beta_{nm}=0$. Hence, in this case, for any given $m\geq 1$ and $n\geq 1$,
	for our multiple testing procedure we must maximize $\sum_{k=1}^Kd_kv_{knm}$ with respect to $\bd$. This of course yields $\hat d_k=1$, for $k=1,\ldots,K$.
	For this decision $\hat\bd$, we obtain using (\ref{eq:post_v2}): 
	\begin{equation}
		cFDR_{nm}=\frac{\sum_{k=1}^{K}\hat d_k(1-v_{knm})}{\sum_{k=1}^{K}\hat d_k \vee 1}=\frac{\sum_{k=1}^K(1-v_{knm})}{K}
		\stackrel{a.s.}{\longrightarrow} K^{-1},~\mbox{as}~m\rightarrow\infty,~n\rightarrow\infty.
		\label{eq:cfdr_alpha_max1}
	\end{equation}
	Uniform integrability and (\ref{eq:cfdr_alpha_max1}) shows that %for any $\hat\bd$ such that $d_{\hat k}=1$, 
	when $\beta_{nm}=0$ for any $m\geq 1$ and $n\geq 1$,
	\begin{equation}
		pBFDR_{nm}\rightarrow K^{-1},~\mbox{as}~m\rightarrow\infty,~n\rightarrow\infty.
		\label{eq:pbfdr_alpha_max1}
	\end{equation}
	Now consider any sequence $\beta_{nm}$ that yields any decision $\hat\bd$ such that $\hat d_{\tilde k}=1$ almost surely, for sufficiently large $m$ and $n$. 
	Note that $\hat d_{\tilde k}=1$ can occur only if $v_{\tilde k nm}>\beta_{nm}$. Since $v_{\tilde k nm}\stackrel{a.s.}{\longrightarrow}0$ by (\ref{eq:post_v2}),
	we must have $\beta_{nm}\rightarrow 0$ as $m\rightarrow\infty$ and $n\rightarrow\infty$ in such cases. Also since 
	$v_{k nm}\stackrel{a.s.}{\longrightarrow}1$ for $k\neq\tilde k$ due to (\ref{eq:post_v2}), it follows that $\hat d_k=1$ almost surely for large enough $m$ and $n$,
	for $k\neq\tilde k$.
	Hence, the limits (\ref{eq:cfdr_alpha_max1}) and (\ref{eq:pbfdr_alpha_max1}) continue to hold in all cases such that 
	$\hat d_{\tilde k}=1$, for sufficiently large $m$ and $n$.

	%Now note that if $\underset{m\rightarrow\infty,n\rightarrow\infty}{\lim\inf}~\beta_{nm}>0$ for such $\beta_{nm}$-sequences, then 
	%$cFDR_{nm}\stackrel{a.s.}{\longrightarrow} 0$ and
	%$pBFDR_{nm}\rightarrow 0$ as $m\rightarrow\infty$ and
	%$n\rightarrow\infty$, due to (\ref{eq:cfdr_conv}) and (\ref{eq:pbfdr_conv}), which would contradict (\ref{eq:cfdr_alpha_max1}) and 
	%(\ref{eq:pbfdr_alpha_max1}). Hence, we must have
	%$\beta_{nm}\rightarrow 0$ as $m\rightarrow\infty$ and $n\rightarrow\infty$, for such $\beta_{nm}$-sequences. 

	On the other hand, for any sequence $\beta_{nm}$ that yields any decision $\hat\bd$ such that $\hat d_{\tilde k}=0$ almost surely for sufficiently large $m$ and $n$,
	it is easily seen that $cFDR_{nm}\stackrel{a.s.}{\longrightarrow}0$ and $pBFDR_{nm}\rightarrow 0$, as $m\rightarrow\infty$ and $n\rightarrow\infty$.

	%Let $\epsilon< K^{-1}-\alpha$. Due to (\ref{eq:pbfdr_alpha_max1}), for $\beta_{nm}=0$, there exist $m_0\geq 1$ and $n_0\geq 1$ such that for $m\geq m_0$ and $n\geq n_0$,
	%$pFDR_{nm}>K^{-1}-\epsilon>\alpha$. Because of this and since for any $m\geq 1$ and $n\geq 1$, $pBFDR_{nm}$ is continuous and decreasing in $\beta_{nm}$, 
	%for all $m\geq m_0$ and $n\geq n_0$, there exists $\tilde\beta_{nm}\in (0,1)$
	%such that 
	%\begin{equation}
	%pBFDR_{nm}=\alpha.
	%	\label{eq:pbfdr_alpha_max2}
	%\end{equation}

	In other words, asymptotic control of $cFDR_{nm}$ and $pBFDR_{nm}$ is possible only at $\alpha=K^{-1}$.
\end{proof}

%\begin{remark}
%	\label{remark:cfdr_alpha_control}
%	It follows from \ctn{Chandra19} that $cFDR_{nm}$ is decreasing in $\beta_{nm}$, for any given $m\geq 1$ and $n\geq 1$, although continuity does not hold.
%	In fact, it can be easily seen that $cFDR_{nm}$ is a step function of $\beta_{nm}$. Hence, except for at most a countable number of points in $\left(0,K^{-1}\right)$, 
%	it is guaranteed that there exists $\tilde\beta_{nm}\in (0,1)$ for large enough $m$ and $n$ such that $cFDR_{nm}=\alpha\in\left(0,K^{-1}\right)$. 
%	%$cFDR_{nm}$ can get sufficiently close to $\alpha$.
%\end{remark}

\begin{theorem}
	\label{theorem:fnr_alpha_control}
	%Assume the conditions of Theorem \ref{theorem:conv_vknm}. 
	Assume that either of $cFDR_{nm}$ or $pBFDR_{nm}$ is asymptotically controlled at $\alpha=K^{-1}$.
	%for any $\alpha\in\left(0,K^{-1}\right)$, there exists a sequence $\beta_{nm}\rightarrow 0$ as $m\rightarrow\infty$ and $n\rightarrow\infty$ such that
	Then for sufficiently large $m$ and $n$,
	\begin{align}
		%&\underset{m\rightarrow\infty}{\lim}\underset{n\rightarrow\infty}{\lim}~cFDR_{nm}\stackrel{a.s.}{=}\alpha;\label{eq:cfdr_alpha1}\\
		%&\underset{m\rightarrow\infty}{\lim}\underset{n\rightarrow\infty}{\lim}~pBFDR_{nm}~=~\alpha.\label{eq:pbfdr_alpha1}
		cBFNR_{nm}&\stackrel{a.s.}{=}0;\label{eq:cbfnr_alpha1}\\
		pBFNR_{nm}&~=~0.\label{eq:pbfnr_alpha1}
	\end{align}
\end{theorem}
\begin{proof}
From the proof of Theorem \ref{theorem:fdr_alpha_control}, recall that for asymptotic control of $cFDR_{nm}$ or $pBFDR_{nm}$ at $\alpha=K^{-1}$, we must obtain 
decision $\hat\bd$ where $\hat d_k=1$, for $k=1,\ldots,K$, for large enough $m$ and $n$. Hence, (\ref{eq:cbfnr_alpha1}) and (\ref{eq:pbfnr_alpha1}) follow simply 
from the definitions of $cBFNR_{nm}$ and $pBFNR_{nm}$ with $\bd=\hat\bd$ for sufficiently large $m$ and $n$.
\end{proof}
\begin{remark}
	\label{remark:fnr_alpha_control}
	Theorem \ref{theorem:fnr_alpha_control} shows that $cBFNR_{nm}$ and $pBFNR_{nm}$ are exactly zero for large enough $m$ and $n$. Needless to mention,
	these are far stronger results than convergence to zero in the limit. In other words, essentially in keeping with the classical hypothesis testing
	paradigm, $\alpha$-control of the Type-I error actually minimizes the Type-II error for sufficiently large $m$ and $n$.
\end{remark}

\section{Modification of the multiple testing procedure for practical implementation}
\label{sec:modification_practical}
Note that the constants $a_k$ in (\ref{eq:H0}) and (\ref{eq:H1}), which depend upon the true parameter(s) $\theta_0$, are unknown, since $\theta_0$ is unknown.
The constants $a_k$ also depend upon $\tilde\theta_k$, the minimizer of the KL-divergence of model $\mathcal M_k$ from the true model. Since the true model
itself is generally unknown, $\tilde\theta_k$ is usually unknown. Estimation of these parameters need not be reliable unless assumptions regarding
the true model is accurate enough. 

In practice, the considered models $\mathcal M_k$; $k=1,\ldots,K$, are expected to be carefully chosen for final model selection
so that misspecifications, if any, are not expected to be severe. Hence, for finite samples, where the variability of $T^{(k)}(\bX_n)$, and hence
the desired credible intervals, are reasonably large,
$a_k$ is not expected to play significant role. In such cases, it makes sense to set $a_k=0$. Similarly, setting $\varepsilon=0$ also makes sense.

Also in practice, one might set $\tilde\Theta_k=\Theta_k$
since accurate specification of a small set containing $\tilde\theta_k$ is not possible without knowledge of $\tilde\theta_k$.
With these, for practical purposes we re-formulate (\ref{eq:H0}) and (\ref{eq:H1}) as follows: 
\begin{equation}
	H_{0k}:\zeta=k,
	T^{(k)}(\tilde \bX_n)-T^{(k)}(\bX_n)\in [\tilde\ell_{knm},\tilde u_{knm}]
	\label{eq:H0_1}
\end{equation}
versus
\begin{align}
	&H_{1k}:\left\{\zeta\neq k\right\}%\bigcup\left\{\zeta=k,\theta_k\in\tilde\Theta^c_k\right\}
	\bigcup\left\{\zeta=k,
	T^{(k)}(\tilde \bX_n)-T^{(k)}(\bX_n)\in [\tilde\ell_{knm},\tilde u_{knm}]^c\right\}.
	\label{eq:H1_1}
\end{align}
We shall consider the above hypotheses for our applications.

\section{First simulation study: selection among Poisson and geometric parametric and nonparametric inverse regression models}
\label{sec:simstudy_ms}
For our simulation experiments we consider the same data and models considered in \ctn{Chat20a} for their forward and inverse pseudo-Bayes factor illustration.
Specifically, we set $n=m=10$ and generate data from relevant Poisson distribution with the log-linear link function  and consider modeling the data with 
Poisson and geometric distributions with log, logit and probit links for linear regression as well as nonparametric regression modeled by 
Gaussian process having linear mean function
and squared exponential covariance. We also consider variable selection in these setups with respect to two different covariates. 

Here we demonstrate that the forward and inverse pseudo-Bayes factor results obtained by \ctn{Chat20a} for both the experiments involving 
model selection and variable selection can be significantly improved with our inverse multiple testing framework. 
Let us begin with the model selection framework. The true, data-generating distribution and the competing inverse regression models are of course detailed in \ctn{Chat20a}
but to make this article as self-contained as possible, we briefly describe these next.

%\subsection{Poisson versus geometric linear and nonparametric regresison models when the true model is Poisson linear regression}
%\label{subsec:poisson_vs_geometric}

\subsection{True and competing inverse regression models}
\label{subsec:true_competing}

\subsubsection{True distribution}
\label{subsubsec:true_distribution}
%Let us first consider the case 
%where $y_{ij}\sim Poisson(\lambda(x_i))$, where $\lambda(x)=H(\eta(x))$, as briefed in Section \ref{subsec:illustrations_prior} (ii).
%In particular, we let $H(\cdot)=\exp(\cdot)$ and $\eta(\cdot)$ be a Gaussian process with mean function $\mu(x)=\alpha+\beta x$ and covariance 
%$Cov\left(\eta(x_1),\eta(x_2)\right)=\sigma^2\exp\left\{-(x_1-x_2)^2\right\}$, where $\sigma$ is unknown. We assume that 
The true data-generating distribution for this experiment
is $y_{ij}\sim Poisson(\lambda(x_i))$, with $\lambda(x)=\exp(\alpha_0+\beta_0 x)$. We generate the data by simulating $\alpha_0\sim U(-1,1)$, $\beta_0\sim U(-1,1)$
and $x_i\sim U(-1,1)$; $i=1,\ldots,n$, and then finally simulating $y_{ij}\sim Poisson(\lambda(x_i))$; $j=1,\ldots,m$, $i=1,\ldots,n$.
We shall also consider the true model as one of the competing models when no misspecification is assumed. 

\subsubsection{Inverse Poisson linear regression model}
\label{subsubsec:inverse_poisson_linear}
In this setup we model the data as follows: $y_{ij}\sim Poisson(\lambda(x_i))$, with $\lambda(x)=\exp(\alpha+\beta x)$, and set the prior
$\pi\left(\alpha,\beta\right)=1$, for $-\infty<\alpha,\beta<\infty$. 
The prior for $\tilde x_i$ is given by $\pi(\tilde x_i|\alpha,\beta)\equiv U(a,b)$, where 
\begin{equation}
a=\min\left\{\beta^{-1}\left(\log\left(\bar y_i-\frac{c_1s_i}{\sqrt{m}}\right)-\alpha\right),
\beta^{-1}\left(\log\left(\bar y_i+\frac{c_2s_i}{\sqrt{m}}\right)-\alpha\right)\right\}
	\label{eq:a1}
\end{equation}
and 
\begin{equation}
b=\max\left\{\beta^{-1}\left(\log\left(\bar y_i-\frac{c_1s_i}{\sqrt{m}}\right)-\alpha\right),
\beta^{-1}\left(\log\left(\bar y_i+\frac{c_2s_i}{\sqrt{m}}\right)-\alpha\right)\right\}.
	\label{eq:b1}
\end{equation}
We set $c_1=1$ and $c_2=100$, for ensuring positive value of $\bar y_i-\frac{c_1s_i}{\sqrt{m}}$ (so that logarithm of this quantity is well-defined) 
and a reasonably large support of the prior for $\tilde x_i$. 

%We then compute 
%$$
%\pi(y_{i1}|\bY_{nm,-i},\bX_{n,-i},\mathcal M)=\int_{\mathcal X}\int_{\Theta} f(y_{i1}|\theta,\tilde x_i,\bY^{(i-1)}_1,\mathcal M)
%	d\pi(\tilde x_i,\theta|\bY_{nm,-i},\bX_{n,-i},\mathcal M)
%$$
%by Monte Carlo averaging of the relevant Poisson probability of $y_{i1}$ over realizations of $(\tilde x_i,\theta)=(\tilde x_i,\alpha,\beta)$
%generated from $\pi(\tilde x_i,\theta|\bY_{nm,-i},\bX_{n,-i},\mathcal M)$. Since it follows from (\ref{eq:inv1}) that
%$\pi(\tilde x_i,\theta|\bY_{nm,-i},\bX_{n,-i},\mathcal M)$ $=$ $\pi(\tilde x_i|\theta,\mathcal M)\pi(\theta|\bY_{nm,-i},\bX_{n,-i},\mathcal M)$, and 
%since realizations of $\theta$ from $\pi(\theta|\bY_{nm,-i},\bX_{n,-i},\mathcal M)$ are already available in the forward context, we simply generate
%$\tilde x_i$ given $\theta$ from the prior for $\tilde x_i$ to obtain realizations from $\pi(\tilde x_i,\theta|\bY_{nm,-i},\bX_{n,-i},\mathcal M)$.
%Note that for different $i$, only sub-samples of $\theta$ of size $1000$ from the original sample of size $20,000$ from the full posterior of $\theta$ are available,
%and each $\theta$ is repeated $100$ times. However, realizations of $\tilde x_i$ are all distinct in spite of repetitions of $\theta$-values.
%
%Once for each $i=1,\ldots,n$, the Monte Carlo estimates of $\pi(y_{i1}|\bY_{nm,-i},\bX_{n,-i},\mathcal M)$ are available, we finally obtain the estimate of
%$\frac{1}{n}\sum_{i=1}^n\log\pi(y_{i1}|\bY_{nm,-i},\bX_{n,-i},\mathcal M)$ using the individual Monte Carlo estimates.

\subsubsection{Inverse Poisson nonparametric regression model}
\label{subsubsec:inverse_poisson_nonpara}

We now consider the case where $y_{ij}\sim Poisson(\lambda(x_i))$, where $\lambda(x)=\exp(\eta(x))$, where
$\eta(\cdot)$ is a Gaussian process with mean function $\mu(x)=\alpha+\beta x$ and covariance 
$Cov\left(\eta(x_1),\eta(x_2)\right)=\sigma^2\exp\left\{-(x_1-x_2)^2\right\}$, where $\sigma$ is unknown. 
%We assume that the true data-generating distribution
%is $y_{ij}\sim Poisson(\lambda(x_i))$, with $\lambda(x)=\exp(\alpha_0+\beta_0(x))$. We generate the data by simulating $\alpha_0\sim Uniform(-1,1)$, $\beta_0\sim Uniform(-1,1)$
%and $x_i\sim Uniform(-1,1)$; $i=1,\ldots,n$, and then finally simulating $y_{ij}\sim Poisson(\lambda(x_i))$; $j=1,\ldots,m$, $i=1,\ldots,n$.
%
We reparameterize $\sigma^2$ as $\exp(\omega)$, where $-\infty<\omega<\infty$. 
For the prior on the parameters, we set $\pi\left(\alpha,\beta,\omega\right)=1$, for $-\infty<\alpha,\beta,\omega<\infty$. 
%
%In the inverse case, for the reason of prior specification, we linearize $\eta(\tilde x_i)$ as $\alpha+\beta\tilde x_i$; see Section \ref{subsubsec:inverse_poisson_nonpara}.
%Hence, for comparability with the inverse counterpart, we set $\eta(x_i)=\alpha+\beta x_i$. Thus, in the forward case, 
%$\theta=(\alpha,\beta,\eta(x_1),\ldots,\eta(x_{i-1}),\eta(x_{i+1}),\ldots,\eta(x_n),\omega)$. 
%We obtain $\frac{1}{n}\sum_{i=1}^n\log\pi(y_{i1}|\bY_{nm,-i},\bX_n,\mathcal M)$ 
%using the same method of Monte Carlo averaging described in Section \ref{subsubsec:forward_poisson_linear}, where $\theta$ is again first generated using TMCMC
%from the full posterior of $\theta$ by discarding the first $10,000$ iterations and retaining the next $20,000$ for inference, which are re-used to approximate
%the desired posteriors $\pi(\theta|\bY_{nm,-i},\bX_{n,-i},\mathcal M)$. As before, we obtain Monte Carlo averages over $100,000$ realizations of $\theta$.
%
%
%\subsubsection{Inverse Poisson nonparametric regression model}
%\label{subsubsec:inverse_poisson_nonpara}
%The model in this case remains the same as that in Section \ref{subsubsec:forward_poisson_nonpara}, but now a prior on $\tilde x_i$ is needed.
Note that the prior for $\tilde x_i$, which is uniform on 
$B_{im}(\eta)=\left\{x:\eta(x)\in \log\left\{\left[\bar y_i-\frac{c_1s_i}{\sqrt{m}},\bar y_i+\frac{c_2s_i}{\sqrt{m}}\right]\right\}\right\}$, does not have a closed form,
since the form of $\eta(x)$ is unknown. However, if $m$ is large, the interval 
$\log\left\{\left[\bar y_i-\frac{c_1s_i}{\sqrt{m}},\bar y_i+\frac{c_2s_i}{\sqrt{m}}\right]\right\}$ is small, and $\eta(x)$ falling in this small interval
can be reasonably well-approximated by a straight line. Hence, we set $\eta(x)=\mu(x)=\alpha+\beta x$, for $\eta(x)$ falling in this interval.
Thus it follows that $\pi(\tilde x_i|\eta)\equiv U(a,b)$, where $a$ and $b$ are given by (\ref{eq:a1}) and (\ref{eq:b1}), respectively.
%Hence, we obtain the same prior for $\tilde x_i$ as in the case of linear Poisson regression described in Section \ref{subsubsec:inverse_poisson_linear}.
As before we set $c_1=1$ and $c_2=100$.

%The method for obtaining $\frac{1}{n}\sum_{i=1}^n\log\pi(y_{i1}|\bY_{nm,-i},\bX_{n,-i},\mathcal M)$ remains the same as discussed in Section 
%\ref{subsubsec:inverse_poisson_linear}.

\subsubsection{Inverse geometric logit and probit linear and Gaussian process regression models}
\label{subsubsec:geometric_models}
%lso report results of our simulation experiments where data generated from Poisson linear regression is modeled by geometric regression models of the form
We also model the data by geometric models of the form
\begin{equation}
	f(y_{ij}|\theta,x_i)=(1-p(x_i))^{y_{ij}}p(x_i),
	\label{eq:geo1}
\end{equation}
where $p(x_i)$ is modeled as logit or probit linear or nonparametric regression having the following forms: 
%. In other words, we consider the following possibilities of modeling $p(x)$:
\begin{align}
	&\log\left(\frac{p(x)}{1-p(x)}\right)=\alpha+\beta x;~\log\left(\frac{p(x)}{1-p(x)}\right)=\eta(x);\notag\\
	&p(x)=\Phi\left(\alpha+\beta x\right);~p(x)=\Phi\left(\eta(x)\right).\notag
\end{align}
In the above, $\Phi$ is the cumulative distribution function of the standard normal distribution and $\eta$ is modeled by a Gaussian process with
mean function $\mu(x)=\alpha+\beta x$ and covariance function given by $Cov(\eta(x_1),\eta(x_2))=\sigma^2\exp\left\{-(x_1-x_2)^2\right\}$.
As before, we set $\sigma^2=\exp(\omega)$, where $-\infty<\omega<\infty$, and consider the improper prior $\pi(\alpha,\beta,\omega)=1$ for $-\infty<\alpha,\beta,\omega<\infty$.

We assign prior on $\tilde x_i$ such that the mean of the geometric distribution, namely, $\frac{1-p(x)}{p(x)}$, lies in 
$\left[\bar y_i-\frac{c_1s_i}{\sqrt{m}},\bar y_i+\frac{c_2s_i}{\sqrt{m}}\right]$. The same principles as before shows that for the logit link, either
for linear or Gaussian process regression, the prior for $\tilde x_i$ is $U(a_1,b_1)$, where  
\begin{equation}
a_1=\min\left\{-\beta^{-1}\left(\log\left(\bar y_i-\frac{c_1s_i}{\sqrt{m}}\right)+\alpha\right),
-\beta^{-1}\left(\log\left(\bar y_i+\frac{c_2s_i}{\sqrt{m}}\right)+\alpha\right)\right\}
	\label{eq:a2}
\end{equation}
and 
\begin{equation}
b_1=\max\left\{-\beta^{-1}\left(\log\left(\bar y_i-\frac{c_1s_i}{\sqrt{m}}\right)+\alpha\right),
-\beta^{-1}\left(\log\left(\bar y_i+\frac{c_2s_i}{\sqrt{m}}\right)+\alpha\right)\right\}.
	\label{eq:b2}
\end{equation}
We set $c_1=1$ and $c_2=100$, as before.

In the case of geometric probit regression, let us first define $\ell_{im}=\bar y_i-\frac{c_1s_i}{\sqrt{m}}$ and $u_{im}=\bar y_i+\frac{c_2s_i}{\sqrt{m}}$. Then with
\begin{align}
	a_2&=\min\left\{\frac{\Phi^{-1}\left(\frac{1}{u_{im}+1}\right)-\alpha}{\beta},\frac{\Phi^{-1}\left(\frac{1}{\ell_{im}+1}\right)-\alpha}{\beta}\right\};\label{eq:a3}\\
	b_2&=\max\left\{\frac{\Phi^{-1}\left(\frac{1}{u_{im}+1}\right)-\alpha}{\beta},\frac{\Phi^{-1}\left(\frac{1}{\ell_{im}+1}\right)-\alpha}{\beta}\right\}.\label{eq:b3}
\end{align}
the prior for $\tilde x_i$, for both linear and Gaussian process based geometric probit regression, is $U(a_2,b_2)$.

%The rest of the methodology for computing FPBF and IPBF for geometric regression remains the same as for Poisson regression described in Section \ref{subsec:poisson_models}.

\subsection{Implementation of our multiple testing procedure for inverse model selection}
\label{subsec:implementation}

We now briefly discuss our strategy for implementing our multiple testing procedure for hypotheses (\ref{eq:H0_1}) and (\ref{eq:H1_1}). We set $\tilde\Theta_k$
to $\Theta_k$, so we shall denote $\pi(\tilde x_i|\bX_{n,-i},\bY_{nm},\mathcal M_k,\tilde\Theta_k)$ by $\pi(\tilde x_i|\bX_{n,-i},\bY_{nm},\mathcal M_k)$.

\subsubsection{Obtaining the posterior distributions of the discrepancy measures using IRMCMC and TMCMC}
\label{subsubsec:disc_posterior}
For each competing model $\mathcal M_k$; $k=1,\ldots,K$, we obtain samples from the cross-validation posterior distribution 
$\pi(\tilde x_i|\bX_{n,-i},\bY_{nm},\mathcal M_k)$,
for $i=1,\ldots,n$, using fast and efficient IRMCMC. %Importance Re-sampling MCMC (IRMCMC) proposed by \ctn{Bhatta07}.
The key idea is to first generate realizations of size $N$ from some appropriate ``importance sampling density" 
of the form $\pi(\tilde x_{i^*},\theta_k|\bX_{n,-i^*},\bY_{nm},\mathcal M_k)$, for some $i^*\in\{1,\ldots,n\}$  
using TMCMC. Note that a major advantage of TMCMC over regular MCMC is that it effectively reduces the dimensionality of the parameters 
to a single dimension, thus drastically improving the acceptance rate and computational
speed, while ensuring good mixing properties at the same time. %transformation based Markov chain Monte Carlo (TMCMC) (\ctn{Dutta13}). 
Appropriate choice of $i^*$, which is equivalent to appropriate choice of the importance sampling density, has been proposed in \ctn{Bhatta07}.
For $i\in\{1,\ldots,n\}$, 
a sub-sample of the realizations of $\theta_k$ (but not of $\tilde x_{i^*}$) of size $M~(<N)$ is selected without replacement with importance weights 
proportional to the ratio of 
$\pi(\tilde x_{i},\theta_k|\bX_{n,-i},\bY_{nm},\mathcal M_k)$ and $\pi(\tilde x_{i^*},\theta_k|\bX_{n,-i^*},\bY_{nm},\mathcal M_k)$.
For each member $\theta_k$ of the sub-sampled realizations, 
$R$ realizations of $\tilde x_i$ are generated using TMCMC from $\pi(\tilde x_{i}|\theta_k,\bX_{n,-i},\bY_{nm},\mathcal M_k)$, to yield a total of
$R\times M$ realizations from $\pi(\tilde x_i|\bX_{n,-i},\bY_{nm},\mathcal M_k)$.

In our examples, we generate $30,000$ TMCMC samples from $\pi(\tilde x_{i^*},\theta_k|\bX_{n,-i^*},\bY_{nm},\mathcal M_k)$
of which we discard the first $10,000$ as burn-in, and re-sample $1000$ $\theta_k$-realizations without replacement from the remaining $20,000$ realizations
with importance weights proportional to the ratio of $\pi(\tilde x_{i},\theta_k|\bX_{n,-i},\bY_{nm},\mathcal M_k)$ and 
$\pi(\tilde x_{i^*},\theta_k|\bX_{n,-i^*},\bY_{nm},\mathcal M_k)$. 
For each re-sampled $\theta_k$-value, we generate $100$ TMCMC realizations of $\tilde x_i$. We discard the first $10,000$ realizations of $\tilde x_i$ as burn-in for the first
re-sampled $\theta_k$-realization, and for the subsequent $\theta_k$-realizations, we set the final value of $\tilde x_i$ of the previous value of $\theta_k$ as the initial value
for $\tilde x_i$ given the current $\theta_k$-value, and continue TMCMC without any further burn-in. We thus obtain $1000\times 100=100,000$ realizations of $\tilde x_i$
for each $i=1,\ldots,n$. In all our examples, the above IRMCMC strategy, in conjunction with efficient implementation of additive TMCMC, has led to excellent mixing
properties.

Using the $100,000$ IRMCMC samples, we obtain the posterior distribution of any given discrepancy measure $T^{(k)}(\tilde\bX_n)$.

\subsubsection{Obtaining the posterior model probabilities using Gibbs sampling}
\label{subsubsec:posterior_model_probs}

To obtain the posterior distribution of $\zeta$, we first need to specify a prior for $(p_1,\ldots,p_K)$. We consider the Dirichlet prior with parameters
$(\alpha_1,\ldots,\alpha_K)$, where $\alpha_k>0$, for $k=1,\ldots,K$. Given $\zeta$, the posterior distribution of $(p_1,\ldots,p_K)$ is again a Dirichlet
distribution with parameters $(\alpha_1+I(\zeta=1),\ldots,\alpha_K+I(\zeta=K))$. In other words,
\begin{equation}
	\pi(p_1,\ldots,p_K|\bX_n,\bY_{nm},\zeta)\equiv Dirichlet(\alpha_1+I(\zeta=1),\ldots,\alpha_K+I(\zeta=K)).
	\label{eq:dir}
\end{equation}
Given $(p_1,\ldots,p_K)$, the posterior distribution of $\zeta$ is given by (\ref{eq:post_zeta}), which is a function of the Bayes factors 
$BF^{(nm)}(\mathcal M_k,\mathcal M_{\tilde k})$; $k=1,\ldots,K$. \ctn{Chat20a} have shown that the corresponding pseudo-Bayes factors 
$PBF^{(nm)}(\mathcal M_k,\mathcal M_{\tilde k})$; $k=1,\ldots,K$, have the same asymptotic properties as the Bayes factors and are computationally far more efficient.
Moreover, unlike Bayes factors, pseudo-Bayes factors do not suffer from Lindley's paradox. 
%(see \ctn{Jeffreys39}, \ctn{Lindley57}, \ctn{Bartlett57}, \ctn{Robert93}, \ctn{Villa15})
Thus, it seems reasonable to replace $BF^{(nm)}(\mathcal M_k,\mathcal M_{\tilde k})$ in (\ref{eq:post_zeta}) with the corresponding
$PBF^{(nm)}(\mathcal M_k,\mathcal M_{\tilde k})$. In other words, we approximate the posterior probability $\pi(\zeta=k|\bX_n,\bY_{nm},p_1,\ldots,p_K)$ as
\begin{equation}
	\pi(\zeta=k|\bX_n,\bY_{nm},p_1,\ldots,p_K)
	\approx \frac{p_kPBF^{(nm)}(\mathcal M_k,\mathcal M_{\tilde k})}{\sum_{\ell=1}^Kp_\ell PBF^{(nm)}(\mathcal M_{\ell},\mathcal M_{\tilde k})};~k=1,\ldots,K.
	\label{eq:post_zeta3}
\end{equation}
Since the model probabilities are associated with the forward part, that is, where all the covariate values are treated as fixed, we consider the forward,
or the traditional pseudo-Bayes factor in (\ref{eq:post_zeta3}).
In our examples, the values of $PBF^{(nm)}(\mathcal M_k,\mathcal M_{\tilde k})$; $k=1,\ldots,K$, are already available from \ctn{Chat20a}
who provide estimates of $\frac{1}{n}\sum_{i=1}^n\log\pi(y_{i1}|\bY_{nm,-i},\bX_n,\mathcal M_k)$ in the second last column of Table 9.1.
Note that $$\frac{1}{n}\log PBF^{(nm)}(\mathcal M_k,\mathcal M_{\tilde k})=\frac{1}{n}\sum_{i=1}^n\log\pi(y_{i1}|\bY_{nm,-i},\bX_n,\mathcal M_k)
-\frac{1}{n}\sum_{i=1}^n\log\pi(y_{i1}|\bY_{nm,-i},\bX_n,\mathcal M_{\tilde k}).$$
Here %$\mathcal M_{\tilde k}$ corresponds to the maximum value of 
$\tilde k=\underset{k=1,\ldots,K}{\arg\max}~\frac{1}{n}\sum_{i=1}^n\log\pi(y_{i1}|\bY_{nm,-i},\bX_n,\mathcal M_k)$.

Using the full conditional distributions (\ref{eq:dir}) and (\ref{eq:post_zeta3}), we obtain $100,000$ realizations from the posterior
distribution of $(\zeta,p_1,\ldots,p_K)$ using Gibbs sampling, after discarding the first $10,000$ iterations as burn-in.

\subsubsection{Obtaining the posterior probabilities of the alternative hypotheses $H_{1k}$}
\label{subsubsec:posterior_alternatives}
Note that for $k=1,\ldots,K$, the posterior probability of $H_{1k}$ is given by
\begin{align}
	%v_{knm}=1-\pi\left(\zeta=k,T^{(k)}(\tilde \bX_n)-T^{(k)}(\bX_n)\in [\tilde\ell_{nk}-a_k-\varepsilon,\tilde u_{nk}-a_k+\varepsilon] |\bX_n,\bY_{nm}\right)
	v_{knm}&=1-\pi\left(\zeta=k,T^{(k)}(\tilde \bX_n)-T^{(k)}(\bX_n)\in [\tilde\ell_{knm},\tilde u_{knm}] \big |\bX_n,\bY_{nm}\right)\notag\\
	&=1-\pi\left(\zeta=k\big |\bX_n,\bY_{nm}\right)\pi\left(T^{(k)}(\tilde \bX_n)-T^{(k)}(\bX_n)\in [\tilde\ell_{knm},\tilde u_{knm}] \big |\zeta=k,\bX_n,\bY_{nm}\right).
	\label{eq:mc_average}
\end{align}
Once we obtain realizations from the posteriors of $T^{(k)}(\tilde\bX_n)$ for $k=1,\ldots,K$, and $(\zeta,p_1,\ldots,p_K)$, evaluation of 
the posterior probabilities of $H_{1k}$, denoted by $v_{knm}$; $k=1,\ldots,K$, follows simply by Monte Carlo averaging associated with the two factors
of (\ref{eq:mc_average}).

\subsection{Results of the simulation experiment for model selection}
\label{subsec:results_model}

\subsubsection{Non-misspecified situation}
\label{subsubsec:nonmiss_ms}
Section \ref{subsec:true_competing} shows that for this experiment, $K=6$, when no misspecification is considered. 
We set $\alpha_k=1$; $k=1,\ldots,K$, for the parameters of the Dirichlet prior for $(p_1,\ldots,p_K)$. That is, we assume a uniform prior distribution for 
$(p_1,\ldots,p_K)$ on the simplex. We report our results with respect to this prior, but our experiments with other values of $(\alpha_1,\ldots,\alpha_K)$
did not yield different results.

For $n=m=10$, the $cFDR_{nm}$ and $cFNR_{nm}$, for $\beta_{nm}\in [0.01,0.99]$ are provided in Figure \ref{fig:error_rates1}.
The red and green colours correspond to $T^{(k)}_1(\tilde\bX_n)-T^{(k)}_1(\bX_n)$ and $T^{(k)}_2(\tilde\bX_n)-T^{(k)}_2(\bX_n)$, respectively.
In the plots we denote these red and green coloured cFDRs as cFDR1 and cFDR2, respectively. Similarly, cFNR1 and cFNR2 denote the red and green coloured cFNRs. 
When $T^{(k)}_1(\tilde\bX_n)-T^{(k)}_1(\bX_n)$ is considered, $cFDR_{nm}=0.024$ for $\beta_{nm}<0.86$ and equals $9.023\times 10^{-6}$ for $\beta_{nm}\geq 0.86$.
On the other hand, for $T^{(k)}_2(\tilde\bX_n)-T^{(k)}_2(\bX_n)$, $cFDR_{nm}=0.087$ for $0.01\leq\beta_{nm}<0.48$ and falls to $5.444\times 10^{-5}$ for 
$0.48\leq\beta_{nm}\leq 0.99$. 
In the first case, the multiple testing procedure selects $H_{1k}$ for $k=1,\ldots,K$ when $0.01\leq\beta_{nm}<0.86$. When $0.86<\beta_{nm}\leq 0.99$, the method
selects $H_{0\tilde k}$ and $H_{1k}$ for $k\neq\tilde k$. Here $\tilde k$ corresponds to the true data-generating model, namely, the Poisson log-linear regression model.
In the second case, all the alternative hypotheses are selected when $0.01\leq \beta_{nm}<0.48$; the true null and remaining alternative hypotheses are chosen for 
$0.48\leq\beta_{nm}\leq 0.99$.
Thus, for both the discrepancy measures, the correct model is selected for appropriate values of $\beta_{nm}$. However, cFDR2 falls close to zero much faster than cFDR1,
and from the point onwards where the true decision occurs, cFNR2 is much lesser than cFNR1. These demonstrate that $T^{(k)}_2(\tilde\bX_n)-T^{(k)}_2(\bX_n)$ is a more efficient
choice compared to $T^{(k)}_1(\tilde\bX_n)-T^{(k)}_1(\bX_n)$.

Here is an important point regarding comparison with our multiple testing result with that of inverse pseudo-Bayes factor reported in the last column of Table 9.1
of \ctn{Chat20a}. The column shows that the inverse pseudo-Bayes factor identifies the true Poisson log-linear regression model as only the second best.
However our multiple testing procedure correctly identifies the true model as the best one, for appropriate values of $\beta_{nm}$.

It is also important to remark in this context that the posterior probabilities of $T^{(k)}(\tilde\bX_n)-T^{(k)}(\bX_n)\in[\tilde\ell_{knm},\tilde u_{knm}]$
when $k$ is the true model, is significantly smaller than several other models. That the true model still turns out to be the best is due to its
much larger posterior model probability compared to the others. The point is that even the true data-generating model need not have large posterior probabilities associated with
the inverse discrepancy measure, and if the corresponding posterior model probability is not significantly large, then any other model can turn out to be the best
on the basis of its stronger inverse perspective.
\begin{figure}
	\centering
	\subfigure [$cFDR_{nm}$ as a function of $\beta_{nm}$]{ \label{fig:cfdr1}
	\includegraphics[width=7.5cm,height=7.5cm]{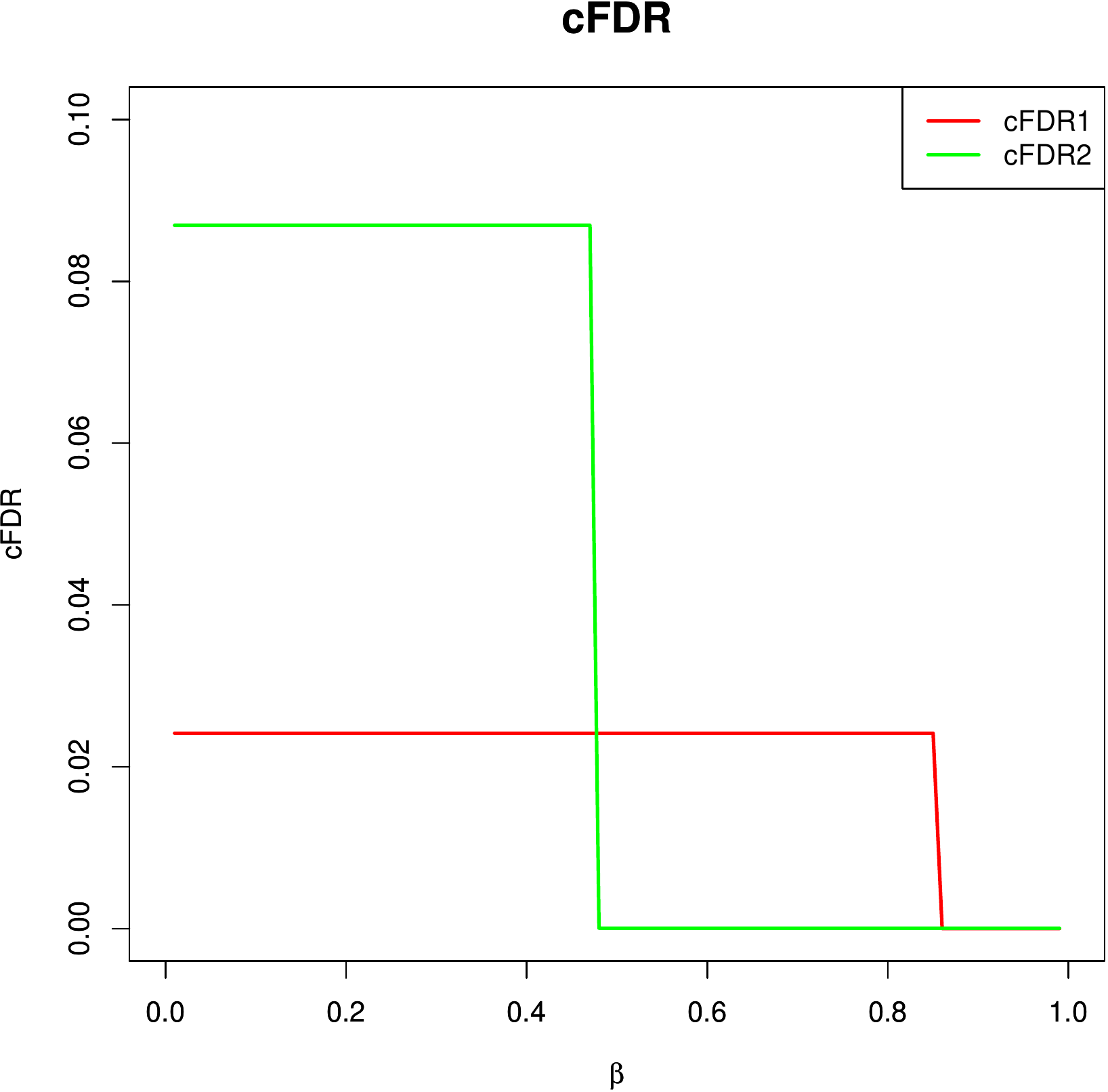}}
	\hspace{2mm}
	\subfigure [$cFNR_{nm}$ as a function of $\beta_{nm}$]{ \label{fig:cfnr1}
	\includegraphics[width=7.5cm,height=7.5cm]{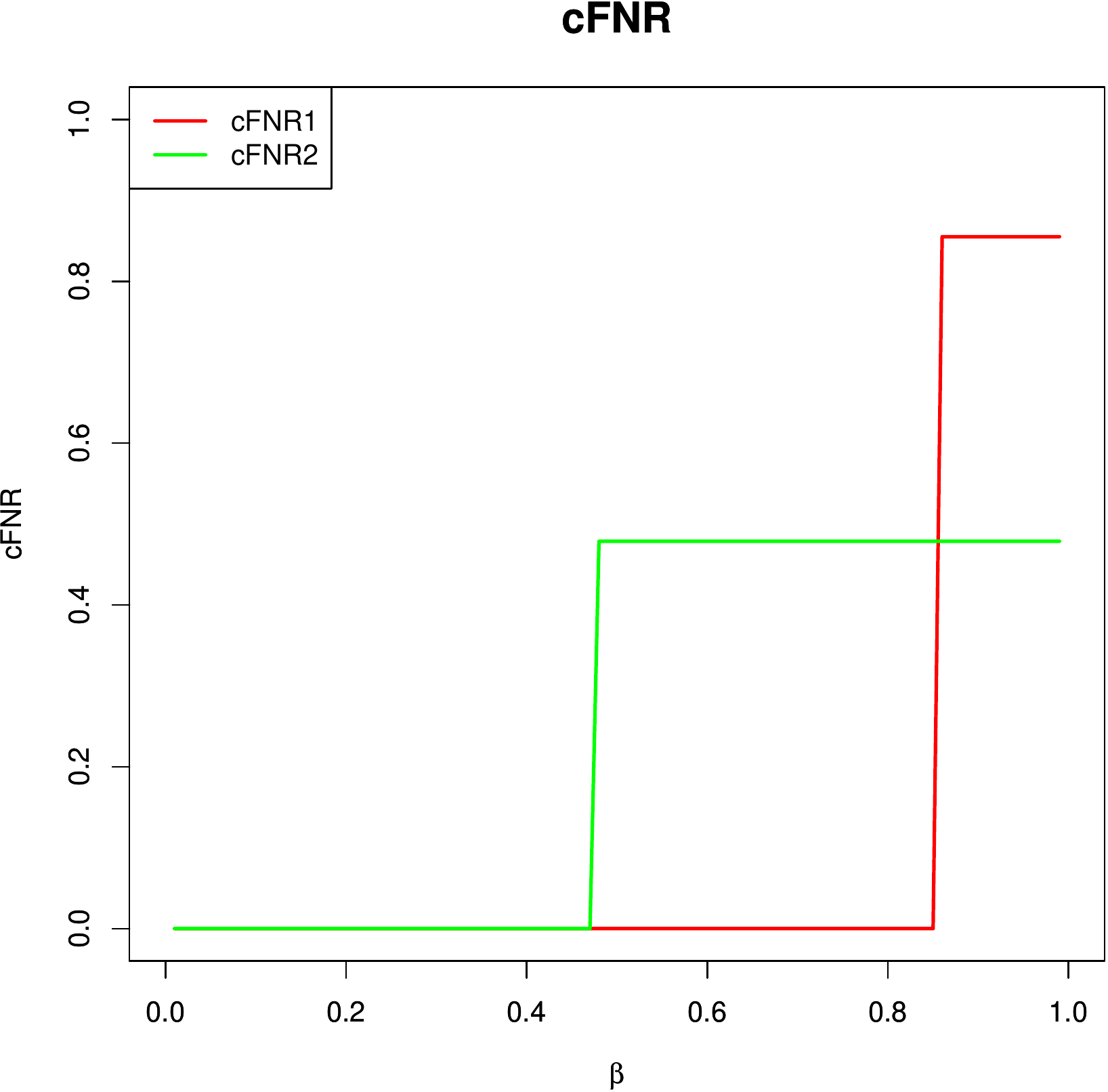}}
	\caption{$cFDR_{nm}$ and $cFNR_{nm}$ as functions of $\beta_{nm}$ in the non-misspecified case.}
	\label{fig:error_rates1}
\end{figure}

\subsubsection{Misspecified situation}
\label{subsubsec:miss_ms}
Let us now consider the case of misspecification, that is, when the true Poisson log-linear model is left out from consideration among the competing models. Thus, $K=5$
in this case. The remaining setup is the same as in the non-misspecified scenario. Figure \ref{fig:error_rates2} display the cFDRs and cFNRs for this situation, each associated
with both $T^{(k)}_1(\tilde\bX_n)-T^{(k)}_1(\bX_n)$ and $T^{(k)}_2(\tilde\bX_n)-T^{(k)}_2(\bX_n)$. In this case, for both the discrepancy measures, the correct decision,
namely, the null hypothesis for the Poisson log-Gaussian process and the alternative hypotheses for the remaining models,
is reached for relatively large values of $\beta_{nm}$. 
Indeed, cFDR1 $= 0.002$ for $0.01\leq \beta_{nm}<0.99$ and $0.0003$ for $\beta_{nm}=0.99$ and cFDR2 $= 0.020$ for $0.01\leq \beta_{nm}<0.91$ 
and $0.0003$ for $0.91\leq \beta_{nm}\leq 0.99$. Again, $T^{(k)}_2(\tilde\bX_n)-T^{(k)}_2(\bX_n)$ performs better than $T^{(k)}_1(\tilde\bX_n)-T^{(k)}_1(\bX_n)$
in terms of faster decrease of $cFDR_{mn}$ towards zero and lesser value of $cFNR_{nm}$ once the right decision has been obtained.

Here the multiple testing procedure turns out to be consistent with both forward and inverse pseudo-Bayes factor, since the last two columns of Table 9.1 of \ctn{Chat20a} show 
that if the Poisson log-linear model is not considered among the competing models, then the Poisson log-Gaussian process model is the best.
Here the corresponding posterior probability of $T^{(k)}(\tilde\bX_n)-T^{(k)}(\bX_n)\in[\tilde\ell_{knm},\tilde u_{knm}]$ is higher than those of the other models,
in addition to higher posterior model probability.
\begin{figure}
	\centering
	\subfigure [$cFDR_{nm}$ as a function of $\beta_{nm}$]{ \label{fig:cfdr2}
	\includegraphics[width=7.5cm,height=7.5cm]{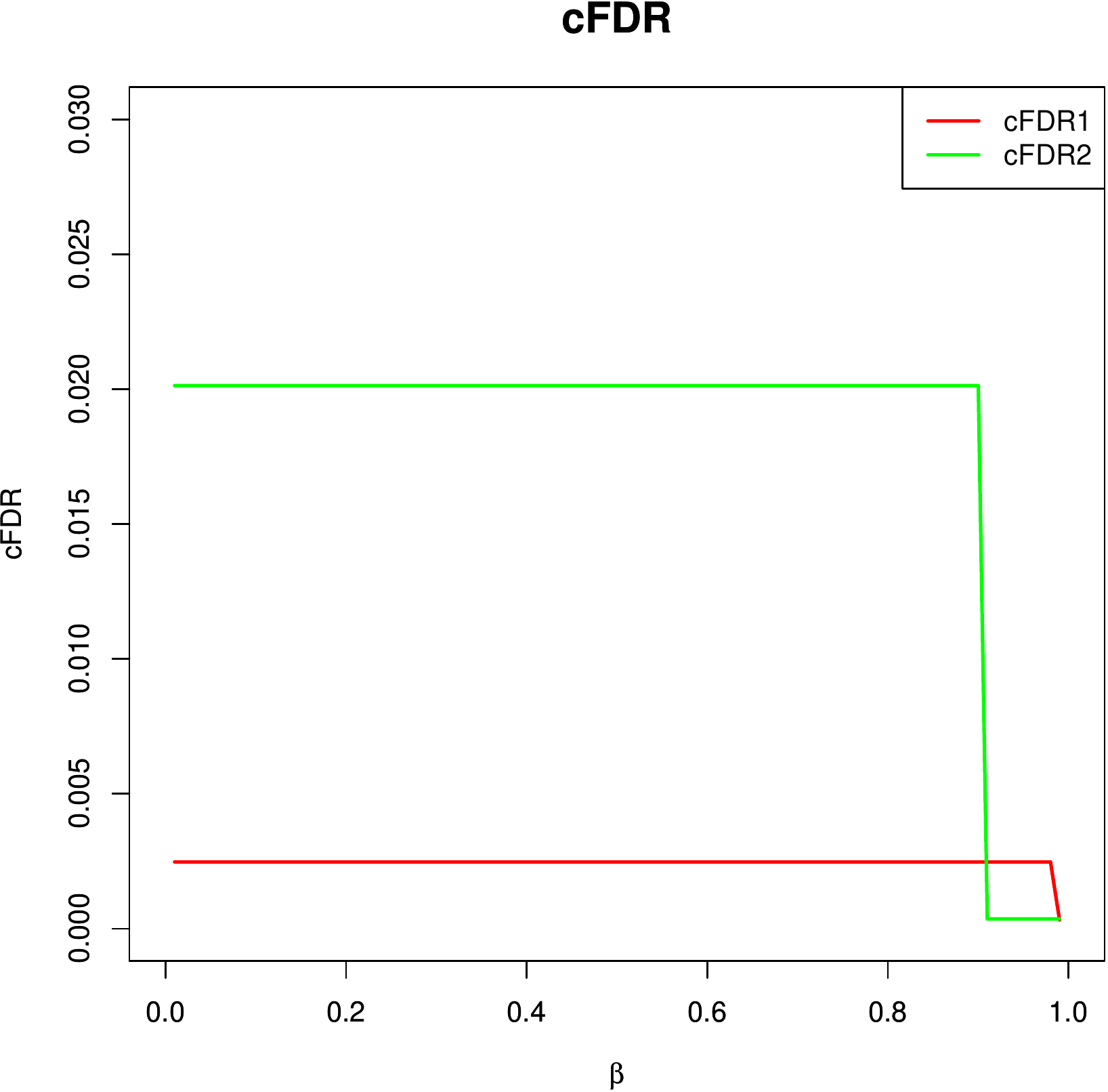}}
	\hspace{2mm}
	\subfigure [$cFNR_{nm}$ as a function of $\beta_{nm}$]{ \label{fig:cfnr2}
	\includegraphics[width=7.5cm,height=7.5cm]{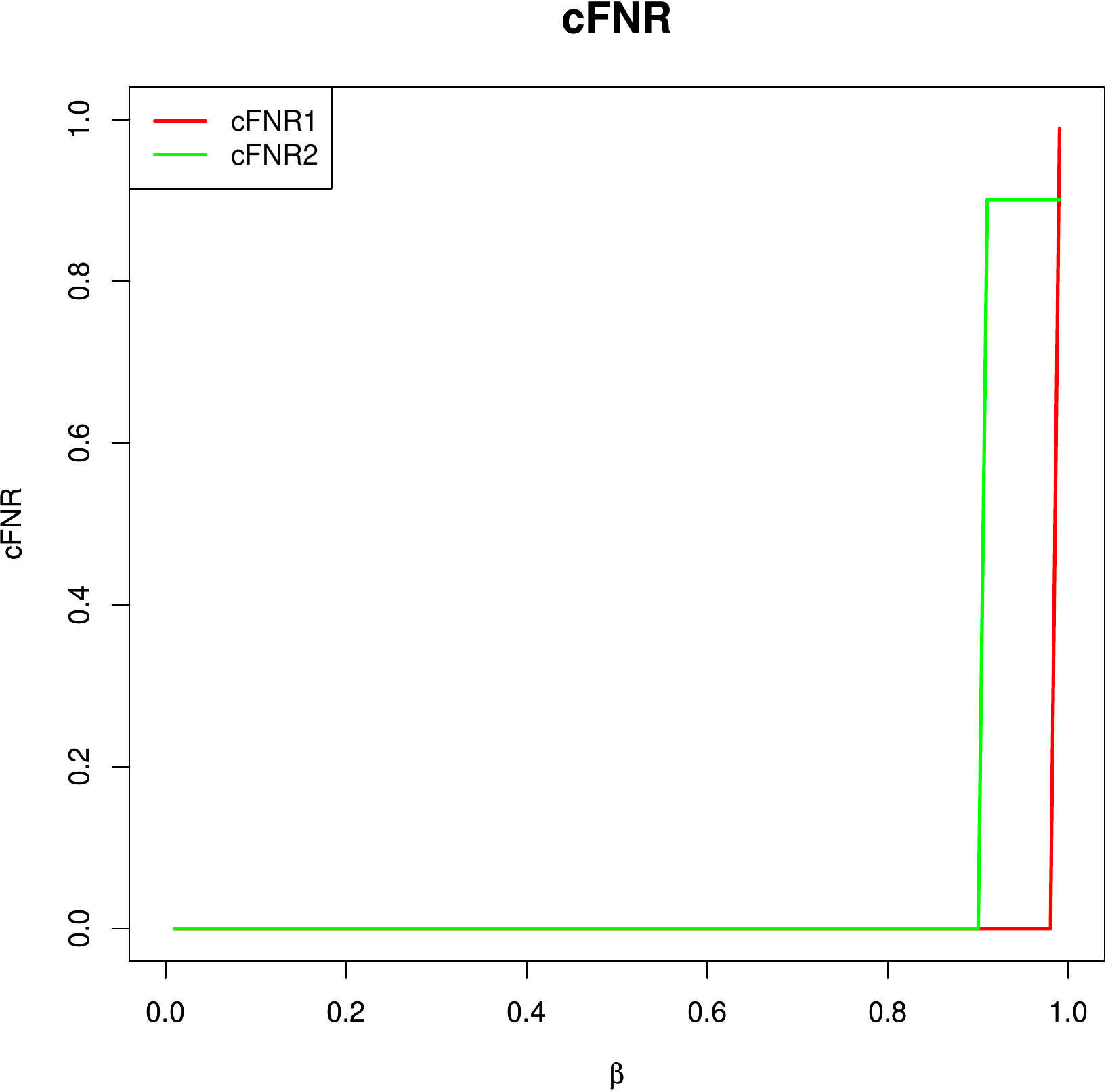}}
	\caption{$cFDR_{nm}$ and $cFNR_{nm}$ as functions of $\beta_{nm}$ in the misspecified case.}
	\label{fig:error_rates2}
\end{figure}

\section{Second simulation study: variable selection in Poisson and  geometric linear and nonparametric regression models when true model is Poisson linear regression}
\label{sec:simstudy_vs}

Again, for the purpose of making this article as self-contained as possible, we begin with brief descriptions of the true and competing inverse regression models
in the variable selection context.

We now consider covariates $x$ and $z$, where the
true data-generating distribution
is $y_{ij}\sim Poisson(\lambda(x_i,z_i))$, with $\lambda(x,z)=\exp(\alpha_0+\beta_0 x+\gamma_0 z)$. 
The data is generated as follows. We simulate $\alpha_0,\beta_0,\gamma_0\sim U(-1,1)$ independently
and $x_i\sim U(-1,1)$, $z_i\sim U(0,2)$; $i=1,\ldots,n$, independently. Finally, we generate $y_{ij}\sim Poisson(\lambda(x_i,z_i))$; $j=1,\ldots,m$, $i=1,\ldots,n$,
independently.

As in \ctn{Chat20a} we model the data $y_{ij}$; $i=1,\ldots,n$; $j=1,\ldots,m$ with both Poisson and geometric models letting the regression part consist of
either $x$ or $z$, or both. We denote the linear regression coefficients of the intercept, $x$ and $z$ as $\alpha$, $\beta$ and $\gamma$, respectively,
and give the improper prior density $1$ to $(\alpha,\beta)$, $(\alpha,\gamma)$ and $(\alpha,\beta,\gamma)$ when the models consist of these combinations of parameters. 
For Gaussian process regression with both $x$ and $z$, we let $\eta(x,z)$ be the regression function modeled by a Gaussian process with mean 
$\mu(x,z)=\alpha+\beta x+\gamma z$ and covariance function $Cov\left(\eta(x_1,z_1),\eta(x_2,z_2)\right)=\exp\left(\omega\right)
\exp\left[-\left\{(x_1-x_2)^2+(z_1-z_2)^2\right\}\right]$, and we assign prior mass $1$ to 
$(\alpha,\beta,\omega)$, $(\alpha,\gamma,\omega)$ and $(\alpha,\beta,\gamma,\omega)$ when the models consist of the covariates $x$, $z$ or both.
Where the model consists of the single covariate $x$ or $z$, the priors for $\tilde x_i$ and $\tilde z_i$ remain the same as
in the previous cases. 

But wherever the models consist of both the covariates $x$ and $z$, we need to assign priors for both $\tilde x_i$ and $\tilde z_i$,
%in addition to requiring that $E(y_{ij}|\theta,x_i,z_i)$ under the postulated model fall in $\left[\bar y_i-\frac{c_1s_i}{\sqrt{m}},\bar y_i+\frac{c_2s_i}{\sqrt{m}}\right]$.
%The 
and the same priors for $\tilde x_i$ and $\tilde z_i$ as the previous situations where the models consisted of single covariates, will not be consistent here.
%For consistent priors we adopt the following strategy.
Letting $\alpha$ be the intercept, $\beta$ and $\gamma$ the coefficients of $x_i$ and $z_i$ respectively in the regression forms, we consider the same  
consistent priors for $\tilde x_i$ and $\tilde z_i$ as proposed in \ctn{Chat20a}.
In Sections \ref{subsubsec:x_z_prior_poisson}, \ref{subsubsec:x_z_prior_geo_logit} and we provide the forms of the priors for $\tilde x_i$ and $\tilde z_i$
when the models consist of both the covariates $x$ and $z$.

\subsubsection{Prior for $\tilde x_i$ and $\tilde z_i$ for Poisson regression}
\label{subsubsec:x_z_prior_poisson}
For the Poisson linear or Gaussian process regression model with log link consisting of both the covariates $x$ and $z$, we set 
$\tilde x_i\sim U\left(a^{(1)}_x,b^{(1)}_x\right)$ and  $\tilde z_i\sim U\left(a^{(1)}_z,b^{(1)}_z\right)$, where
\begin{equation*}
	a^{(1)}_x=\min\left\{\beta^{-1}\left(\log\left(\bar y_i-\frac{c_1s_i}{\sqrt{m}}\right)-\alpha-\gamma z_i\right),
\beta^{-1}\left(\log\left(\bar y_i+\frac{c_2s_i}{\sqrt{m}}\right)-\alpha-\gamma z_i\right)\right\},
	%\label{eq:ax1}
\end{equation*}
 
\begin{equation*}
	b^{(1)}_x=\max\left\{\beta^{-1}\left(\log\left(\bar y_i-\frac{c_1s_i}{\sqrt{m}}\right)-\alpha-\gamma z_i\right),
\beta^{-1}\left(\log\left(\bar y_i+\frac{c_2s_i}{\sqrt{m}}\right)-\alpha-\gamma z_i\right)\right\},
	%\label{eq:bx1}
\end{equation*}

\begin{equation*}
	a^{(1)}_z=\min\left\{\gamma^{-1}\left(\log\left(\bar y_i-\frac{c_1s_i}{\sqrt{m}}\right)-\alpha-\beta x_i\right),
\gamma^{-1}\left(\log\left(\bar y_i+\frac{c_2s_i}{\sqrt{m}}\right)-\alpha-\beta x_i\right)\right\}
	%\label{eq:az1}
\end{equation*}
and 
\begin{equation*}
	b^{(1)}_z=\max\left\{\gamma^{-1}\left(\log\left(\bar y_i-\frac{c_1s_i}{\sqrt{m}}\right)-\alpha-\beta x_i\right),
\gamma^{-1}\left(\log\left(\bar y_i+\frac{c_2s_i}{\sqrt{m}}\right)-\alpha-\beta x_i\right)\right\}.
	%\label{eq:bz1}
\end{equation*}
%Note that the priors for $\tilde x_i$ and $\tilde z_i$ depend upon $z_i$ and $x_i$ respectively.
%This is somewhat in keeping with (\ref{eq:B2}) where the prior for $\tilde x_i$ depends upon $x_i$ itself. 
%The discussion following (\ref{eq:B2}) is enough to justify that the priors for $\tilde x_i$ and $\tilde z_i$ in the current situation do make sense, apart from
%ensuring consistency. 
%It is unusual in Bayesian inference to make the prior depend upon the truth. 
%Indeed, the true parameter is always unknown; had it been known, then one would give full prior probability to the true
%parameter. In our case $x_i$ is actually known but a prior is needed for $\tilde x_i$ for the sake of cross-validation. Moreover, the prior does not consider 
%$x_i$ to be known as long as the sample sizes $n$ and $m$ remain finite and $\theta$ is unknown or takes false values. The prior has substantial variance
%in these cases. Hence, although unusual, such a prior on $\tilde x_i$ is not untenable for inverse cross-validation.

\subsubsection{Prior for $\tilde x_i$ and $\tilde z_i$ for geometric regression with logit link}
\label{subsubsec:x_z_prior_geo_logit}
For the geometric linear or Gaussian process regression model with logit link consisting of both the covariates $x$ and $z$, we set 
$\tilde x_i\sim U\left(a^{(2)}_x,b^{(2}_x\right)$ and  $\tilde z_i\sim U\left(a^{(2)}_z,b^{(2)}_z\right)$, where
\begin{equation*}
	a^{(2)}_x=\min\left\{-\beta^{-1}\left(\log\left(\bar y_i-\frac{c_1s_i}{\sqrt{m}}\right)+\alpha+\gamma z_i\right),
-\beta^{-1}\left(\log\left(\bar y_i+\frac{c_2s_i}{\sqrt{m}}\right)+\alpha+\gamma z_i\right)\right\},
	%\label{eq:ax2}
\end{equation*}

\begin{equation*}
	b^{(2)}_x=\max\left\{-\beta^{-1}\left(\log\left(\bar y_i-\frac{c_1s_i}{\sqrt{m}}\right)+\alpha+\gamma z_i\right),
-\beta^{-1}\left(\log\left(\bar y_i+\frac{c_2s_i}{\sqrt{m}}\right)+\alpha+\gamma z_i\right)\right\},
	%\label{eq:bx2}
\end{equation*}

\begin{equation*}
	a^{(2)}_z=\min\left\{-\gamma^{-1}\left(\log\left(\bar y_i-\frac{c_1s_i}{\sqrt{m}}\right)+\alpha+\beta x_i\right),
-\gamma^{-1}\left(\log\left(\bar y_i+\frac{c_2s_i}{\sqrt{m}}\right)+\alpha+\beta x_i\right)\right\}
	%\label{eq:az2}
\end{equation*}
and 
\begin{equation*}
	b^{(2)}_z=\max\left\{-\gamma^{-1}\left(\log\left(\bar y_i-\frac{c_1s_i}{\sqrt{m}}\right)+\alpha+\beta x_i\right),
-\gamma^{-1}\left(\log\left(\bar y_i+\frac{c_2s_i}{\sqrt{m}}\right)+\alpha+\beta x_i\right)\right\}.
	%\label{eq:bz2}
\end{equation*}

\subsubsection{Prior for $\tilde x_i$ and $\tilde z_i$ for geometric regression with probit link}
\label{subsubsec:x_z_prior_geo_probit}
For the geometric linear or Gaussian process regression model with probit link consisting of both the covariates $x$ and $z$, we set 
$\tilde x_i\sim U\left(a^{(3)}_x,b^{(3}_x\right)$ and  $\tilde z_i\sim U\left(a^{(3)}_z,b^{(3)}_z\right)$, where
%For geometric probit regression, first let $\ell_{im}=\bar y_i-\frac{c_1s_i}{\sqrt{m}}$ and $u_{im}=\bar y_i+\frac{c_2s_i}{\sqrt{m}}$. Let
\begin{equation*}
a^{(3)}_x=\min\left\{\frac{\Phi^{-1}\left(\frac{1}{u_{im}+1}\right)-\alpha-\gamma z_i}{\beta},
	\frac{\Phi^{-1}\left(\frac{1}{\ell_{im}+1}\right)-\alpha-\gamma z_i}{\beta}\right\},
	%\label{eq:ax3}
\end{equation*}

\begin{equation*}
b^{(3)}_x=\max\left\{\frac{\Phi^{-1}\left(\frac{1}{u_{im}+1}\right)-\alpha-\gamma z_i}{\beta},
	\frac{\Phi^{-1}\left(\frac{1}{\ell_{im}+1}\right)-\alpha-\gamma z_i}{\beta}\right\},
	%\label{eq:bx3}
\end{equation*}

\begin{equation*}
a^{(3)}_z=\min\left\{\frac{\Phi^{-1}\left(\frac{1}{u_{im}+1}\right)-\alpha-\beta x_i}{\gamma},
	\frac{\Phi^{-1}\left(\frac{1}{\ell_{im}+1}\right)-\alpha-\beta x_i}{\gamma}\right\}
	%\label{eq:az3}
\end{equation*}
and
\begin{equation*}
b^{(3)}_z=\max\left\{\frac{\Phi^{-1}\left(\frac{1}{u_{im}+1}\right)-\alpha-\beta x_i}{\gamma},
	\frac{\Phi^{-1}\left(\frac{1}{\ell_{im}+1}\right)-\alpha-\beta x_i}{\gamma}\right\}.
	%\label{eq:bz3}
\end{equation*}

\subsection{Discrepancy measure and Dirichlet prior parameters for more than one covariate}
\label{subsec:two_cov}

In models where both the covariates are considered, for any two $n$-dimensional vectors $\bv_{1n}=(v_{11},\ldots,v_{1n})$ and $\bv_n=(v_{21},\ldots,v_{2n})$, 
letting $\bv_i=(v_{1i},v_{2i})^T$, $\bV_n=(\bv_1,\ldots,\bv_n)$ 
and denoting the posterior mean vector and covariance matrix of $\tilde\bu_i=(\tilde x_i,\tilde z_i)^T$ 
by $E_{k}(\tilde\bu_i)$ and $Var_{k}(\tilde\bu_i)$
respectively, for $i=1,\ldots,n$,
we set 
\begin{equation}
T^{(k)}_3(\bV_n)=\frac{1}{n}\sum_{i=1}^n(\tilde\bv_i-E_{k}(\tilde\bu_i))^T\left(Var_{k}(\tilde\bu_i)+c\mathbb I\right)^{-1}(\tilde\bv_i-E_{k}(\tilde\bu_i)),
	\label{eq:T3}
\end{equation}
where $c>0$ and $\mathbb I$ is the identity matrix. Here $E_k(\tilde\bu_i)$ and $Var_k(\tilde\bu_i)$ correspond to the cross-validation posterior 
$\pi(\tilde\bu_i |\bX_{n,-i},\bY_{nm},\mathcal M_k)$.

In our experiment, as before we shall compare the results corresponding to $T^{(k)}_1(\tilde\bW_n)-T^{(k)}_1(\bW_n)$ and 
$T^{(k)}_2(\tilde\bW_n)-T^{(k)}_2(\bW_n)$, where $\tilde\bW_n$ is either $\tilde\bX_n$ or $\tilde\bZ_n$ and 
$\bW_n$ is either $\bX_n$ or $\bZ_n$. But for any inverse model that consists of both
the covariates $x$ and $z$, we replace both $T^{(k)}_1(\tilde\bW_n)-T^{(k)}_1(\bW_n)$ and $T^{(k)}_2(\tilde\bW_n)-T^{(k)}_2(\bW_n)$ 
with $T^{(k)}_3(\tilde\bV_n)-T^{(k)}_3(\bV_n)$, where $\tilde\bv_i=(\tilde x_{i},\tilde z_{i})^T$, $\tilde\bV_n=(\tilde\bv_1,\ldots,\tilde\bv_n)$, 
$\bv_i=(x_{i},z_{i})^T$ and $\bV_n=(\bv_1,\ldots,\bv_n)$.

For models having both $x$ and $z$ as covariates, the corresponding discrepancy measures
$T^{(k)}_3(\tilde\bV_n)-T^{(k)}_3(\bV_n)$ are associated with joint cross-validation posterior distributions of $(\tilde x_i,\tilde z_i)$, and hence the 
corresponding posterior probabilities of the hypotheses are expected to be much smaller than posterior probabilities of the hypotheses of the models with single
covariates. We make amends for this by setting the parameters $\alpha_k$ of the Dirichlet prior for $(p_1,\ldots,p_K)$ for any model $\mathcal M_k$
with both covariates to be $5$ times that of the remaining parameters. So, in our case, we set $\alpha_k=5$ for those $k$ associated with both the covariates, and
set the remaining parameters to $1$.

Note that in this experiment, $K=18$, including the true inverse Poisson log-linear regression model with both the covariates $x$ and $z$.
The implementation details remain the same as described in Section \ref{subsec:implementation}.

\subsection{Results of our multiple testing experiment for model and variable selection}
\label{subsec:results_vs}

\subsubsection{Non-misspecified situation}
\label{subsubsec:nonmiss_vs}
For $n=m=10$, when the true model is Poisson with log-linear regression on both the covariates $x$ and $z$, Figure \ref{fig:error_rates3} shows $cFDR_{nm}$ and
$cFNR_{nm}$ as functions of $\beta_{nm}$. 
In this case cFDR1 decreases towards zero slightly faster than cFDR2.
The numerical values of step functions cFDR1 and cFDR2 are provided as follows:
\begin{equation}
	cFDR1=\begin{cases}
       0.025 & \mbox{if}~0.01\leq\beta_{nm}<0.67;\\
       0.007 & \mbox{if}~0.67\leq\beta_{nm}<0.91;\\
       0.001 & \mbox{if}~0.91\leq\beta_{nm}<0.99;\\
	6.214\times 10^{-7} & \mbox{if}~\beta_{nm}=0.99
	\end{cases}
	\label{eq:cfdr1_vs}
\end{equation}
and
\begin{equation}
	cFDR2=\begin{cases}
       0.032 & \mbox{if}~0.01\leq\beta_{nm}<0.67;\\
       0.014 & \mbox{if}~0.67\leq\beta_{nm}<0.80;\\
       0.002 & \mbox{if}~0.80\leq\beta_{nm}<0.98;\\
	5.767\times 10^{-6} & \mbox{if}~0.98\leq\beta_{nm}\leq 0.99.
	\end{cases}
	\label{eq:cfdr2_vs}
\end{equation}
Note that the first change point for both cFDR1 and cFDR2 occurs at $\beta_{mn}=0.67$, and at this point, we obtain the decision configuration that selects
the null hypothesis of the true, Poisson log-linear model with both covariates $x$ and $z$, and alternative hypotheses of all other models.
For $\beta_{mn}<0.67$, for all the models, the alternative hypotheses are selected. Thus, the first change point associated with both cFDR1 and cFDR2 yields
the correct decision configuration. The next change points $\beta_{nm}=0.91$ and $\beta_{nm}=0.80$ for cFDR1 and cFDR2 are associated with selecting the
null hypothesis for the model with the Poisson log-linear model with covariate $x$, in addition to the null hypothesis of the true, 
Poisson log-linear model with both covariates $x$ and $z$. The final change points $\beta_{nm}=0.99$ and $\beta_{nm}=0.98$ yield the decision configurations
that select the null hypothesis for the model with the Poisson log-linear model with covariate $z$, in addition to the previous null hypotheses.
Thus, cFDR1 and cFDR2 behave quite consistently in this example and there seems to be no obvious reason for preferring one discrepancy measure to the other.
Observe in Figure \ref{fig:error_rates3} that cFNR1 and cFNR2 are also quite consistently behaved.

Again the important observation is that our multiple testing procedure seems to easily identify the true inverse model, while neither forward nor
inverse pseudo-Bayes factor successfully identified the true inverse model, as shown in the last two columns of Table 9.2 of \ctn{Chat20a}.
The second and third best models, namely, the Poisson log-linear model with covariate $x$ and the Poisson log-linear model with covariate $z$, respectively,
are however, consistent with forward and inverse pseudo-Bayes factor results reported in \ctn{Chat20a}.

Again we find that the posterior probabilities of $T^{(k)}(\tilde\bX_n)-T^{(k)}(\bX_n)\in[\tilde\ell_{knm},\tilde u_{knm}]$
when $k$ is the true model, is significantly smaller than most of the other models, but its much higher posterior model probability compared to the others
succeeds in making it the winner. The above inverse posterior probabilities for the second and third best models are also not higher than the remaining ones.  
%The point is that even the true data-generating model need not have large posterior probabilities associated with
%the inverse discrepancy measure, and if the corresponding posterior model probability is not significantly large, then any other model can turn out to be the best
%on the basis of its stronger inverse perspective.

\begin{figure}
	\centering
	\subfigure [$cFDR_{nm}$ as a function of $\beta_{nm}$]{ \label{fig:cfdr3}
	\includegraphics[width=7.5cm,height=7.5cm]{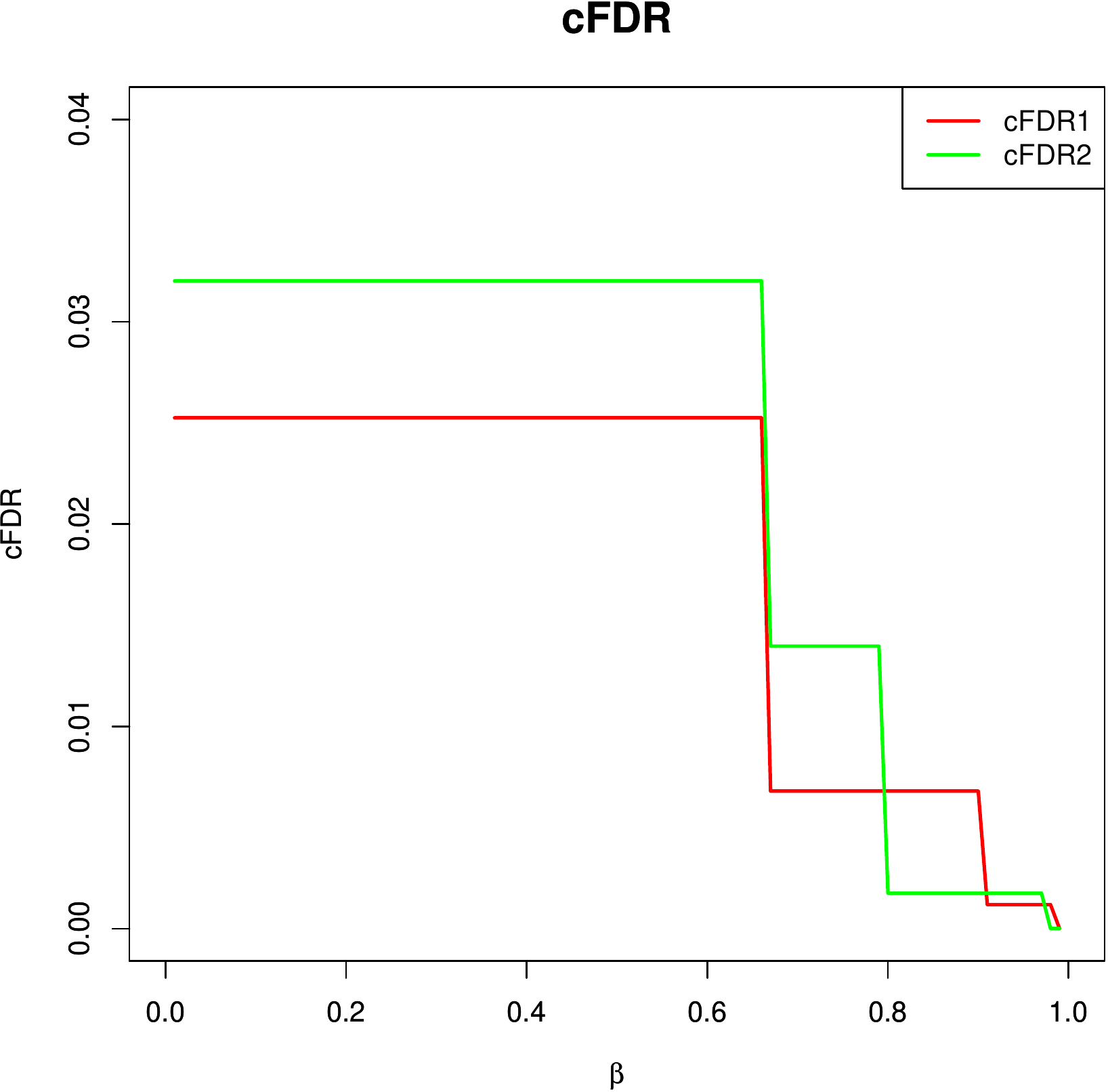}}
	\hspace{2mm}
	\subfigure [$cFNR_{nm}$ as a function of $\beta_{nm}$]{ \label{fig:cfnr3}
	\includegraphics[width=7.5cm,height=7.5cm]{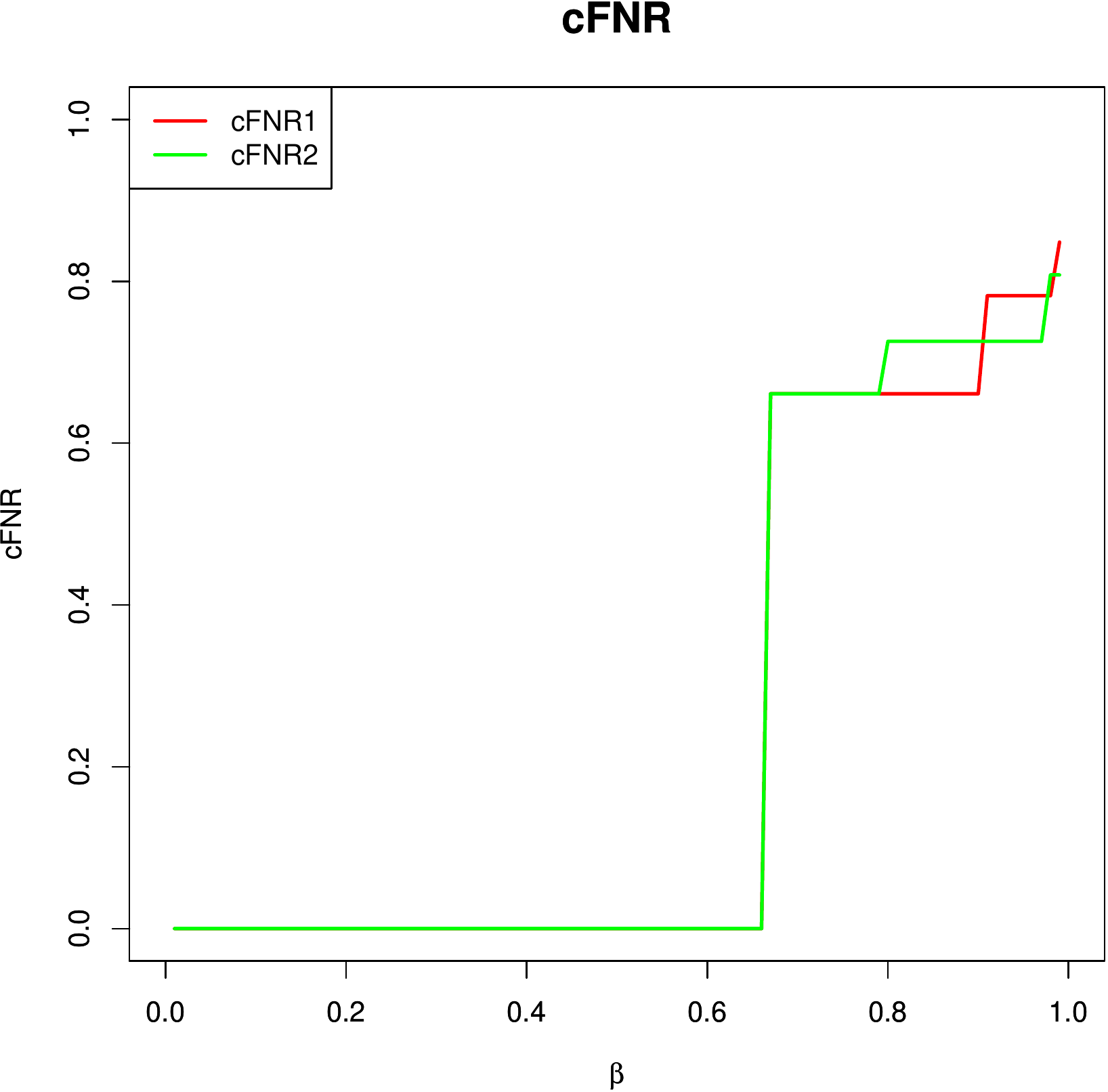}}
	\caption{$cFDR_{nm}$ and $cFNR_{nm}$ as functions of $\beta_{nm}$ in the non-misspecified situation of the model and variable selection problem.}
	\label{fig:error_rates3}
\end{figure}

\subsubsection{Misspecified situation}
\label{subsubsec:miss_vs}

In the misspecified situation we leave out the true Poisson log-linear model with both covariates $x$ and $z$ from among the competing models and implement 
our multiple testing procedure to obtain the best possible inverse models among the remaining ones. Figure \ref{fig:error_rates4} summarizes the results
of our implementation in this direction. Both cFDR1 and cFDR2 yield the Poisson log-linear model with covariate $x$ and the Poisson log-linear model with covariate $z$
as the best and the next best inverse models, corresponding to the two change points observed in the graphs of cFDR1 and cFDR2. 
Recall that these were the second and the third best models in the non-misspecified situation, showing that our results for this misspecified case is 
very much coherent.

Observe that the best model in this case
is detected by cFDR2 much earlier than cFDR1, and its value falls close to zero much earlier than that of cFDR1 in the process. The graphs for cFNR1 and cFNR2 shows
that at points where the best and the next best models are selected, cFNR2 is significantly smaller than cFNR1. Hence, in this misspecified situation, $T^{(k)}_2$
is again a better performer than $T^{(k)}_1$.

\begin{figure}
	\centering
	\subfigure [$cFDR_{nm}$ as a function of $\beta_{nm}$]{ \label{fig:cfdr4}
	\includegraphics[width=7.5cm,height=7.5cm]{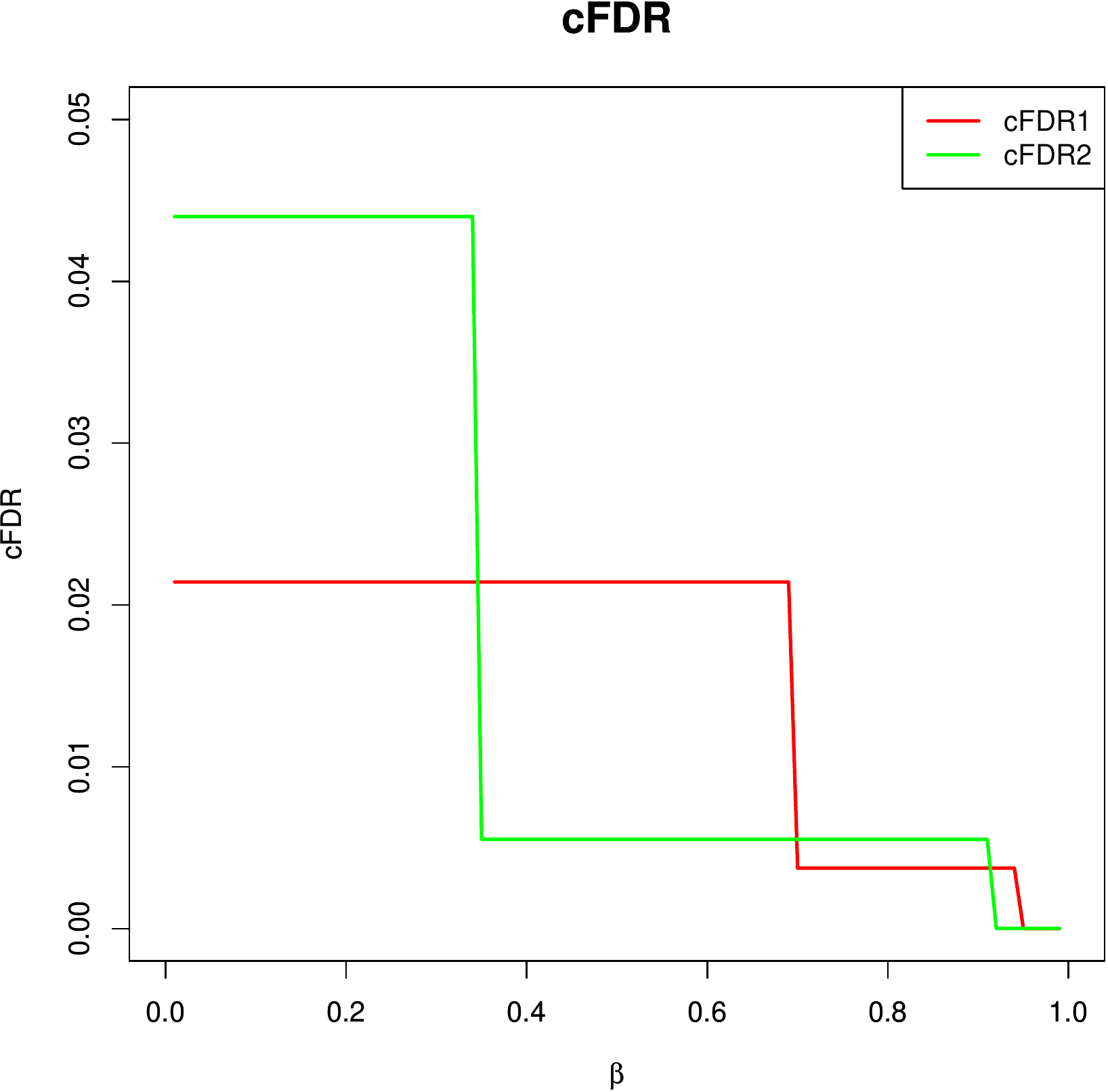}}
	\hspace{2mm}
	\subfigure [$cFNR_{nm}$ as a function of $\beta_{nm}$]{ \label{fig:cfnr4}
	\includegraphics[width=7.5cm,height=7.5cm]{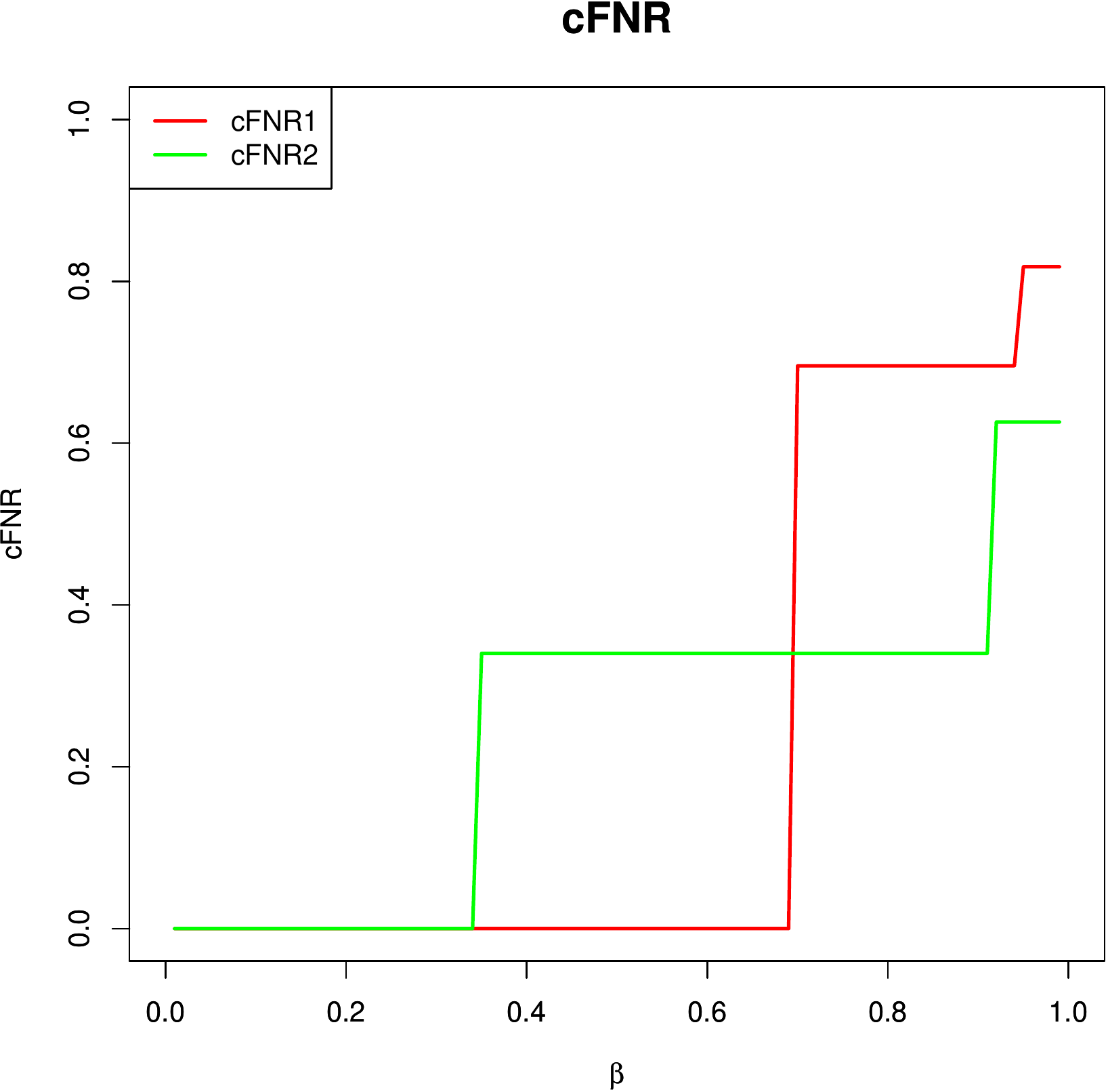}}
	\caption{$cFDR_{nm}$ and $cFNR_{nm}$ as functions of $\beta_{nm}$ in the misspecified situation of the model and variable selection problem.}
	\label{fig:error_rates4}
\end{figure}

\section{Summary and discussion}
\label{sec:conclusion}
Inverse regression problems have received little attention and Bayesian inverse regression problems occupy even lesser space in the 
statistical literature (see \ctn{Chatterjee17} for an overview). In particular, model selection procedures that account for the inverse perspective
has not even been touched upon so far, except the recent pseudo-Bayes factor undertaking by \ctn{Chat20a}. In this article we propose and develop a novel
Bayesian multiple testing formulation for the above purpose. Despite the relevance and elegance 
of the asymptotic theory, the real importance of our contribution
lies in realistic, small sample situations where the inverse perspective of the competing models are expected to be most pronounced. The fast and efficient
computational strategy that we employ for implementing our multiple testing procedure renders inverse model selection straightforward in the realistic
finite sample context. Interestingly, the forward pseudo-Bayes factor also features in our computational methodology, lending efficiency once it is available
for the competing models. Most importantly, our simulation experiments demonstrate that our Bayesian multiple testing procedure can improve upon the results
of both forward and inverse pseudo-Bayes factors. 

Although in this article we have exclusively considered the consistent prior for $\tilde x_i$ developed by \ctn{Chat20}, at least for applications
there is no bar to specifying any other sensible prior for $\tilde x_i$. Even though such priors need not lead to consistency of the inverse cross-validation
posteriors, acceptable finite-sample based Bayesian inference can be obtained as in any other situations, for any $n>1$ and $m\geq 1$.  

Although we shall consider applications of our multiple testing procedure to various real data problems, let us present here some of our previous results
on assessment of some palaeoclimate reconstruction models using the inverse reference distribution approach of \ctn{Bhattacharya13} 
in the light of our new multiple testing strategy. 

\ctn{Vasko00} reported a regular MCMC based inverse cross-validation exercise for a data set comprising
multivariate counts $y_i$ on $m = 52$ species of chironomid at $n = 62$ lakes (sites)
in Finland. The unidimensional $x_i$ denote mean July air temperature. As species respond
differently to summer temperature, the variation in the composition provides the analyst
with information on summer temperatures. This information is exploited to reconstruct
past climates from count data derived from fossils in the lake sediment; see \ctn{Korhola02}.
The Bayesian model is a Multinomial-Dirichlet model for the species counts with a Gaussian response function of the species parameters. 
However, \ctn{Bhattacharya13} showed that the posterior probabilities associated with the discrepancy measures $T_1$ and $T_2$ given by (\ref{eq:T1_old}) and
(\ref{eq:T2_old}) were almost zero.
\ctn{Bhatta06} proposed an improved Bayesian model for the same dataset, by replacing the unimodal Gaussian response function with a Dirichlet process (\ctn{Ferguson73}) based
mixture of Gaussian functions, which very flexibly allows unknown number of climate preferences and tolerance levels for each species. Although this model brought about
marked improvement over that of \ctn{Vasko00} in terms of including significantly more $x_i$ in the associated 95\% highest posterior density credible intervals of the
cross-validation posteriors, the posterior probabilities associated with $T_1$ and $T_2$ were still almost zero. 
A much improved palaeoclimate model was finally postulated by \ctn{Sabya13} by replacing the multinomial model with zero-inflated multinomial to account for excess
zero species counts typically present in the data. The other features of the model are similar to that of \ctn{Bhatta06}. Not only does this model 
far surpasses the previous models in terms of including the percentage of $x_i$ in the corresponding 95\% highest posterior density credible intervals
of the cross-validation posteriors (indeed, about 97\% $x_i$ are included in the respective intervals), inverse reference distributions for 
various discrepancy measures, including $T_1$ and $T_2$, comfortably 
contain the observed discrepancy measures in their respective 95\% highest posterior density credible intervals 
such that the relevant posterior probabilities associated with the discrepancy measures are significantly large. Recast in our multiple testing framework,
the results show that irrespective of the posterior probabilities of the aforementioned three Bayesian models, the multiple testing method would select
the model of \ctn{Sabya13} because of the overwhelming impact of its inverse regression part compared to the other two competing models.

In \ctn{Haslett06} pollen data was used, rather than chironomid data. The training data consisted of $7815$ observations of two climate variables and 14 species of pollen. 
The model proposed by \ctn{Haslett06} is again a Multinomial-Dirichlet distribution, but the two-dimensional response surface  
is based on lattice Gaussian Markov Random Field (GMRF) (see, for example, \ctn{Rue05}) which is responsible for creation of a very large number of parameters.
Indeed, their model consists of about $10,000$ parameters. The other limitations of this model are summarized in \ctn{Sabya13}. 
Applying the inverse reference distribution approach to this model and data \ctn{Bhattacharya04c} (Chapter 7) obtained almost zero posterior probability of the inverse
part. In fact, he demonstrated that this model overfits the pollen data; see also \ctn{Sabya13} who point out that such overfit is the consequence of the
very large number of parameters and the GMRF assumption. The general zero-inflated Multinomial-Dirichlet model along with the Dirichlet process based bivariate
Gaussian mixture model for the response functions proposed by \ctn{Sabya13} again turned out to be very successful in handling 
this pollen based palaeoclimate data. While including 
more than 94\% of the two observed climate variables in their respective 95\% highest posterior density credible intervals, the inverse reference distributions
well-captured the observed discrepancy measures, so that again the posterior probability of the inverse part turned out to be emphatically pronounced.
Thus, recast in our multiple testing paradigm, one can easily see that the zero-inflated Multinomial-Dirichlet model with the Dirichlet process based
response function would emerge the clear winner.

\newpage

\begin{appendix}

\section*{Appendix}

\section{Preliminaries for ensuring posterior consistency under general setup}
\label{sec:shalizi}

Following \ctn{Shalizi09} we consider a probability space $(\Omega,\mathcal F, P)$, 
and a sequence of random variables $y_1,y_2,\ldots$,   
taking values in some measurable space $(\Xi,\mathcal Y)$, whose
infinite-dimensional distribution is $P$. Let $\bY_n=\{y_1,\ldots,y_n\}$. The natural filtration of this process is
$\sigma(\bY_n)$, the smallest $\sigma$-field with respect to which $\bY_n$ is measurable. %where $\bY_n=(Y_1,Y_2,\ldots,Y_n)^T$.

We denote the distributions of processes adapted to $\sigma(\bY_n)$ 
by $F_{\theta}$, where $\theta$ is associated with a measurable
space $(\Theta,\mathcal T)$, and is generally infinite-dimensional. 
For the sake of convenience, we assume, as in \ctn{Shalizi09}, that $P$
and all the $F_{\theta}$ are dominated by a common reference measure, with respective
densities $f_{\theta_0}$ and $f_{\theta}$. The usual assumptions that $P\in\Theta$ or even $P$ lies in the support 
of the prior on $\Theta$, are not required for Shalizi's result, rendering it very general indeed.

\subsection{Assumptions and theorems of Shalizi}
\label{subsec:assumptions_shalizi}

\begin{itemize}
\item[(S1)] Consider the following likelihood ratio:
\begin{equation*}
R_n(\theta)=\frac{f_{\theta}(\bY_n)}{f_{\theta_0}(\bY_n)}.
%\label{eq:R_n}
\end{equation*}
Assume that $R_n(\theta)$ is $\sigma(\bY_n)\times \mathcal T$-measurable for all $n>0$.
\end{itemize}

\begin{itemize}
\item[(S2)] For every $\theta\in\Theta$, the KL-divergence rate
\begin{equation*}
h(\theta)=\underset{n\rightarrow\infty}{\lim}~\frac{1}{n}E\left(\log\frac{f_{\theta_0}(\bY_n)}{f_{\theta}(\bY_n)}\right).
%\label{eq:S3}
\end{equation*}
exists (possibly being infinite) and is $\mathcal T$-measurable.
\end{itemize}

\begin{itemize}
\item[(S3)] For each $\theta\in\Theta$, the generalized or relative asymptotic equipartition property holds, and so,
almost surely,
\begin{equation*}
\underset{n\rightarrow\infty}{\lim}~\frac{1}{n}\log R_n(\theta)=-h(\theta).
\end{equation*}
\end{itemize}

\begin{itemize}
\item[(S4)] 
Let $I=\left\{\theta:h(\theta)=\infty\right\}$. 
The prior $\pi$ satisfies $\pi(I)<1$.
\end{itemize}

%Following the notation of \ctn{Shalizi09}, for $A\subseteq\Theta$, let
%\begin{align}
%h\left(A\right)&=\underset{\theta\in A}{\mbox{ess~inf}}~h(\theta);\label{eq:h2}\\
%J(\theta)&=h(\theta)-h(\Theta);\label{eq:J}\\
%J(A)&=\underset{\theta\in A}{\mbox{ess~inf}}~J(\theta).\label{eq:J2}
%\end{align}
\begin{itemize}
\item[(S5)] There exists a sequence of sets $\mathcal G_n\rightarrow\Theta$ as $n\rightarrow\infty$ 
such that: %along with $\pi(\mathcal G_T)>0$
\begin{enumerate}
\item[(1)]
\begin{equation}
\pi\left(\mathcal G_n\right)\geq 1-\zeta\exp\left(-\gamma n\right),~\mbox{for some}~\zeta>0,~\gamma>2h(\Theta);
\label{eq:S5_1}
\end{equation}
\item[(2)]The convergence in (S3) is uniform in $\theta$ over $\mathcal G_n\setminus I$.
\item[(3)] $h\left(\mathcal G_n\right)\rightarrow h\left(\Theta\right)$, as $n\rightarrow\infty$.
\end{enumerate}
\end{itemize}
For each measurable $A\subseteq\Theta$, for every $\delta>0$, there exists a random natural number $\tau(A,\delta)$
such that
\begin{equation}
n^{-1}\log\int_{A}R_n(\theta)\pi(\theta)d\theta
\leq \delta+\underset{n\rightarrow\infty}{\lim\sup}~n^{-1}
\log\int_{A}R_n(\theta)\pi(\theta)d\theta,
\label{eq:limsup_2}
\end{equation}
for all $n>\tau(A,\delta)$, provided 
$\underset{n\rightarrow\infty}{\lim\sup}~n^{-1}\log\pi\left(\mathbb I_A R_n\right)<\infty$.
%$\mathbb I_A$ denotes the indicator function of the set $A$.
Regarding this, the following assumption has been made by Shalizi:
\begin{itemize}
\item[(S6)] The sets $\mathcal G_n$ of (S5) can be chosen such that for every $\delta>0$, the inequality
$n>\tau(\mathcal G_n,\delta)$ holds almost surely for all sufficiently large $n$.
\end{itemize}
\begin{itemize}
\item[(S7)] The sets $\mathcal G_n$ of (S5) and (S6) can be chosen such that for any set $A$ with $\pi(A)>0$, 
\begin{equation}
h\left(\mathcal G_n\cap A\right)\rightarrow h\left(A\right),
\label{eq:S7}
\end{equation}
as $n\rightarrow\infty$.
\end{itemize}
%Under the above assumptions, \ctn{Shalizi09} proved the following results.

%\begin{theorem}[\ctn{Shalizi09}]
%\label{theorem:shalizi1}
%Consider assumptions (S1)--(S7) and any set $A\in\mathcal T$ with $\pi(A)>0$ and $h(A)>h(\Theta)$. Then,
%\begin{equation*}
%\underset{n\rightarrow\infty}{\lim}~\pi(A|\bY_n)=0~\mbox{almost surely},
%%\label{eq:supp_post_conv1}
%\end{equation*}
%where $\pi(\cdot|\bY_n)$ denotes the posterior distribution of $\theta$ given $\bY_n$.
%\end{theorem}

\end{appendix}

\bibliography{irmcmc}

\end{document}